\documentclass[11pt, oneside]{article}
\usepackage{geometry}                	
\usepackage{amssymb}
\usepackage{mleftright}
\usepackage{amsmath,amsthm,mathtools,color}
\usepackage{usebib}
\usepackage{esint}
\usepackage{graphicx, mathptmx}
\usepackage{mathtools, stackengine, mleftright}
\usepackage[mathcal]{euscript} 
\usepackage{newtxtext, newtxmath, dsfont, mathrsfs, accents, verbatim} 
\usepackage{enumerate}
\usepackage{epsfig}
\usepackage{delarray}
\usepackage{makeidx}
\usepackage[titletoc,title]{appendix}
\usepackage{tasks}
\numberwithin{equation}{section}
\newtheorem{theorem}{Theorem}[section]
\newtheorem{corollary}[theorem]{Corollary}
\newtheorem{lemma}[theorem]{Lemma}
\newtheorem{proposition}[theorem]{Proposition}
\newtheorem{Remark}[theorem]{Remark}
\newtheorem{definition}[theorem]{Definition}

\newcommand{\divx}{\mathop{\mathrm{div}}}
\newcommand{\esssup}{\mathop{\mathrm{ess~sup}}}
\newcommand{\essinf}{\mathop{\mathrm{ess~inf}}}
\newcommand{\argmin}{\mathop{\mathrm{arg~min}}}

\allowdisplaybreaks[4]

\title{Continuous differentiability of a weak solution to very singular elliptic equations involving anisotropic diffusivity}
\author{Shuntaro Tsubouchi\footnote{Graduate School of Mathematical Sciences, The University of Tokyo, Japan. The author was partly supported by Grant-in-Aid for JSPS Fellows (No. 22J12394). \textit{Email}: \texttt{tsubos@ms.u-tokyo.ac.jp}}}\date{}							
\begin{document}
\maketitle
\begin{abstract}
In this paper we consider a very singular elliptic equation that involves an anisotropic diffusion operator, including one-Laplacian, and is perturbed by a $p$-Laplacian-type diffusion operator with $1<p<\infty$. This equation seems analytically difficult to handle near a facet, the place where the gradient vanishes. Our main purpose is to prove that weak solutions are continuously differentiable even across the facet. Here it is of interest to know whether a gradient is continuous when it is truncated near a facet. To answer this affirmatively, we consider an approximation problem, and use standard methods including De Giorgi's truncation and freezing coefficient methods.
\end{abstract}
\bigbreak
\bigbreak
\section{Introduction}\label{Sect: Introduction}
This paper is concerned with continuous differentiability of a weak solution to a very singular elliptic equation of the form
\begin{equation}\label{Eq: 1,p-Laplace}
L_{b,\,p}u\coloneqq -b\Delta_{1}u-\Delta_{p}u=f(x)\quad \textrm{for }x=(x_{1},\,\dots\,,\,x_{n})\in\Omega\subset{\mathbb R}^{n}.
\end{equation}
Here \(b\in(0,\,\infty)\) and \(p\in(1,\,\infty)\) are fixed constants, the dimension \(n\) satisfies \(n\ge 2\), \(\Omega\) is a \(n\)-dimensional bounded domain with Lipschitz boundary, and \(f\) is a given real-valued function in \(\Omega\). The partial differential operators \(\Delta_{1}\) and \(\Delta_{p}\) often called, respectively, one-Laplacian and $p$-Laplacian, are in divergence forms given by
\[\Delta_{1}u\coloneqq \mathrm{div}\,\mleft(\frac{\nabla u}{\lvert \nabla u\rvert}\mright),\quad\Delta_{p}u\coloneqq \mathrm{div}\,\mleft(\lvert \nabla u\rvert^{p-2}\nabla u\mright).\]
Here $\nabla u=(\partial_{x_{1}}u,\,\dots\,,\,\partial_{x_{n}}u)$ with $\partial_{x_{j}}u=\partial u/\partial x_{j}$ for a function $u=u(x_{1},\,\dots\,,\,x_{n})$, and $\divx X=\sum_{j=1}^{n}\partial_{x_{j}}X_{j}$ for a vector field $X=X(x_{1},\,\dots\,,\,x_{n})$.

For the $p$-Poisson equations, that is, in the case $b=0$, H\"{o}lder continuity of a gradient was already established by many experts (\cite{MR709038}, \cite{MR672713}, \cite{MR721568}, \cite{MR969420}, \cite{MR727034}, \cite{MR474389}, \cite{MR0244628}).
These regularity results should not hold for the one-Poisson equations, even if $f$ is sufficiently smooth. In fact, in the one-dimensional case, a non-smooth function $u=u(x_{1})$ may solve $-\Delta_{1}u=0$, as long as it is absolutely continuous and non-decreasing. Hence, it seems impossible to show continuous differentiability of weak solutions, even for the one-Laplace equation. This problem is due to the fact that the ellipticity of one-Laplacian degenerates in the direction $\nabla u$, which is substantially different from $p$-Laplacian.
On the other hand, in the multi-dimensional case, the ellipticity of one-Laplacian is non-degenerate in directions that are orthogonal to $\nabla u$. It should be mentioned that this ellipticity becomes singular in a facet, the place where a gradient vanishes.
By anisotropic diffusivity of one-Laplace operator $\Delta_{1}$, which consists of both degenerate and singular ellipticity, $\Delta_{1}$ seems, unlike $\Delta_{p}$,  impossible to deal with in existing elliptic regularity theory.

For this reason, it will be interesting to consider a question whether a solution to (\ref{Eq: 1,p-Laplace}), where $L_{b,\,p}$ contains both $\Delta_{1}$ and $\Delta_{p}$, is continuously differentiable ($C^{1}$).
In the special case where the solution $u$ is convex, this $C^{1}$-regularity problem is already answered affirmatively \cite[Theorem 1]{giga2021continuity}. The aim of this paper is to establish continuity of derivatives without the convexity assumption.

The equation (\ref{Eq: 1,p-Laplace}) is concerned with a minimizing problem of the energy functional
\[\mathcal{F}(u)\coloneqq \int_{\Omega}\mleft(E(\nabla u)-fu\mright)\,{\mathrm d}x\quad \textrm{with}\quad E(z)=b\lvert z\rvert+\frac{\lvert z\rvert^{p}}{p}.\]
Here, as the principal space, we often choose a closed affine space $u_{0}+W_{0}^{1,\,p}(\Omega)$ for some fixed $u_{0}\in W^{1,\,p}(\Omega)$.
Formally considering the Euler--Lagrange equation corresponding to this energy minimizing problem, we can obtain (\ref{Eq: 1,p-Laplace}).
However, since the mapping ${\mathbb R}^{n}\ni z\mapsto \lvert z\rvert\in\lbrack 0,\,\infty)$ is no longer differentiable at the origin, the term $\nabla u/\lvert \nabla u\rvert$ should be treated not in the classical sense on a place $\{\nabla u=0\}$, often called a facet of $u$.
This problem can be overcome by introducing a subdifferenital, which is often useful for analysis of non-smooth convex functions.  
Also, since a solution $u$ is assumed to be in the first order Sobolev space, we should treat a weak solution in a distributional sense.
Our definition of a weak solution is described in Section \ref{Subsect: Generalized result and novelity}.

The energy density $E$ can be found in fields of materials science and fluid mechanics. We would like to mention some mathematical models for each field briefly.

A typical example of the former is the relaxation dynamics of a crystal surface below the roughening temperature, which was modeled by Spohn in \cite{spohn1993surface}. There, evolution of $u=u(x,\,t)$, a height function of the crystal in a two-dimensional domain $\Omega$, is modeled to satisfy
\[\partial_{t}u=-\Delta \mu\quad \textrm{and}\quad\mu=-\frac{\delta \Phi}{\delta u}\]
from a thermodynamic viewpoint (see \cite{MR3289366} and the references therein).
Here $\mu$ is often called the chemical potential, and $\Phi$ denotes the crystal surface energy, given by
\[\Phi(u)\coloneqq \int_{\Omega}\lvert\nabla u\rvert\,{\mathrm d}x+\kappa \int_{\Omega}\lvert \nabla u\rvert^{3}\,{\mathrm d}x\quad \textrm{with a constant }\kappa>0.\]
Finally, evolution of a crystal surface results in a fourth order parabolic equation $b\partial_{t}u=\Delta L_{b,\,3}u$ with $b=1/(3\kappa)>0$.
When $u$ is stationary, this model becomes a second order elliptic equation $L_{b,\,3}u=b\mu$. Here $\mu$ is sufficiently smooth since it is harmonic.

For the latter model, the density $E$ often appears with $(n,\,p)=(2,\,2)$ when modeling an incompressible laminar flow of a Bingham fluid in a cylinder $\Omega\times {\mathbb R}\subset{\mathbb R}^{3}$. To be precise, we assume that the velocity field $U$ is of the form $U=(0,\,0,\,u(x_{1},\,x_{2}))$, and hence the incompressible condition $\divx U=0$ clearly holds. It is also assumed that its speed is sufficiently slow enough to discard convection effects. Under these settings, the Euler--Lagrange equation for a Bingham fluid is given by
\[\frac{\delta \Psi}{\delta u}=-\partial_{x_{3}}\pi\quad \textrm{with}\quad\Psi(u)\coloneqq \gamma\int_{\Omega}\lvert\nabla u \rvert\,{\mathrm d}x+\frac{\mu}{2}\int_{\Omega}\lvert\nabla u \rvert^{2}\,{\mathrm d}x.\]
Here $\gamma$ and $\mu$ are positive constants, and $\pi=\pi(x_{3})$ denotes the pressure function, whose slope becomes constant under the laminar setting.
Mathematical formulations on motion of Bingham fluids are found in \cite[Chapter VI]{MR0521262}, \cite{MR635927}. Bingham fluids have two different aspects of plasticity and viscosity, which are respectively reflected by the diffusion operators $\Delta_{1}u$ and $\Delta_{2}u=\Delta u$ in the equation (\ref{Eq: 1,p-Laplace}). It is worth mentioning that when $\gamma=0$, this problem models motion of incompressible Newtonian fluids.

The density $E$ also arises in an optimal transport problem for a congested traffic dynamics (\cite{MR2651987}, \cite{MR2407018}).
There the optimal traffic flow is connected with a certain very degenerate elliptic problem, and higher regularity results related to this have been well-established.
Our $C^{1}$-regularity result is inspired by a recent work by B\"{o}gelein--Duzaar--Giova--Passarelli di Napoli \cite{BDGPdN}, which is also concerned with continuity of gradients on the very degenerate problem related to the optimal traffic flow. For further explanations and comparisons, see Section \ref{Subsect: Very degenerate}.

More generally, we will also consider
\begin{equation}\label{Eq: main eq elliptic}
\mathcal{L}u\coloneqq \mathcal{L}_{1}u+\mathcal{L}_{p}u=f\quad \textrm{in }\Omega\subset{\mathbb R}^{n},
\end{equation}
where the partial differential operators $\mathcal{L}_{1}$ and $\mathcal{L}_{p}$ generalize $-\Delta_{1}$ and $-\Delta_{p}$ respectively. The detailed conditions of $\mathcal{L}_{1},\,\mathcal{L}_{p}$ will be described later in Section \ref{Subsect: Generalized result and novelity}.

\subsection{Our main results and strategy}\label{Subsect: Our strategy}
For simplicity, we first consider the model case (\ref{Eq: 1,p-Laplace}). Our result is 
\begin{theorem}\label{Thm: C1-regularity model case}
Let \(u\) be a weak solution to (\ref{Eq: 1,p-Laplace}) with \(f\in L^{q}(\Omega)\) for some \(q\in(n,\, \infty\rbrack\). Then, \(u\) is in \(C^{1}(\Omega)\).
\end{theorem}
By Giga and the author \cite{giga2021continuity}, this type of $C^{1}$-regularity result was already established in the special case where \(u\) is convex.
The proof therein is substantially based on convex analysis and a strong maximum principle.
In particular, the gradient \(\nabla u\) is often considered not only as a distribution, but also as a subgradient, which made it successful to show $C^{1}$-regularity rather elementarily.

After that work was completed, the author has found it possible to prove continuity of derivatives of non-convex weak solutions. Instead of the strong maximum principle, we use standard methods from the Schauder theory and the De Giorgi--Nash--Moser theory to justify continuity of a gradient $\nabla u$ when it is truncated near a facet.
Precisely speaking, we would like to show that for each small \(\delta\in(0,\,1)\), the truncated gradient
\begin{equation}\label{Eq: Truncated vector field}
{\mathcal G}_{\delta}(\nabla u)\coloneqq (\lvert \nabla u\rvert-\delta)_{+}\frac{\nabla u}{\lvert \nabla u\rvert}
\end{equation}
is locally H\"{o}lder continuous. This result can be explained by computing eigenvalues of the Hessian matrix $\nabla^{2} E(\nabla u)=(\partial_{x_{i}x_{j}} E(\nabla u))_{i,\,j}$.
In fact, $\nabla^{2} E$ satisfies
\begin{align*}
(\textrm{ellipticity ratio of $\nabla^{2}E(z_{0})$})&=\frac{(\textrm{the largest eigenvalue of $\nabla^{2}E(z_{0})$})}{(\textrm{the lowest eigenvalue of $\nabla^{2}E(z_{0})$})}\\&\le C(p)\mleft(1+b\lvert z_{0}\rvert^{1-p}\mright)
\end{align*}
for all $z_{0}\in{\mathbb R}^{n}\setminus\{0\}$. In particular, for each fixed $\delta>0$, the operator $L_{b,\,p}$ can be said to be \textit{uniformly elliptic over a place} \(\{\lvert \nabla u\rvert>\delta\}\). It should be emphasized that ${\mathcal G}_{\delta}(\nabla u)$ is supported in this place, and therefore this vector field will be H\"{o}lder continuous.
Since it is easy to check that the continuous vector field \({\mathcal G}_{\delta}(\nabla u)\) uniformly converges to \(\nabla u\), this implies that \(\nabla u\) is also continuous.
We should note that our analysis for ${\mathcal G}_{\delta}(\nabla u)$ substantially depends on the truncation parameter $\delta>0$. Thus, our H\"{o}lder continuity estimates may blow up as $\delta$ tends to $0$. In particular, on quantitative growth estimates of a gradient near a facet, less is known. This problem is essentially because that the equation (\ref{Eq: 1,p-Laplace}) becomes non-uniformly elliptic near a facet. It should be emphasized that our problem substantially differs from what is often called the $(p,\,q)$-growth problem, where non-uniform ellipticity appears as a gradient blows up \cite{MR2291779}, \cite{MR4258810}.

The proof of H\"{o}lder continuity of \({\mathcal G}_{2\delta}(\nabla u)\), where we have replaced $\delta$ by $2\delta$, broadly consists of two parts; relaxing the singular operator $L_{b,\,p}$, and establishing a priori H\"{o}lder estimates on relaxed vector fields. For relaxations, we will have to consider relaxed problems that are uniformly elliptic.
In the special case (\ref{Eq: 1,p-Laplace}), for an approximating parameter \(\varepsilon\in(0,\,1)\), the regularized equation is given by
\begin{equation}\label{Eq: Approximating problem special case}
L_{b,\,p}^{\varepsilon}u_{\varepsilon}=f_{\varepsilon}.
\end{equation}
Here \(f_{\varepsilon}\) converges to \(f\) in a suitable sense, and the relaxed operator \(L_{b,\,p}^{\varepsilon}\) is given by
\[L_{b,\,p}^{\varepsilon}u_{\varepsilon}\coloneqq -\mathrm{div}\,\mleft(\frac{b\nabla u_{\varepsilon}}{\sqrt{\varepsilon^{2}+\lvert \nabla u_{\varepsilon}\rvert^{2}}}+\mleft(\varepsilon^{2}+\lvert \nabla u_{\varepsilon}\rvert^{2}\mright)^{p/2-1}\nabla u_{\varepsilon}\mright)\]
for each \(\varepsilon\in(0,\,1)\).
This operator naturally appears when one approximates the density $E$ by \[E_{\varepsilon}(z)\coloneqq b\sqrt{\varepsilon^{2}+\lvert z\rvert^{2}}+\frac{1}{p}\mleft(\varepsilon^{2}+\lvert z\rvert^{2}\mright)^{p/2}.\]
In particular, we regularize $\lvert z\rvert^{s}$ by $\left(\varepsilon^{2}+\lvert z\rvert^{2}\right)^{s/2}$ for each $s\in\{\,1,\,p\,\}$, which allows for linearization.
It should be noted that for each fixed $\varepsilon\in(0,\,1)$, the equation (\ref{Eq: Approximating problem special case}) is uniformly elliptic, in the sense that there holds
\[\textrm{(the ellipticity ratio of $\nabla^{2}E_{\varepsilon}(z_{0})$)}\le C(p)\left[1+b\left(\varepsilon^{2}+\lvert z_{0}\rvert^{2} \right)^{\frac{1-p}{2}}\right]\le C(b,\,p,\,\varepsilon)\]
for all $z_{0}\in{\mathbb R}^{n}$. Moreover, sufficient smoothness for coefficients of the operator $L_{b,\,p}^{\varepsilon}$ is also guaranteed. Therefore, from existing results on elliptic regularity theory, it is not restrictive to assume that a weak solution to (\ref{Eq: Approximating problem special case}) admits higher regularity $u_{\varepsilon}\in W_{{\mathrm{loc}}}^{1,\,\infty}\cap W_{\mathrm{loc}}^{2,\,2}$. Moreover, when the external force term $f_{\varepsilon}$ is in $C^{\infty}$, it is possible to let $u_{\varepsilon}$ be even in $C^{\infty}$ with the aid of bootstrap arguments \cite[Chapter V]{MR0244627}.
For each fixed $\delta\in(0,\,1)$, we would like to establish an a priori H\"{o}lder continuity of a truncated gradient
\begin{equation}\label{Eq: Approximated truncated vector field}
{\mathcal G}_{2\delta,\,\varepsilon}(\nabla u_{\varepsilon})\coloneqq \mleft(\sqrt{\varepsilon^{2}+\lvert\nabla u_{\varepsilon}\rvert^{2}}-2\delta\mright)_{+}\frac{\nabla u_{\varepsilon}}{\lvert\nabla u_{\varepsilon}\rvert},
\end{equation}
whose estimate is independent of the approximation parameter $\varepsilon$ that is sufficiently smaller than $\delta$. Since the singular operator $L_{b,\,p}$ itself has to be relaxed by a non-degenerate one $L_{b,\,p}^{\varepsilon}$, the truncating mapping ${\mathcal G}_{2\delta}$ should also be replaced by another regularized one ${\mathcal G}_{2\delta,\,\varepsilon}$.
For this reason, we have to consider another vector field given by (\ref{Eq: Approximated truncated vector field}), which clearly differs from (\ref{Eq: Truncated vector field}).
The a priori H\"{o}lder estimate, which plays as a key lemma in this paper, enables us to apply the Arzel\`{a}--Ascoli theorem and to conclude that ${\mathcal G}_{2\delta}(\nabla u)$ is also H\"{o}lder continuous. We will obtain the desired H\"{o}lder continuity estimates by standard methods, including De Giorgi's truncation and a freezing coefficient argument. Our proofs of a priori H\"{o}lder continuity estimates are highly inspired by \cite[\S 4--7]{BDGPdN}.

For a general problem (\ref{Eq: main eq elliptic}), we need to establish suitable approximation of the singular operator ${\mathcal L}$.
There, for an approximation parameter $\varepsilon\in(0,\,1)$, we have to introduce a uniformly elliptic operator ${\mathcal L}^{\varepsilon}$ that is neither degenerate nor singular.
A main problem herein is that we have to justify that the regularized solution $u_{\varepsilon}$ converges to the original solution $u$ in the distributional sense. 
In the special case ${\mathcal L}=L_{b,\,p}$, ${\mathcal L}^{\varepsilon}=L_{b,\,p}^{\varepsilon}$, and $f_{\varepsilon}\equiv f$, this justification is already discussed in \cite[Lemma 4]{MR4201656}, which is based on arguments from \cite[Theorem 6.1]{MR1451535}.
In this paper, we would like to give a more general approximating operator ${\mathcal L}^{\varepsilon}$ that is based on the convolution by the Friedrichs mollifier. The novelty of our new approximation argument is that this works for a very singular operator ${\mathcal L}_{1}$ that differs from $b\Delta_{1}$. In particular, we are able to generalize the total variation energy functional in ${\mathcal F}$, as long as its corresponding density $E_{1}$ is convex and positively one-homogeneous.
Our setting concerning $E_{1}$ will be substantially different from previous works \cite{zbMATH04075921}, \cite{zbMATH07373229}, \cite{zbMATH06137867}, \cite{zbMATH07045537}, which are concerned with energy minimizers of purely linear growth, since ellipticity conditions therein appear to differ from ours.

It is worth mentioning that our strategy works even in the system case, and $C^{1}$-regularity on vector-valued problems has been established in another recent paper \cite{T-system} by the author. It should be emphasized that when one considers everywhere regularity for the system problem, it would be natural to let the divergence operator ${\mathcal L}$ have a symmetric structure, often called the Uhlenbeck structure (see \cite{zbMATH04124460}, \cite{zbMATH06446371}, \cite{MR474389}). 
In particular, when it comes to everywhere $C^{1}$-regularity for vector-valued solutions, assumptions related to the Uhlenbeck structure may force us to restrict \({\mathcal L}_{1}=b\Delta_{1}\) (see Section \ref{Rmk: Uhlenbeck structure case} for details).
In absence of this assumption, it is well-known that energy minimizers has partial regularity, and often lacks everywhere regularity. 
Both classical partial regularity results and counterexamples to full regularity are found in \cite[Chapter 4]{MR3887613}, \cite[Chapter 9]{MR1962933}.
On partial regularity of local minimizers in $\mathop{\mathrm{BV}}$ spaces, both $p$-growth problems and purely linear growth problems are well-discussed in \cite{zbMATH04075921}, \cite{zbMATH07373229}, \cite{zbMATH07045537}, \cite{zbMATH06405920}.
Although it will be interesting to consider partial regularity on non-Uhlenbeck-type problems instead of full regularity, we will not discuss this problem in the paper.

\subsection{Comparisons to previous results on a very degenerate equation}\label{Subsect: Very degenerate}
Our proofs on continuity of derivatives are inspired by \cite{BDGPdN}, which is concerned with everywhere regularity for gradients of solutions to very degenerate equations, or even systems, of the form
\begin{equation}\label{Eq: Very degenerate eq}
-\divx\mleft((\lvert\nabla v\rvert-b)_{+}^{p^{\prime}-1}\frac{\nabla v}{\lvert\nabla v\rvert}\mright)=f\in L^{q}(\Omega)\quad \textrm{in }\Omega\subset {\mathbb R}^{n}.
\end{equation}
Here $p^{\prime}$ denotes the H\"{o}lder conjugate of $p\in(1,\,\infty)$, i.e., $p^{\prime}=p/(p-1)$. This equation is motivated by a mathematical model on congested traffic dynamics \cite{MR2407018}. There this model results in a minimizing problem of the form
\begin{equation}\label{Eq: minimizing problem in optimal transport}
\sigma_{\mathrm{opt}}\in \argmin\mleft\{\int_{\Omega}E(\sigma)\,{\mathrm d}x\mathrel{}\middle|\mathrel{} \left.\begin{array}{c}
\sigma\in L^{p}(\Omega;\,{\mathbb R}^{n}),\\ -\divx \sigma=f\textrm{ in }\Omega,\,\sigma\cdot\nu=0\textrm{ on }\partial\Omega
\end{array} \right. \mright\}.
\end{equation}
The equation (\ref{Eq: Very degenerate eq}) is connected with the minimizing problem (\ref{Eq: minimizing problem in optimal transport}).
In fact, in a paper \cite{MR2651987}, it is proved that the optimal traffic flow $\sigma_{\mathrm{opt}}$ of the problem (\ref{Eq: minimizing problem in optimal transport}) is uniquely given by $\nabla E^{\ast}(\nabla v)$, where $v$ solves (\ref{Eq: Very degenerate eq}) under a Neumann boundary condition.
Also, it is worth mentioning that (\ref{Eq: Very degenerate eq}) can be written by \(-\divx(\nabla E^{\ast}(\nabla v))=f\), where \[E^{\ast}(z)=\frac{1}{p^{\prime}}\mleft(\lvert z\rvert-b\mright)_{+}^{p^{\prime}}\] denotes the Legendre transform of $E$. 

The question whether the flow $\sigma_{\mathrm{opt}}=\nabla E^{\ast}(\nabla v)$ is continuous has been resolved affirmatively. There it has been to establish continuity of the vector filed ${\mathcal G}_{b+\delta}(\nabla v)$ with $\delta>0$.
This expectation on continuity can be explained, similarly to (\ref{Eq: 1,p-Laplace}), by calculating the ellipticity ratio of the Hessian matrix $\nabla^{2} E^{\ast}(\nabla v)$. In fact, for every $z_{0}\in {\mathbb R}^{n}$ with $\lvert z_{0}\rvert\ge \delta+b$, there holds 
\begin{align*}
(\textrm{ellipticity ratio of $\nabla^{2}E^{\ast}(z_{0})$})&=\frac{(\textrm{the largest eigenvalue of $\nabla^{2}E^{\ast}(z_{0})$})}{(\textrm{the lowest eigenvalue of $\nabla^{2}E^{\ast}(z_{0})$})}\\&\le (p-1)\left(1+\delta^{-1}\right)
\end{align*}
as $\delta\to 0$. This estimate suggests that for each fixed $\delta>0$, the truncated gradient ${\mathcal G}_{b+\delta}(\nabla v)$ should be H\"{o}lder continuous. We note that it is possible to show ${\mathcal G}_{b+\delta}(\nabla v)$ uniformly converges to ${\mathcal G}_{b}(\nabla v)$ as $\delta\to 0$, and thus ${\mathcal G}_{b}(\nabla v)$ will be also continuous.

When $v$ is a scalar function, this expectation on continuity of ${\mathcal G}_{b+\delta}(\nabla v)$ with $\delta>0$ was first mathematically justified by Santambrogio--Vespri \cite{MR2728558} in 2010 for the special case $n=2$ with $b=1$. The proof therein is based on oscillation estimates related to Dirichlet energy. For technical reasons related to this method, the condition $n=2$ is essentially required.
Later in 2014, Colombo--Figalli \cite{MR3133426} established a more general proof that works for any $n\ge 2$ and any function $E^{\ast}$ whose zero-levelset is sufficiently large enough to define a real-valued Minkowski gauge. This Minkowski gauge plays as a basic modulus to estimate ellipticity ratios on equations they treated. It should be mentioned that their strategy will not work for our problem (\ref{Eq: 1,p-Laplace}), since some structures of the density functions therein seems substantially different from ours. In fact, in our problem, the zero-levelset of $E$ is only a singleton, and thus it is impossible to define the Minkowski gauge as a real-valued function.
The recent work by B\"{o}gelein--Duzaar--Giova--Passarelli di Napoli \cite{BDGPdN} also focuses on proving H\"{o}lder continuity of ${\mathcal G}_{b+\delta}(\nabla v)$ for each $\delta>0$ with $b=1$. A novelty of this paper is that their strategy works even when $v$ is a vector-valued function. There they considered an approximating problem of the form
\begin{equation}\label{Eq: Relaxed very degenerate equation}
-\varepsilon\Delta v_{\varepsilon}-\divx\,(\nabla E^{\ast}(\nabla v_{\varepsilon}))=f\quad \textrm{in }\Omega.
\end{equation}
The paper \cite{BDGPdN} establishes a priori H\"{o}lder continuity of ${\mathcal G}_{1+2\delta}(\nabla v_{\varepsilon})$ for each fixed $\delta\in(0,\,1)$, whose estimate is independent of an approximation parameter $\varepsilon\in(0,\,1)$.

Our proofs on a priori H\"{o}lder estimates are highly inspired by \cite[\S 4--7]{BDGPdN}. There are three significant differences between their proofs and ours. The first is that we need to treat another relaxed vector field ${\mathcal G}_{2\delta,\,\varepsilon}(\nabla u_{\varepsilon})$ as in (\ref{Eq: Approximated truncated vector field}), different from ${\mathcal G}_{2\delta}(\nabla v_{\varepsilon})$. This is essentially because our regularized problems as in (\ref{Eq: Approximating problem special case}) are based on relaxing the very singular operator $L_{b,\,p}$ (or the general operator ${\mathcal L}$) itself, whereas their approximation given in (\ref{Eq: Relaxed very degenerate equation}) does not change principal part at all. Hence, in the relaxed problem (\ref{Eq: Relaxed very degenerate equation}), it is not necessary to adapt another mapping than ${\mathcal G}_{2\delta,\,\varepsilon}$. The second is some structural differences between density functions $E$ and $E^{\ast}$, which make our energy estimates in Section \ref{Subsect: Energy estimates} different from \cite[\S 4--5]{BDGPdN}. The third is that we have to make the approximation parameter $\varepsilon$ sufficiently smaller than the truncation parameter $\delta$. To be precise, when we prove a priori H\"{o}lder continuity of ${\mathcal G}_{2\delta,\,\varepsilon}(\nabla u_{\varepsilon})$, we carefully utilize the setting $0<\varepsilon<\delta/8$. This condition will be required especially when two different moduli $\lvert\nabla u_{\varepsilon}\rvert$ and $\sqrt{\varepsilon^{2}+\lvert\nabla u_{\varepsilon}\rvert^{2}}$ are both used. In the paper \cite{BDGPdN}, no restrictions on $\varepsilon\in(0,\,1)$ are made, since their analysis for (\ref{Eq: Relaxed very degenerate equation}) is based on a single modulus $\lvert \nabla v_{\varepsilon}\rvert$ only.




\subsection{Some notations and our general results}\label{Subsect: Generalized result and novelity}
In Section \ref{Subsect: Generalized result and novelity}, we describe our settings on the equation (\ref{Eq: main eq elliptic}) and describe the main result. Before stating our general result, we fix some notations.

We denote ${\mathbb Z}_{\ge 0}\coloneqq \{\,0,\,1,\,2,\,\dots\,\,\}$ by the set of all non-negative integers, and ${\mathbb N}\coloneqq {\mathbb Z}_{\ge 0}\setminus\{ 0\}$ by the set of all natural numbers.
For real numbers $a,\,b\in{\mathbb R}$, we often write $a\vee b\coloneqq \max\{\,a,\,b\,\},\,a\vee b\coloneqq \min\{\,a,\,b\,\}$ for notational simplicity.

For \(n\times 1\) column vectors \(x=(x_{i})_{1\le i\le n},\,y=(y_{i})_{1\le i\le n}\in{\mathbb R}^{n}\), we write
\[\langle x\mid y\rangle\coloneqq \sum_{i=1}^{n}x_{i}y_{i},\,\quad \lvert x\rvert\coloneqq \sqrt{\langle x\mid x\rangle}=\sqrt{\sum_{i=1}^{n}\,\lvert x_{i}\rvert^{2}}.\]
In other words, we equip \({\mathbb R}^{n}\) with the Euclidean inner product $\langle\,\cdot\mid\cdot\,\rangle$ and the Euclidean norm \(\lvert\,\cdot\,\rvert\).
We also define a tensor \(x\otimes y\), which is regarded as a real-valued $n\times n$ matrix
\[x\otimes y\coloneqq (x_{i}y_{j})_{i,\,j}=\begin{pmatrix}x_{1}y_{1}&\cdots & x_{1}y_{n}\\ \vdots & \ddots &\vdots\\ x_{n}y_{1}&\cdots &x_{n}y_{n}\end{pmatrix}.\]

We denote $\mathop{\mathrm{id}}$ by the $n\times n$ identity matrix. For an \(n\times n\) real matrix \(A\), we write \(\lVert A\rVert\) as the operator norm of \(A\) with respect to the canonical Euclid norm,
\[\textrm{i.e., }\lVert A\rVert\coloneqq \sup\mleft\{\frac{\lvert Ax\rvert}{\lvert x\rvert}\mathrel{}\middle|\mathrel{}x\in{\mathbb R}^{n}\setminus\{0\}\mright\}.\]
We also define the Frobenius norm
\[\lvert A\rvert\coloneqq \sqrt{\sum_{i,\,j=1}^{n}\lvert a_{i,\,j}\rvert^{2}}\quad \textrm{for }A=(a_{i,\,j})_{1\le i,\,j\le n}.\]
For $n\times n$ real symmetric matrices $A,\,B$, we write $A\leqslant B$ or $B\geqslant A$ when $B-A$ is non-negative. 

For a scalar function of \(u=u(x_{1},\,\dots,\,x_{n})\), we denote \(\nabla u,\,\nabla^{2}u,\,\nabla^{3}u\) respectively by gradient vector, Hessian matrix, and third-order derivatives of \(u\);
\[\textrm{i.e., }\nabla u\coloneqq (\partial_{x_{i}}u)_{1\le i\le n},\,\nabla^{2}u\coloneqq (\partial_{x_{i}x_{j}}u)_{1\le i,\,j\le n},\,\nabla^{3}u\coloneqq (\partial_{x_{i}x_{j}x_{k}}u)_{1\le i,\,j,\,k\le n}.\]
Similarly to $\lvert\nabla u\rvert$ and $\mleft\lvert \nabla^{2}u\mright\rvert$, we define \[\mleft\lvert \nabla^{3}u\mright\rvert\coloneqq \sqrt{\sum_{i,\,j,\,k=1} \lvert \partial_{x_{i}x_{j}x_{k}}u\rvert^{2}}.\]
For given ${\mathbb R}^{n}$-valued vector field $U=(U_{1},\,\dots\,,\,U_{n})$ with $U_{j}=U_{j}(x_{1},\,\dots\,,\,x_{n})$ for each $j\in\{\,1,\,\dots\,,\,n\,\}$, we denote $DU=(\partial_{x_{i}}U_{j})_{i,\,j}$ by the Jacobian matrix of \(U\).

Let $U\subset{\mathbb R}^{n}$ be an open set. For given numbers $s\in\lbrack 1,\,\infty\rbrack$, $d\in{\mathbb N}$, and $k\in{\mathbb N}$, we denote $L^{s}(U;\,{\mathbb R}^{d})$ and $W^{k,\,s}(U;\,{\mathbb R}^{d})$ respectively by the Lebesgue space and the Sobolev space. To shorten the notations, we often write $L^{s}(U)\coloneqq L^{s}(U;\,{\mathbb R})$ and $W^{k,\,s}(U)\coloneqq W^{k,\,s}(U;\,{\mathbb R})$. It is mentioned that the Lebesgue space $L^{s}(U;\,{\mathbb R}^{d})$ itself makes sense even when $U$ is only assumed to be Lebesgue measurable.

Throughout this paper, we let
\[{\mathcal L}_{1}u\coloneqq -\divx\,(\nabla E_{1}(\nabla u)),\quad {\mathcal L}_{p}u\coloneqq -\divx\,\mleft(\nabla E_{p}(\nabla u)\mright),\]
where $E_{1}$ and $E_{p}$ are non-negative convex functions in ${\mathbb R}^{n}$.
For regularity of the densities, we require \(E_{1}\in C({\mathbb R}^{n})\cap C_{\mathrm{loc}}^{2,\,\beta_{0}}({\mathbb R}^{n}\setminus\{0\})\) and \(E_{p}\in C^{1}({\mathbb R}^{n})\cap C_{\mathrm{loc}}^{2,\,\beta_{0}}({\mathbb R}^{n}\setminus\{0\})\) with $\beta_{0}\in(0,\,1\rbrack$. In this paper, we let \(E_{1}\) be positively one-homogeneous. In other words, $E_{1}$ satisfies
\begin{equation}\label{Eq: positive hom of deg 1}
E_{1}(\lambda z)=\lambda E_{1}(z)\quad\textrm{for all }z\in{\mathbb R}^{n},\,\lambda>0.
\end{equation}
For \(E_{p}\), we assume that there exists constants \(0<\lambda_{0}\le \Lambda_{0}<\infty\) such that
\begin{equation}\label{Eq: Gradient bound condition for Ep}
\lvert \nabla E_{p}(z)\rvert\le \Lambda_{0}\lvert z\rvert^{p-1}\quad \textrm{for all }z\in{\mathbb R}^{n},
\end{equation}
\begin{equation}\label{Eq: Hessian condition for Ep}
\lambda_{0}\lvert z\rvert^{p-2}\mathrm{id}\leqslant \nabla^{2} E_{p}(z) \leqslant \Lambda_{0}\lvert z\rvert^{p-2}\mathrm{id}\quad \textrm{for all }z\in{\mathbb R}^{n}\setminus\{0\}.
\end{equation}
For $\beta_{0}$-H\"{o}lder continuity of Hessian matrices $\nabla^{2} E_{p}$, it is assumed that
\begin{equation}\label{Eq: C-2-beta growth for Ep}
\mleft\lVert \nabla^{2} E_{p}(z_{1})-\nabla^{2} E_{p}(z_{2})\mright\rVert\le \Lambda_{0}\mu^{p-2-\beta_{0}}\lvert z_{1}-z_{2}\rvert^{\beta_{0}}
\end{equation}
holds for all $\mu\in(0,\,\infty)$, and $z_{1},\,z_{2}\in{\mathbb R}^{n}\setminus\{ 0\}$ satisfying
\begin{equation}\label{Eq: C-2-beta variable condition}
\frac{\mu}{16}\le \lvert z_{1}\rvert\le 3\mu,\quad \textrm{and} \quad \lvert z_{1}-z_{2}\rvert\le \frac{\mu}{32}.
\end{equation}

A typical example is
\begin{equation}\label{Eq: Special case of E1 and Ep}
E_{1}(z)=b\lvert z\rvert,\quad E_{p}(z)=\frac{1}{p}\lvert z\rvert^{p},
\end{equation}
It is easy to check that \(E_{1}\) and \(E_{p}\) defined as in (\ref{Eq: Special case of E1 and Ep}) satisfy (\ref{Eq: positive hom of deg 1})--(\ref{Eq: Hessian condition for Ep}) with \(\lambda_{0}\coloneqq \min\{\,1,\,p-1\,\},\,\Lambda_{0}\coloneqq \max\{\,1,\,p-1\,\}\). 
Moreover, by replacing \(\Lambda_{0}=\Lambda_{0}(n,\,p)\) sufficiently large, we may assume that (\ref{Eq: C-2-beta growth for Ep}) also holds with $\beta_{0}=1$ (see Section \ref{Rmk: Uhlenbeck structure case} for further explanations).

We give the definition of weak solutions to (\ref{Eq: main eq elliptic}).
\begin{definition}\label{Def of Weak sol}\upshape
Let \(p\in(1,\,\infty)\) and \(q\in \lbrack1,\,\infty\rbrack\) satisfy
\begin{equation}\label{Eq: Condition of q general}
\mleft\{\begin{array}{cc} \displaystyle\frac{np}{np-n+p}<q\le \infty& (1<p<n),\\ 1<q\le\infty & (p=n),\\ 1\le q\le \infty & (n<p<\infty).\end{array} \mright.
\end{equation}
For a given function \(f\in L^{q}(\Omega)\), a function \(u\in W^{1,\,p}(\Omega)\) is called a \textit{weak} solution to (\ref{Eq: main eq elliptic}) when there exists a vector field \(Z\in L^{\infty}(\Omega;\,{\mathbb R}^{n})\) such that the pair \((u,\,Z)\) satisfies
\begin{equation}\label{Eq: Weak formulation on the equation}
\int_{\Omega}\mleft\langle Z+\nabla E_{p}(\nabla u)\mathrel{}\middle|\mathrel{} \nabla\phi\mright\rangle\,{\mathrm d}x=\int_{\Omega}f\phi\,{\mathrm d}x\quad \textrm{for all }\phi\in W_{0}^{1,\,p}(\Omega),
\end{equation}
and
\begin{equation}\label{Eq: Subgradient Z}
Z(x)\in\partial E_{1}(\nabla u(x))\quad \textrm{for a.e. }x\in\Omega.
\end{equation}
Here $\partial E_{1}$ denotes the subdifferential of $E_{1}$, defined by
\[\partial E_{1}(z)\coloneqq \left\{\zeta\in{\mathbb R}^{n}\mathrel{} \middle| \mathrel{}E_{1}(w)\ge E_{1}(z)+\langle\zeta\mid w-z \rangle\textrm{ for all }w\in{\mathbb R}^{n} \right\}.\]
\end{definition}
It should be noted that (\ref{Eq: Condition of q general}) enables us to apply the compact embedding \(W_{0}^{1,\,p}(\Omega)\hookrightarrow L^{q^{\prime}}(\Omega)\) (see e.g., \cite[Chapters 4 \& 6]{MR2424078}), so that the weak formulation (\ref{Eq: Weak formulation on the equation}) makes sense. Clearly, the condition $q\in(n,\,\infty\rbrack$ implies (\ref{Eq: Condition of q general}).

Our main result is the following Theorem \ref{Theorem: C1-regularity}.
\begin{theorem}[$C^{1}$-regularity theorem]\label{Theorem: C1-regularity}
Let $p\in(1,\,\infty)$ and $\beta_{0}\in(0,\,1\rbrack$, and assume that the functions \(E_{1}\in C({\mathbb R}^{n})\cap C_{\mathrm{loc}}^{2,\,\beta_{0}}({\mathbb R}^{n}\setminus\{ 0\})\) and \(E_{p}\in C^{1}({\mathbb R}^{n})\cap C_{\mathrm{loc}}^{2,\,\beta_{0}}({\mathbb R}^{n}\setminus\{ 0\})\) satisfy (\ref{Eq: positive hom of deg 1})--(\ref{Eq: C-2-beta growth for Ep}). If \(u\) is a weak solution to (\ref{Eq: main eq elliptic}) with \(f\in L^{q}(\Omega)\) for some $q\in(n,\,\infty\rbrack$, then \(u\) is in \(C^{1}(\Omega)\).
\end{theorem}
In the case (\ref{Eq: Special case of E1 and Ep}), the equation (\ref{Eq: main eq elliptic}) becomes (\ref{Eq: 1,p-Laplace}), and hence Theorem \ref{Theorem: C1-regularity} clearly generalizes Theorem \ref{Thm: C1-regularity model case}.
Although we have succeeded in showing \(C^{1}\)-regularity of weak solutions, it should be mentioned that we have required \(E_{1},\,E_{p}\in C_{\mathrm{loc}}^{2,\,\beta_{0}}({\mathbb R}^{n}\setminus\{ 0\})\) in Theorem \ref{Theorem: C1-regularity}. On the other hand, the previous proof of $C^{1}$-regularity of convex solutions \cite[Theorem 4]{giga2021continuity}, where some tools from convex analysis are basically used, works under weaker assumptions that \(E_{1},\,E_{p}\in C^{2}({\mathbb R}^{n}\setminus\{ 0\})\).
Local $C^{2,\,\beta_{0}}$-regularity of $E=E_{1}+E_{p}$ outside the origin is used for technical reasons to make our perturbation arguments successful (see Remark \ref{Rmk: beta condition} in Section \ref{Subsect: Average integral estimates}).

\subsection{Outlines of the paper}
We conclude Section \ref{Sect: Introduction} by briefly describing outlines of this paper.

Section \ref{Sect: Approximation} aims to give the proof of Theorem \ref{Theorem: C1-regularity} by approximation arguments.
For preliminaries, in Section \ref{Subsect: Some Lipschitz estimates}, we give basic estimates for gradients and Hessian matrices of relaxed density functions. After stating some basic facts on the density \(E_{1}\) in Section \ref{Subsect: Basic facts on convex integrands}, we would like to check that the relaxation of the density $E=E_{1}+E_{p}$ convoluted by the Friedrichs mollifier satisfies good properties in Sections \ref{Subsect: approximation of density}--\ref{Rmk: Uhlenbeck structure case}. Next in Sections \ref{Subsect: Basic results on convergence of vector fields}--\ref{Subsect: Convergence and solvability}, we consider an approximating problem for the equation (\ref{Eq: main eq elliptic}). There we prove a convergence result on variational inequality problems (Proposition \ref{Prop: Elliptic variational inequality}). As a corollary, we are able to deduce that a solution to an approximated equation surely converges to a weak solution to (\ref{Eq: main eq elliptic}), under suitable Dirichlet boundary conditions (Corollary \ref{Cor: unique existence of elliptic Dirichlet prob}). Finally in Section \ref{Subsect: Proof of main theorem}, we would like to prove Theorem \ref{Theorem: C1-regularity}.
There Theorem \ref{Prop: A priori Hoelder estimate} is used without proofs, which states that a truncated gradient of an approximated solution, given by (\ref{Eq: Approximated truncated vector field}), is locally H\"{o}lder continuous, uniformly for an approximation parameter \(\varepsilon\in(0,\,\delta/8)\).

Section \ref{Sect: A priori Hoelder estimate} is focused on giving the proof of Theorem \ref{Prop: A priori Hoelder estimate}.
First in Section \ref{Subsect: A priori Lipschitz bound}, we briefly describe inner regularity on generalized solutions, including a priori Lipschitz bounds (Proposition \ref{Prop: a priori Lipschitz}).
In Section \ref{Subsect: A priori Hoelder estimate completed}, we would like to give the proof of Theorem \ref{Theorem: C1-regularity}. There we use key Propositions \ref{Prop: Perturbation result}--\ref{Prop: De Giorgi Oscillation lemma}, the proofs of which are given after deducing a weak formulation in Section \ref{Subsect: Weak formulations of regularized equations}.
Throughout Sections \ref{Subsect: Perturbation outside facets}--\ref{Subsect: Average integral estimates}, we would like to justify that the vector field \({\mathcal G}_{2\delta,\,\varepsilon}(\nabla u_{\varepsilon})\) satisfies a Campanato-type growth estimate under suitable conditions where a freezing coefficient method works (Proposition \ref{Prop: Perturbation result}). 
In Section \ref{Subsect: De Giorgi Oscillation}, we show a De Giorgi-type lemma on the scalar-valued function \(\lvert{\mathcal G}_{\delta,\,\varepsilon}(\nabla u_{\varepsilon})\rvert\) (Proposition \ref{Prop: De Giorgi Oscillation lemma}).
Some lemmata in Section \ref{Subsect: De Giorgi Oscillation} are easy to deduce, appealing to standard arguments given in \cite[Chapter 3]{MR3887613}, \cite[Chapter 5]{MR3099262} and \cite[Chapter 10, \S 4--5]{MR2566733}. For the reader's convenience, we provide detailed computations in the appendix (Section \ref{Sect: Appendix De Giorgi-type levelset argument}).

Finally, we would like to mention that our result is based on adaptations of \cite{T-system}, the author's recent work that focuses on system problems. This paper also aims to provide full computations and proofs of some estimates that are omitted in \cite{T-system}. Compared with \cite{T-system}, some computations herein might become somewhat complicated, since the density function $E=E_{1}+E_{p}$ is not assumed to be symmetric in scalar problems.

\section{Approximation schemes and the proof of Theorem \ref{Theorem: C1-regularity}}\label{Sect: Approximation}
The aim of Section \ref{Sect: Approximation} is to provide suitable approximation schemes for the equation (\ref{Eq: main eq elliptic}). 
This approximation is necessary since the ellipticity ratio of (\ref{Eq: main eq elliptic}) blows up near the facets, which makes it difficult to show local Sobolev regularity on second order derivatives. Therefore, we have to introduce relaxed equations whose ellipticity ratio is bounded even across the facets.
When considering the operators \({\mathcal L}={\mathcal L}_{1}+{\mathcal L}_{p}\), we have to give a suitable relaxation of the anisotropic operator ${\mathcal L}_{1}$. This is possible by convolving the density function $E_{1}$ by the Friedrichs mollifier. The novelty concerning our relaxation is that this works as long as $E_{1}$ is positively one-homogeneous. 
\subsection{Preliminaries}\label{Subsect: Some Lipschitz estimates}
First in Section \ref{Subsect: Some Lipschitz estimates}, we prove quantitative estimates related to gradients or Hessian matrices for preliminaries. The results herein will play an important role in Sections \ref{Sect: Approximation}--\ref{Sect: A priori Hoelder estimate}.

First, we would like to prove Lemma \ref{Lemma: Local C11 estimate} below.
\begin{lemma}\label{Lemma: Local C11 estimate}
Let \(s\in\lbrack 1,\,\infty)\) and \(\sigma_{1},\,\sigma_{2},\,\sigma_{3}\in \lbrack 0,\,\infty)\) be fixed constants. Assume that a real-valuied function \(H\in C({\mathbb R}^{n})\cap C^{2}({\mathbb R}^{n}\setminus\{ 0\})\) satisfies
\begin{equation}\label{Eq: H-condition 1/3}
\mleft\lvert \nabla H(z)\mright\rvert\le L\mleft(\sigma_{1}+\lvert z\rvert^{s-1}\mright)\quad \textrm{for all }z\in{\mathbb R}^{n}\setminus\{ 0\},
\end{equation}
\begin{equation}\label{Eq: H-condition 2/3}
\mleft\lVert \nabla^{2}H(z)\mright\rVert\le L\mleft(\sigma_{2}+\lvert z\rvert^{s-2}\mright)\quad \textrm{for all }z\in{\mathbb R}^{n}\setminus\{ 0\},
\end{equation}
where \(L\in(0,\,\infty)\) is constant. Then, for all \(z_{1},\,z_{2}\in{\mathbb R}^{n}\setminus\{ 0\}\), there holds
\begin{align}\label{Eq: Local C11 estimate of H}
&\lvert \nabla H(z_{1})-\nabla H(z_{2})\rvert\nonumber\\&\le C(s)L\cdot {\mathcal E}(\sigma_{1},\,\sigma_{2},\,\lvert z_{1}\rvert,\,\lvert z_{2}\rvert)\cdot\mleft(\lvert z_{1}\rvert^{-1}\wedge \lvert z_{2}\rvert^{-1}\mright)\lvert z_{1}-z_{2}\rvert,
\end{align}
where \(C(s)\in(0,\,\infty)\) is constant and 
\[{\mathcal E}(\sigma_{1},\,\sigma_{2},\,t_{1},\,t_{2})\coloneqq \mleft(t_{1}^{s-1}+t_{2}^{s-1}\mright)+\sigma_{1}+\sigma_{2}\cdot\mleft(t_{1}+t_{2}\mright)\quad \textrm{for }t_{1},\,t_{2}>0.\]
Moreover, if \(H\) satisfies \(H\in C^{3}({\mathbb R}^{n}\setminus\{0\})\) and there holds
\begin{equation}\label{Eq: H-condition 3/3}
\mleft\lvert \nabla^{3} H(z) \mright\rvert\le L\mleft(\sigma_{3}+\lvert z\rvert^{s-3}\mright)\quad \textrm{for all }z\in{\mathbb R}^{n}\setminus\{ 0\},
\end{equation}
then we have 
\begin{align}
&\mleft\lvert \nabla^{2}H(z_{1})(z_{2}-z_{1})-(\nabla H(z_{2})-\nabla H(z_{1})) \mright\rvert\nonumber\\& \le C(s)L \mleft[{\mathcal E}(\sigma_{1},\,\sigma_{2},\,\lvert z_{1}\rvert,\,\lvert z_{2}\rvert)+\sigma_{3}\lvert z_{1}\rvert^{2}\mright]\frac{\lvert z_{1}-z_{2}\rvert^{2}}{\lvert z_{1}\rvert^{2}} \label{Eq: Error estimate on Hess H}
\end{align}
for all \(z_{1},\,z_{2}\in{\mathbb R}^{n}\setminus\{ 0\}\). Here \(C(s)\in(0,\,\infty)\) is constant.
\end{lemma}
\begin{proof}
We distinguish (\ref{Eq: Local C11 estimate of H}) between two cases. In the case \(\lvert z_{1}-z_{2}\rvert\ge \lvert z_{1}\rvert/2\), we just use (\ref{Eq: H-condition 1/3}) and get
\begin{align*}
\lvert \nabla H_{1}(z_{1})-\nabla H_{1}(z_{2})\rvert&\le L\mleft[2\sigma_{1}+\mleft(\lvert z_{1}\rvert^{s-1}+\lvert z_{2}\rvert^{s-1}\mright)\mright]  \\&\le 4L\mleft[\sigma_{1}+\mleft(\lvert z_{1}\rvert^{s-1}+\lvert z_{2}\rvert^{s-1}\mright)\mright]\frac{\lvert z_{1}-z_{2}\rvert}{\lvert z_{1}\rvert}\\&\le 4L\cdot {\mathcal E}(\sigma_{1},\,\sigma_{2},\,\lvert z_{1}\rvert,\,\lvert z_{2}\rvert)\cdot\frac{\lvert z_{1}-z_{2}\rvert}{\lvert z_{1}\rvert}.
\end{align*}
In the remaining case \(\lvert z_{1}-z_{2}\rvert\le \lvert z_{1}\rvert/2\), which yields \(\lvert z_{2}\rvert\le \lvert z_{2}-z_{1}\rvert+\lvert z_{1}\rvert<3\lvert z_{1}\rvert/2\), it is easy to check that
\begin{equation}\label{Eq: Closed line segment separated}
\frac{5}{2}\lvert z_{1}\rvert>\lvert z_{1}\rvert+\lvert z_{2}\rvert \ge\lvert z_{2}+t(z_{1}-z_{2})\rvert\ge\lvert z_{1}\rvert-\lvert z_{1}-z_{2}\rvert>\frac{1}{2}\lvert z_{1}\rvert>0
\end{equation}
for all $t\in\lbrack 0,\,1\rbrack$.
Hence, from (\ref{Eq: H-condition 2/3}), it follows that
\[\mleft\lVert \nabla^{2} H(z_{1}+t(z_{2}-z_{1}))\mright\rVert\le L\mleft[\sigma_{2}+2^{-s+2}\mleft(5^{s-2}\vee 1\mright)\lvert z_{1}\rvert^{s-2}\mright]\quad\textrm{for all }t\in \lbrack 0,\,1\rbrack.\]
In particular, by \(E_{1}\in C^{2}({\mathbb R}^{2}\setminus\{ 0\})\), we can compute
\begin{align*}
\lvert \nabla H(z_{1})-\nabla H(z_{2})\rvert&\le\mleft\lvert\int_{0}^{1}\nabla^{2} H(z_{2}+t(z_{1}-z_{2}))\cdot (z_{1}-z_{2})\,\mathrm{d}t\mright\rvert\\&\le \mleft(\int_{0}^{1}\mleft\lVert \nabla^{2}H(z_{2}+t(z_{1}-z_{2}))\mright\rVert\,{\mathrm d}t\mright)\lvert z_{1}-z_{2}\rvert\\&\le L\mleft[\sigma_{2}\lvert z_{1}\rvert+2^{-s+2}\mleft(5^{s-2}\vee 1\mright)\lvert z_{1}\rvert^{s-1}\mright]\frac{\lvert z_{1}-z_{2}\rvert}{\lvert z_{1}\rvert}\\&\le C(s) L\cdot {\mathcal E}(\sigma_{1},\,\sigma_{2},\,\lvert z_{1}\rvert,\,\lvert z_{2}\rvert)\cdot\frac{\lvert z_{1}-z_{2}\rvert}{\lvert z_{1}\rvert}.
\end{align*}
With \(z_{1}\) and \(z_{2}\) replaced by each other, we obtain (\ref{Eq: Local C11 estimate of H}).

We prove (\ref{Eq: Error estimate on Hess H}) by dividing into two cases. When \(\lvert z_{1}-z_{2}\rvert\ge \lvert z_{1}\rvert/2\), we use (\ref{Eq: H-condition 2/3})--(\ref{Eq: Local C11 estimate of H}) to get
\begin{align*}
&\mleft\lvert \nabla^{2}H(z_{1})(z_{2}-z_{1})-(\nabla H(z_{2})-\nabla H(z_{1})) \mright\rvert\\&\le \mleft\lVert \nabla^{2}H(z_{1})\mright\rVert \lvert z_{2}-z_{1}\rvert+\mleft\lvert \nabla H(z_{2})-\nabla H(z_{1}) \mright\rvert \\&\le 2L\mleft(\sigma_{2}\lvert z_{1}\rvert+\lvert z_{1}\rvert^{s-1} \mright)\frac{\lvert z_{1}-z_{2}\rvert}{\lvert z_{1}\rvert}+C(s)L{\mathcal E}(\sigma_{1},\,\sigma_{2},\,\lvert z_{1}\rvert,\,\lvert z_{2}\rvert)\frac{\lvert z_{1}-z_{2}\rvert}{\lvert z_{1}\rvert}\\&\le C(s)L{\mathcal E}(\sigma_{1},\,\sigma_{2},\,\lvert z_{1}\rvert,\,\lvert z_{2}\rvert)\frac{\lvert z_{1}-z_{2}\rvert^{2}}{\lvert z_{1}\rvert^{2}}.
\end{align*}
In the remaining case \(\lvert z_{1}-z_{2}\rvert\le \lvert z_{1}\rvert/2\), we recall (\ref{Eq: Closed line segment separated}), which yields
\[\mleft\lvert \nabla^{3}H(z_{1}+t(z_{2}-z_{1})) \mright\rvert\le L\mleft[\sigma_{3}+2^{-s+3}\mleft(5^{s-3}\vee 1\mright)\lvert z_{1}\rvert^{s-3}\mright]\quad\textrm{for all }t\in\lbrack 0,\,1\rbrack.\]
By this and \(H\in C^{3}({\mathbb R}^{n}\setminus\{ 0\})\), we obtain
\begin{align*}
&\mleft\lvert \nabla^{2} H(z_{1})(z_{2}-z_{1})-(\nabla H(z_{2})-\nabla H(z_{1}))\mright\rvert\\&=\mleft\lvert\mleft(\int_{0}^{1}\mleft[\nabla^{2} H(z_{1})- \nabla^{2}H(z_{1}+t(z_{2}-z_{1}))\mright]\,{\mathrm d}t\mright)(z_{2}-z_{1})\mright\rvert\\&\le \lvert z_{2}-z_{1}\rvert\int_{0}^{1}\mleft\lvert \nabla^{2} H(z_{1})- \nabla^{2}H(z_{1}+t(z_{2}-z_{1}))\mright\rvert\,{\mathrm d}t\\&\le \lvert z_{2}-z_{1}\rvert\int_{0}^{1}\mleft(\int_{0}^{1}\mleft\lvert \nabla^{3}H(z_{1}+(1-s)t(z_{2}-z_{1})) \mright\rvert^{2}\,{\mathrm d}s\mright)^{1/2} t\lvert z_{2}-z_{1}\rvert\,{\mathrm d}t\\&\le C(s)L\mleft(\sigma_{3}\lvert z_{1}\rvert^{2}+\lvert z_{1}\rvert^{s-1}\mright)\frac{\lvert z_{1}-z_{2}\rvert^{2}}{\lvert z_{1}\rvert^{2}}\\&\le C(s)L \mleft[{\mathcal E}(\sigma_{1},\,\sigma_{2},\,\lvert z_{1}\rvert,\,\lvert z_{2}\rvert)+\sigma_{3}\lvert z_{1}\rvert^{2}\mright]\frac{\lvert z_{1}-z_{2}\rvert^{2}}{\lvert z_{1}\rvert^{2}}.
\end{align*}
In both cases, we are able to conclude (\ref{Eq: Error estimate on Hess H}).
\end{proof}
Secondly, we would like to show that some growth estimates given in Lemma \ref{Lemma: Local C11 estimate} will be kept under the convolution by mollifiers (Lemma \ref{Lemma: Mollified estimates}). Also, we will check that ellipticity or monotonicity estimates can also be obtained (Lemma \ref{Lemma: Coerciveness keeps under convolution}).

Let \(\{j_{\varepsilon}\}_{0<\varepsilon<1}\subset C_{c}^{\infty}({\mathbb R}^{n})\) denote the Friedrichs mollifiers. That is, we fix a spherically symmetric real-valued function \(j\in C_{c}^{\infty}({\mathbb R}^{n})\) satisfying
\[0\le j\textrm{ in }{\mathbb R}^{n},\,\mathrm{supp}\,j\subset\overline{B_{1}(0)},\,\int_{{\mathbb R}^{n}}\,j(x)\,\mathrm{d}x=1,\]
and define
\[0\le j_{\varepsilon}(x)\coloneqq \varepsilon^{-n}j(x/\varepsilon)\quad \textrm{for}\quad x\in{\mathbb R}^{n},\,0<\varepsilon<1.\]
For a given function \(f\in L_{\mathrm{loc}}^{1}({\mathbb R}^{n};\,{\mathbb R}^{m})\), we define \(j_{\varepsilon}\ast f\in C^{\infty}({\mathbb R}^{n};\,{\mathbb R}^{m})\) by
\[(j_{\varepsilon}\ast f)(x)\coloneqq\int_{{\mathbb R}^{n}}j_{\varepsilon}(y)f(x-y)\,{\mathrm d}y\in{\mathbb R}^{m}\]
for each \(x\in{\mathbb R}^{n}\).
We note that when \(f\in W_{\mathrm{loc}}^{1,\,s}({\mathbb R}^{n})\,(1\le s\le \infty)\), there holds \(\nabla (j_{\varepsilon}\ast f)=j_{\varepsilon}\ast \nabla f\) by the definition of Sobolev derivatives. 
In Lemmata \ref{Lemma: Mollified estimates}--\ref{Lemma: Coerciveness keeps under convolution}, we show some quantitative estimates related to functions convoluted by mollifiers. Among them, the inequality (\ref{Eq: Error estimate on Hessian}) in Lemma \ref{Lemma: Mollified estimates} will become a key tool. Although some estimates in Lemmata \ref{Lemma: Mollified estimates}--\ref{Lemma: Coerciveness keeps under convolution} might be well-known in existing literatures (see e.g., \cite{MR1634641}, \cite{MR1612389}), we would like to provide elementary proofs for the reader's convenience.

\begin{lemma}[Growth estimates]\label{Lemma: Mollified estimates}
For a fixed constant \(s\in\lbrack 1,\,\infty)\), we have the following: 
\begin{enumerate}
\item \label{Item 1/3: Growth} Assume that real-valued functions $\{H_{\varepsilon}\}_{0<\varepsilon<1}\subset C^{2}({\mathbb R}^{n})$ satisfy
\begin{equation}\label{Eq: Gradient bound for mollified function}
\mleft\lvert \nabla H_{\varepsilon}(z)\mright\rvert\le L\mleft(\varepsilon^{2}+\lvert z\rvert^{2}\mright)^{(s-1)/2}\quad \textrm{for all }z\in{\mathbb R}^{n},
\end{equation}
\begin{equation}\label{Eq: Hessian bound for mollified function}
\mleft\lVert \nabla^{2} H_{\varepsilon}(z)\mright\rVert\le L\mleft(\varepsilon^{2}+\lvert z\rvert^{2}\mright)^{s/2-1}\quad \textrm{for all }z\in{\mathbb R}^{n},
\end{equation}
where the constant $L\in(0,\,\infty)$ is independent of $\varepsilon\in(0,\,1)$. If $1\le s<2$, then there holds
\begin{equation}\label{Eq: L2 upper bound lemma case p<2}
\lvert \nabla H_{\varepsilon}(z_{1})-\nabla H_{\varepsilon}(z_{2})\rvert\le C\min\mleft\{\,\lvert z_{1}\rvert^{s-2},\,\lvert z_{2}\rvert^{s-2} \,\mright\}\lvert z_{1}-z_{2}\rvert
\end{equation}
for all \((z_{1},\,z_{2})\in({\mathbb R}^{n}\times{\mathbb R}^{n})\setminus\{(0,\,0)\}\). Here the constant $C\in(0,\,\infty)$ is given by $C\coloneqq 2^{-s}L$.
\item \label{Item 2/3: Growth molified} Assume that a real-valued function \(H\in C({\mathbb R}^{n})\cap C^{2}({\mathbb R}^{n}\setminus\{0\})\) satisfies (\ref{Eq: H-condition 1/3})--(\ref{Eq: H-condition 2/3}) with \(\sigma_{1}=\sigma_{2}=0\) and \(L_{0}\in(0,\,\infty)\). Then, there exists a constant $\Gamma_{0}=\Gamma_{0}(n,\,s)\in(0,\,\infty)$ such that the function \(H_{\varepsilon}\coloneqq j_{\varepsilon}\ast H\in C^{\infty}({\mathbb R}^{n})\,(0<\varepsilon<1)\) satisfies (\ref{Eq: Gradient bound for mollified function})--(\ref{Eq: Hessian bound for mollified function}) with $L=\Gamma_{0}L_{0}\in(0,\,\infty)$. 
In particular, $H_{\varepsilon}$ satisfies (\ref{Eq: L2 upper bound lemma case p<2}) with $C\coloneqq 2^{-s}\Gamma_{0}L_{0}\in(0,\,\infty)$, provided $1\le s<2$.
\item \label{Item 3/3: Hessian errors} In addition to the assumptions in \ref{Item 2/3: Growth molified}, let $H$ be in $C_{\mathrm{loc}}^{2,\,\beta_{0}}({\mathbb R}^{n}\setminus\{0\})$, and satisfy
\begin{equation}\label{Eq: C-2-beta assumption on H}
\mleft\lVert \nabla^{2} H(z_{1})-\nabla^{2} H(z_{2}) \mright\rVert\le L_{0}\mu^{s-2-\beta_{0}}\lvert z_{1}-z_{2}\rvert^{\beta_{0}}
\end{equation}
for all $z_{1},\,z_{2}\in{\mathbb R}^{n}$ enjoying (\ref{Eq: C-2-beta variable condition}) with $\mu\in(0,\,\infty)$.
Also, assume that positive numbers $\delta,\,\varepsilon$ satisfy
\begin{equation}\label{Eq: Range of delta-epsilon}
0<\varepsilon<\frac{\delta}{8},\quad \textrm{and}\quad 0<\delta<1.
\end{equation} 
Then, for all $\mu\in(\delta,\,\infty)$, and \(z_{1},\,z_{2}\in{\mathbb R}^{n}\setminus\{ 0\}\) satisfying
\begin{equation}\label{Eq: Outside condition on variable}
\delta+\frac{\mu}{4}\le \lvert z_{1}\rvert\le \delta+\mu,\quad \lvert z_{2}\rvert\le \delta+\mu,
\end{equation}
we have
\begin{equation}\label{Eq: Error estimate on Hessian}
\mleft\lvert \nabla^{2} H_{\varepsilon}(z_{1})(z_{2}-z_{1})-(\nabla H_{\varepsilon}(z_{2})-\nabla H_{\varepsilon}(z_{1}))\mright\rvert\le CL_{0}\mu^{s-2-\beta_{0}}\lvert z_{1}-z_{2}\rvert^{1+\beta_{0}},
\end{equation}
where $C\in(0,\,\infty)$ depends at most on $n$, $s$, and $\beta_{0}$.
\end{enumerate}
\end{lemma}
\begin{proof}
\ref{Item 1/3: Growth}.
By \(s/2-1<0\) and (\ref{Eq: Hessian bound for mollified function}), we compute
\begin{align*}
\lvert \nabla H_{\varepsilon}(z_{1})-\nabla H_{\varepsilon}(z_{2})\rvert&=\mleft\lvert\int_{0}^{1} \nabla^{2} H_{\varepsilon}(tz_{1}+(1-t)z_{2})\cdot (z_{1}-z_{2})\,{\mathrm d}t\mright\rvert\\&\le \lvert z_{1}-z_{2}\rvert\int_{0}^{1}\mleft\lVert\nabla^{2}H_{\varepsilon}(tz_{1}+(1-t)z_{2})\mright\rVert\,{\mathrm d}t \\&\le L\lvert z_{1}-z_{2}\rvert\int_{0}^{1}\mleft(\varepsilon^{2}+\lvert tz_{1}+(1-t)z_{2}\rvert^{2}\mright)^{s/2-1}\,{\mathrm d}t\\&\le L\lvert z_{1}-z_{2}\rvert\int_{0}^{1}\lvert tz_{1}+(1-t)z_{2}\rvert^{s-2}\,{\mathrm d}t.
\end{align*}
Hence it suffices to show
\begin{equation}\label{Eq: Claim for L2 upper bound lemma case p<2}
\int_{0}^{1} \lvert tz_{1}+(1-t)z_{2}\rvert^{s-2}\,{\mathrm d}t \le 2^{-s}\min\mleft\{\,\lvert z_{1}\rvert^{s-2},\,\lvert z_{2}\rvert^{s-2}\,\mright\} 
\end{equation}
to prove (\ref{Eq: L2 upper bound lemma case p<2}).
Without loss of generality we may assume that \(\lvert z_{1}\rvert\le \lvert z_{2}\rvert\). Then, \(z_{2}\not= 0\) clearly follows from $(z_{1},\,z_{2})\neq(0,\,0)$. By the triangle inequality, we have
\begin{equation}\label{Eq: triangle inequality 1}
\lvert tz_{1}+(1-t)z_{2}\rvert\ge (1-t)\lvert z_{2}\rvert-t\lvert z_{1}\rvert\ge(1-2t)\lvert z_{2}\rvert\ge \frac{1}{2}\lvert z_{2}\rvert>0\quad \textrm{when }0\le t\le \frac{1}{4}.
\end{equation}
Again by $s-2<0$, we obtain
\[\int_{0}^{1}\lvert tz_{1}+(1-t)z_{2}\rvert^{s-2}\,{\mathrm d}t\le \frac{\lvert z_{2}\rvert^{s-2}}{2^{s-2}}\cdot\frac{1}{4}=\frac{\lvert z_{2}\rvert^{s-2}}{2^{s}},\]
which completes the proof of (\ref{Eq: Claim for L2 upper bound lemma case p<2}).

\ref{Item 2/3: Growth molified}.
To find the constant $\Gamma_{0}(n,\,s)\in(0,\,\infty)$, for each fixed constant \(\sigma\in\lbrack -1,\,\infty)\), we would like to show
\begin{equation}\label{Eq: Mollified estimate on absolute value functions}
h_{\sigma,\,\varepsilon}(z)\coloneqq (j_{\varepsilon}\ast \lvert\,\cdot\,\rvert^{\sigma})(z)\le C(n,\,\sigma)\mleft(\varepsilon^{2}+\lvert z\rvert^{2}\mright)^{\sigma/2}\quad \textrm{for all }z\in{\mathbb R}^{n},\,\varepsilon\in(0,\,1).
\end{equation}
Here the constant \(C(n,\,\sigma)\in(0,\,\infty)\) is independent of $\varepsilon\in(0,\,1)$.
We note that the inclusion \(\lvert\,\cdot\,\rvert^{\sigma}\in L_{\mathrm{loc}}^{1}({\mathbb R}^{n})\) follows from $n\ge 2$ and $-1\le \sigma<\infty$. Hence, we are able to define \(h_{\sigma}\coloneqq j\ast\lvert\,\cdot\,\rvert^{\sigma}\in C^{\infty}({\mathbb R}^{n})\).
By change of variables, we have
\begin{align}
h_{\sigma,\,\varepsilon}(z)&=\int_{B_{\varepsilon}(0)}j_{\varepsilon}(y)\lvert z-y\rvert^{\sigma}{\mathrm d}y\label{Eq: First express of h-sigma-epsilon}\\&=\int_{B_{1}(0)}j({\tilde y})\lvert z-{\tilde y}/\varepsilon\rvert^{\sigma}\,{\mathrm{d}}{\tilde y}\nonumber\\&=\varepsilon^{\sigma}\int_{B_{1}(0)}j({\tilde y})\lvert z/\varepsilon-{\tilde y}\rvert^{\sigma}\,{\mathrm d}{\tilde y}=\varepsilon^{\sigma}h_{\sigma}(z/\varepsilon).\label{Eq: Second express of h-sigma-epsilon}
\end{align}
We distinguish (\ref{Eq: Mollified estimate on absolute value functions}) between two cases. In the case \(\lvert z\rvert\le 2\varepsilon\), we use \(\varepsilon\le \sqrt{\varepsilon^{2}+\lvert z\rvert^{2}}\le \sqrt{5}\varepsilon\) and (\ref{Eq: Second express of h-sigma-epsilon}). Then, we have
\[h_{\sigma}(z)\le C(n,\,\sigma)\mleft(1\vee 5^{-\sigma/2} \mright)\mleft(\varepsilon^{2}+\lvert z\rvert^{2}\mright)^{\sigma/2}\quad \textrm{with }C(n,\,\sigma)\coloneqq \max_{B_{2}(0)}h_{\sigma}<\infty.\]
In the remaining case \(\lvert z\rvert> 2\varepsilon\), we take the unique number \(j\in{\mathbb N}\) such that \(2^{j}\varepsilon\le \lvert z\rvert<2^{j+1}\varepsilon\). Then, for all \(y\in B_{\varepsilon}(0)\), we have
\[\frac{\lvert z-y\rvert^{2}}{\varepsilon^{2}+\lvert z\rvert^{2}}\ge \frac{(2^{j}-1)^{2}\varepsilon^{2}}{(4^{j+1}+1)\varepsilon^{2}}=\frac{\mleft(1-2^{-j}\mright)^{2}}{4+4^{-j}}\ge \frac{(1/2)^{2}}{4+1/4}=\frac{1}{17},\]
and
\[\frac{\lvert z-y\rvert^{2}}{\varepsilon^{2}+\lvert z\rvert^{2}}\le 2\cdot\frac{\lvert y\rvert^{2}+\lvert z\rvert^{2}}{\varepsilon^{2}+\lvert z\rvert^{2}}\le 2.\]
By these inequalities and (\ref{Eq: First express of h-sigma-epsilon}), we have
\[h_{\sigma,\,\varepsilon}(z)\le \mleft[2^{\sigma/2}\vee 17^{-\sigma/2}\mright]\mleft(\varepsilon^{2}+\lvert z\rvert^{2}\mright)^{\sigma/2}.\]

The estimate (\ref{Eq: Gradient bound for mollified function}) follows from (\ref{Eq: Mollified estimate on absolute value functions}). In fact, there holds \(\nabla H_{\varepsilon}=j_{\varepsilon}\ast\nabla H\) a.e. in \({\mathbb R}^{n}\), since \(H\in W_{\mathrm{loc}}^{1,\,\infty}({\mathbb R}^{n})\) follows from (\ref{Eq: H-condition 1/3}) with $\sigma_{1}=0$. By applying (\ref{Eq: Mollified estimate on absolute value functions}) with $\sigma=s-1$, (\ref{Eq: Gradient bound for mollified function}) is easily obtained. 
To deduce (\ref{Eq: Hessian bound for mollified function}), we recall that \(H\) satisfies (\ref{Eq: Local C11 estimate of H}) with $\sigma_{1}=\sigma_{2}=0$, and therefore
\begin{align}
\mleft\lvert\nabla H_{\varepsilon}(z_{1})-\nabla H_{\varepsilon}(z_{2}) \mright\rvert&=\mleft\lvert\int_{B_{\varepsilon}(0)}j_{\varepsilon}(y)\mleft[\nabla H(z_{1}-y)-\nabla H(z_{2}-y) \mright]\,{\mathrm d}y \mright\rvert\nonumber\\&\le C(s)L_{0}\lvert z_{1}-z_{2}\rvert \int_{B_{\varepsilon}(0)}j_{\varepsilon}(y)\frac{\lvert z_{1}-y\rvert^{s-1}+\lvert z_{2}-y\rvert^{s-1}}{\lvert z_{2}-y\rvert}\,{\mathrm d}y\label{Eq: Difference quotient estimate for H-epsilon}
\end{align}
holds for all \(z_{1},\,z_{2}\in{\mathbb R}^{n}\).
Fix \(z_{0}\in {\mathbb R}^{n}\), and let \(\lambda\in{\mathbb R}\) be an arbitrary eigenvalue of the Hessian matrix \(\nabla^{2}H_{\varepsilon}(z_{0})\). We take a unit eigenvector $v_{0}$ corresponding to this \(\lambda\). By \(H_{\varepsilon}\in C^{\infty}({\mathbb R}^{n})\) and (\ref{Eq: Difference quotient estimate for H-epsilon}), we obtain
\begin{align*}
\lvert\lambda\rvert&=\mleft\lvert\mleft\langle \nabla^{2}H_{\varepsilon}(z_{0})v_{0}\mathrel{}\middle|\mathrel{}v_{0}\mright\rangle\mright\rvert\\&= \lim_{t\to 0}\,\mleft\lvert\frac{\mleft\langle \nabla H_{\varepsilon}(z_{0}+tv_{0})-\nabla H_{\varepsilon}(z_{0})\mathrel{}\middle|\mathrel{}v_{0}\mright\rangle}{t}\mright\rvert\\ &\le \limsup_{t\to 0}\,\mleft\lvert \frac{\nabla H_{\varepsilon}(z_{0}+tv_{0})-\nabla H_{\varepsilon}(z_{0})}{t}\mright\rvert\\&=C(s)L\mleft(h_{s-2,\,\varepsilon}(z_{0})+\limsup_{t\to 0}\int_{B_{\varepsilon}(0)}\frac{j_{\varepsilon}(y)}{\lvert z_{0}-y\rvert}\lvert z_{0}+t\nu_{0}-y\rvert^{s-1}\,\mathrm{d}y\mright)\\&=2C(s)Lh_{s-2,\,\varepsilon}(z_{0})\le C(n,\,s)L\mleft(\varepsilon^{2}+\lvert z_{0}\rvert^{2}\mright)^{s/2-1}.
\end{align*}
by (\ref{Eq: Mollified estimate on absolute value functions}) and Lebesgue's dominated convergence theorem. Here we note again that the inclusion $j_{\varepsilon}\lvert z_{0}-\,\cdot\,\rvert^{-1}\in L^{1}(B_{\varepsilon}(0))$ holds by $n\ge 2$. Also, the uniform convergence 
\[\lvert z_{0}+t\nu_{0}-y\rvert^{s-1}\to \lvert z_{0}-y\rvert^{s-1}\quad \textrm{uniformly for }y\in B_{\varepsilon}(0)\]
easily follows from $s\ge 1$. By these facts, we are able to apply Lebesgue's dominated convergence theorem.
This completes the proof of (\ref{Eq: Hessian bound for mollified function}).

\ref{Item 3/3: Hessian errors}.
Let \(z_{1},\,z_{2}\in{\mathbb R}^{n}\) satisfy (\ref{Eq: Outside condition on variable}). We first consider \(\lvert z_{1}-z_{2}\rvert<\mu/32\). In this case, it should be noted that for $y\in B_{\varepsilon}(0)$, there holds 
\[\delta+\frac{\mu}{16}\le \lvert z_{1}-y\rvert\le 3\mu\]
by (\ref{Eq: Range of delta-epsilon})--(\ref{Eq: Outside condition on variable}) and $0<\delta<\mu$. Thus, we can apply (\ref{Eq: C-2-beta assumption on H}) to obtain
\begin{align*}
&\mleft\lvert\nabla^{2} H_{\varepsilon}(z_{1})(z_{2}-z_{1})-(\nabla H_{\varepsilon}(z_{2})-\nabla H_{\varepsilon}(z_{1}))\mright\rvert\\&=\mleft\lvert \mleft(\int_{0}^{1}\mleft[\nabla^{2} H_{\varepsilon}(z_{1})-\nabla^{2}H_{\varepsilon}(z_{1}+t(z_{2}-z_{1}))\mright]\,{\mathrm d}t\mright)(z_{2}-z_{1})\mright\rvert\\&\le \lvert z_{1}-z_{2}\rvert\int_{0}^{1}\mleft(\int_{B_{\varepsilon}(0)}\mleft\lVert \nabla^{2}H(z_{1}-y)-\nabla^{2}H(z_{1}+t(z_{2}-z_{1})-y)\mright\rVert j_{\varepsilon}(y)\,{\mathrm d}y\mright)\,{\mathrm d}t\\&\le L_{0}\mu^{s-2-\beta_{0}}\lvert z_{1}-z_{2}\rvert^{1+\beta_{0}}.
\end{align*}Here we have used an identity \(\nabla^{2}H_{\varepsilon}=j_{\varepsilon}\ast\nabla^{2}H\) in \({\mathbb R}^{n}\setminus\overline{B_{\delta}(0)}\).
In fact, for each fixed \({\hat \varepsilon}>0\), there holds \(H\in W_{\mathrm{loc}}^{2,\,\infty}({\mathbb R}^{n}\setminus\overline{B_{{\hat\varepsilon}}(0)})\) by (\ref{Eq: H-condition 2/3}), which yields \(\nabla^{2}H_{\varepsilon}=j_{\varepsilon}\ast \nabla^{2}H\) in \({\mathbb R}^{n}\setminus \overline{B_{\varepsilon+{\hat\varepsilon}}(0)}\).
By (\ref{Eq: Range of delta-epsilon}), we may take \({\hat \varepsilon}\coloneqq \delta-\varepsilon>0\) to get this identity.
In the remaining case \(\lvert z_{1}-z_{2}\rvert>\mu/32\), we use (\ref{Eq: Local C11 estimate of H}) and (\ref{Eq: Difference quotient estimate for H-epsilon}) to compute
\begin{align*}
&\mleft\lvert\nabla^{2} H_{\varepsilon}(z_{1})(z_{2}-z_{1})-(\nabla H_{\varepsilon}(z_{2})-\nabla H_{\varepsilon}(z_{1}))\mright\rvert\\& \le \mleft\lVert\nabla^{2}H_{\varepsilon}(z_{1}) \mright\rVert\lvert z_{1}-z_{2}\rvert+C(s)L_{0}\lvert z_{1}-z_{2}\rvert \int_{B_{\varepsilon}(0)}j_{\varepsilon}(y)\frac{\lvert z_{1}-y\rvert^{s-1}+\lvert z_{2}-y\rvert^{s-1}}{\lvert z_{1}-y\rvert}\,{\mathrm d}y\\&\le C(n,\,s)L_{0}\lvert z_{1}-z_{2}\rvert \underbrace{\mleft[\mleft(\varepsilon^{2}+\lvert z_{1}\rvert^{2}\mright)^{s/2-1}+h_{s-2,\,\varepsilon}(z_{1})+h_{-1,\,\varepsilon}(z_{1})\sup_{B_{\varepsilon}(0)}\lvert z_{2}-y\rvert^{s-1} \mright]}_{\eqqcolon {\mathbf E}}.
\end{align*}
From (\ref{Eq: Range of delta-epsilon})--(\ref{Eq: Outside condition on variable}), it is easy to get
\[\frac{\mu^{2}}{16}\le \lvert z_{1}\rvert^{2}\le\varepsilon^{2}+\lvert z_{1}\rvert^{2}\le \delta^{2}+(\delta+\mu)^{2}\le 5\mu^{2},\]
and
\[\lvert z_{2}-y\rvert\le \lvert z_{2}\rvert+\lvert y\rvert\le \delta+\mu\le 2\mu\quad \textrm{for all }y\in B_{\varepsilon}(0).\]
By these estimates and (\ref{Eq: Mollified estimate on absolute value functions}), we have
\[{\mathbf E}\le C(n,\,s)\mu^{s-2}\cdot \mleft(\frac{32\lvert z_{1}-z_{2}\rvert}{\mu}\mright)^{\beta_{0}}\eqqcolon C(n,\,s,\,\beta_{0})\mu^{s-2-\beta_{0}}\lvert z_{1}-z_{2}\rvert^{\beta_{0}},\]
which yields (\ref{Eq: Error estimate on Hessian}).
\end{proof}
\begin{lemma}[Monotonicity estimates]\label{Lemma: Coerciveness keeps under convolution}
For $s\in(1,\,\infty)$ and $\varepsilon\in(0,\,1)$, we have the following:
\begin{enumerate}
\item \label{Item 1/2: Monotonicity} Assume that $\{H_{\varepsilon}\}_{0<\varepsilon<1}\subset C^{2}({\mathbb R}^{n})$ satisfy 
\begin{equation}\label{Eq: Coerciveness on H-epsilon}
\nabla^{2}H_{\varepsilon}(z)\geqslant \gamma\mleft(\varepsilon^{2}+\lvert z\rvert^{2}\mright)^{s/2-1}\mathrm{id}\quad \textrm{for all }z\in{\mathbb R}^{n}
\end{equation}
for some constant $\gamma\in(0,\,\infty)$. If $s\in\lbrack 2,\,\infty)$, then we have
\begin{equation}\label{Eq: L2 lower bound lemma case p>2}
\mleft\langle \nabla H_{\varepsilon}(z_{1})-\nabla H_{\varepsilon}(z_{2})\mathrel{} \middle|\mathrel{} z_{1}-z_{2} \mright\rangle\ge C\max\mleft\{\,\lvert z_{1}\rvert^{s-2},\,\lvert z_{2}\rvert^{s-2}\,\mright\} \lvert z_{1}-z_{2}\rvert^{2}
\end{equation}
for all \(z_{1},\,z_{2}\in{\mathbb R}^{n}\). Here the constant $C\in(0,\,\infty)$ is given by $C(s,\,\gamma)\coloneqq 2^{-s}\gamma$.
\item\label{Item 2/2: Monotonicity molified} Let a convex function $H\in C^{1}({\mathbb R}^{n})\cap C^{2}({\mathbb R}^{n}\setminus\{ 0\})$ satisfy
\begin{equation}\label{Eq: Coerciveness on H}
\nabla^{2}H(z)\geqslant L_{0}\lvert z\rvert^{s-2}\mathrm{id}\quad \textrm{for all }z\in{\mathbb R}^{n}\setminus\{ 0\}
\end{equation}
for some constants $L_{0}\in(0,\,\infty),\,s\in(1,\,\infty)$. Then, there exists a constant $\gamma_{0}=\gamma_{0}(n,\,s)\in (0,\,1)$ such that for every $\varepsilon\in(0,\,1)$, $H_{\varepsilon}\coloneqq j_{\varepsilon}\ast H\in C^{\infty}({\mathbb R}^{n})$ satisfies (\ref{Eq: Coerciveness on H-epsilon}) with $\gamma\coloneqq \gamma_{0}L_{0}\in(0,\,\infty)$.
In particular, $H_{\varepsilon}$ satisfies (\ref{Eq: L2 lower bound lemma case p>2}) with $C\coloneqq 2^{-s}\gamma_{0}L_{0}\in(0,\,\infty)$, provided $2\le s<\infty$.
\end{enumerate}
\end{lemma}
\begin{proof}
As a preliminary, we check that when $2\le s<\infty$, there holds
\begin{equation}\label{Eq: Claim on coercivity}
\int_{0}^{1}\mleft\lvert \varepsilon^{2}+\mleft\lvert tz_{2}+(1-t)z_{1}\mright\rvert^{2}\mright\rvert^{s/2-1}\,{\mathrm d}t\ge 2^{-s}\mleft(\varepsilon^{2}+\max\{\,\lvert z_{1}\rvert^{2},\,\lvert z_{2}\rvert^{2}\,\}\mright)^{s/2-1}
\end{equation}
for all $z_{1},\,z_{2}\in{\mathbb R}^{n}$.
Without loss of generality, we may let \(\lvert z_{1}\rvert\le \lvert z_{2}\rvert\).
Then, (\ref{Eq: triangle inequality 1}) yields
\[\int_{0}^{1}\mleft\lvert \varepsilon^{2}+\mleft\lvert tz_{2}+(1-t)z_{1}\mright\rvert^{2}\mright\rvert^{s/2-1}\,{\mathrm d}t\ge \frac{\mleft(\varepsilon^{2}+\lvert z_{2}\rvert^{2}\mright)^{s/2-1}}{2^{s-2}}\cdot\frac{1}{4}=2^{-s}\mleft(\varepsilon^{2}+\lvert z_{2}\rvert^{2}\mright)^{s/2-1}.\]

\ref{Item 1/2: Monotonicity}.
By (\ref{Eq: Coerciveness on H-epsilon}) and (\ref{Eq: Claim on coercivity}), we have
\begin{align*}
\langle \nabla H_{\varepsilon}(z_{1})-\nabla H_{\varepsilon}(z_{2})\mid z_{1}-z_{2} \rangle&=\int_{0}^{1}\mleft\langle \nabla^{2} H_{\varepsilon}(tz_{1}+(1-t)z_{2})(z_{1}-z_{2})\mathrel{}\middle|\mathrel{} z_{1}-z_{2}\mright\rangle\,{\mathrm d}t\\&\ge \gamma\lvert z_{1}-z_{2}\rvert^{2}\int_{0}^{1}\mleft(\varepsilon^{2}+\lvert tz_{1}+(1-t)z_{2}\rvert^{2}\mright)^{s/2-1}\,{\mathrm d}t\\&\ge 2^{-s}\lambda\lvert z_{1}-z_{2}\rvert^{2}\mleft(\varepsilon^{2}+\max\mleft\{\,\lvert z_{1}\rvert^{2},\,\lvert z_{2}\rvert^{2}\,\mright\}  \mright)^{s/2-1}.
\end{align*}
From this and $s/2-1>0$, the estimate (\ref{Eq: L2 lower bound lemma case p>2}) immediately follows.

\ref{Item 2/2: Monotonicity molified}. It suffices to find a constant $\gamma_{0}\in(0,\,1)$ such that there holds
\[\mleft\langle \nabla H_{\varepsilon}(z_{1})-\nabla H_{\varepsilon}(z_{2})\mathrel{}\middle| \mathrel{}z_{1}-z_{2} \mright\rangle\ge \gamma_{0}L_{0}\mleft(\varepsilon^{2}+\max\,\mleft\{\,\lvert z_{1}\rvert^{2},\,\lvert z_{2}\rvert^{2} \,\mright\}\mright)^{s/2-1}\lvert z_{1}-z_{2}\rvert^{2}\]
for all $z_{1},\,z_{2}\in{\mathbb R}^{n}$. For fixed $\varepsilon\in(0,\,1)$ and $z_{1},\,z_{2}\in{\mathbb R}^{n}$, we set $w(t)\coloneqq z_{1}+t(z_{2}-z_{1})\in {\mathbb R}^{n}$ for $t\in\lbrack0,\,1\rbrack$. We note that for a.e. $y\in B_{1}(0)$, there holds $w(t)\neq \varepsilon y$ for all $t\in\lbrack 0,\,1\rbrack$. Hence by $H\in C^{2}({\mathbb R}^{n}\setminus\{ 0\})$ and Fubini's theorem, we are able to compute
\begin{align*}
&\mleft\langle\nabla H_{\varepsilon}(z_{1})-\nabla H_{\varepsilon}(z_{2})\mathrel{}\middle|\mathrel{} z_{1}-z_{2}\mright\rangle\\&=\int_{B_{1}(0)}\langle\nabla H(z_{1}-\varepsilon y)-\nabla H(z_{2}-\varepsilon y)\mid z_{1}-z_{2}\rangle j(y)\,{\mathrm{d}}y\\&= \int_{B_{1}(0)}\mleft(\int_{0}^{1}\mleft\langle\nabla^{2}H(z_{2}-\varepsilon y+t(z_{1}-t_{2}))(z_{1}-z_{2})\mathrel{}\middle|\mathrel{} z_{1}-z_{2}\mright\rangle\,{\mathrm d}t \mright)j(y)\,{\mathrm{d}}y\\&=\int_{0}^{1}\int_{B_{1}(0)}\mleft(\int_{0}^{1}\mleft\langle\nabla^{2}H(z_{2}-\varepsilon y+t(z_{1}-t_{2}))(z_{1}-z_{2})\mathrel{}\middle|\mathrel{} z_{1}-z_{2}\mright\rangle\,{\mathrm d}t \mright)j(y)\,{\mathrm{d}}y\\&\ge L_{0}\int_{0}^{1}\mleft(\int_{B_{1}(0)}\lvert w(t)-\varepsilon y\rvert^{2(s/2-1)}j(y)\,{\mathrm{d}}y\mright){\mathrm{d}}t
\end{align*}
For each $t\in\lbrack 0,\,1\rbrack$, we replace the integral domain $B_{1}(0)$ by a smaller one, either
\[C_{t}^{-}\coloneqq \{y\in B_{1}(0)\setminus B_{1/2}(0)\mid \langle w(t)\mid y\rangle\le 0\}\quad \textrm{when }s\in\lbrack 2,\,\infty),\]
or
\[C_{t}^{+}\coloneqq \{y\in B_{1}(0)\setminus B_{1/2}(0)\mid \langle w(t)\mid y\rangle\ge 0\}\quad \textrm{when }s\in(1,\,2).\]
By the definitions, both of the sets $C_{t}^{+}$ and $C_{t}^{-}$ occupy at least one half of the domain $B_{1}(0)\setminus B_{1/2}(0)$. Here we should recall that the function $j$ is spherically symmetric. Hence, by direct computations, it is easy to check that 
\[\int_{0}^{1}\mleft(\int_{B_{1}(0)}\lvert w(t)-\varepsilon y\rvert^{2(s/2-1)}j(y)\,{\mathrm{d}}y\mright){\mathrm{d}}t\ge \gamma_{1}\int_{0}^{1}\mleft\lvert \varepsilon^{2}+\mleft\lvert tz_{2}+(1-t)z_{1}\mright\rvert^{2}\mright\rvert^{s/2-1}\,{\mathrm d}t\]
for some constant $\gamma_{1}=\gamma_{1}(n,\,s)\in(0,\,1)$.
When $s\in(1,\,2)$, we simply use \[\mleft\lvert tz_{2}+(1-t)z_{1}\mright\rvert\le \max\,\mleft\{\,\lvert z_{1}\rvert,\,\lvert z_{2}\rvert\,\mright\}\quad \textrm{for all }t\in\lbrack0,\,1\rbrack.\] 
Thus, we may choose $\gamma_{0}\coloneqq \gamma_{1}$. In the remaining case $s\in\lbrack 2,\,\infty)$, we may take $\gamma_{0}\coloneqq 2^{-s}\gamma_{1}$ by (\ref{Eq: Claim on coercivity}).
\end{proof}
Finally, we conclude Section \ref{Subsect: Some Lipschitz estimates} by deducing continuity estimates of truncating mappings. For $0<\delta<1$, we set
\[{\mathcal G}_{\delta}(z)\coloneqq \mleft(\lvert z\rvert- \delta\mright)_{+} \frac{z}{\lvert z\rvert}\quad \textrm{for }z\in{\mathbb R}^{n}.\]
Similarly, for $0<\varepsilon<\delta<1$, we set
\[{\mathcal G}_{\delta,\,\varepsilon}(z)\coloneqq \mleft(\sqrt{\varepsilon^{2}+\lvert z\rvert^{2}}- \delta\mright)_{+} \frac{z}{\lvert z\rvert}\quad \textrm{for }z\in{\mathbb R}^{n}.\]
Lemma \ref{Lemma: Lipschitz estimate on relaxed vector fields} belows states that the mapping ${\mathcal G}_{2\delta,\,\varepsilon}$ has Lipschitz continuity, uniformly for sufficiently small $\varepsilon>0$. In the limiting case $\varepsilon=0$, Lipschitz continuity of ${\mathcal G}_{\delta}$ is found in \cite[Lemma 2.3]{BDGPdN}. Following the arguments therein, we would like to prove Lemma \ref{Lemma: Lipschitz estimate on relaxed vector fields} by elementary computations. Before the proof, we recall an inequality
\begin{equation}\label{Eq: Local C11 estimate of E1 special case}
\mleft\lvert \frac{z_{1}}{\lvert z_{1}\rvert}-\frac{z_{2}}{\lvert z_{2}\rvert}\mright\rvert\le 2\min\mleft\{\,\lvert z_{1}\rvert^{-1},\,\lvert z_{2}\rvert^{-1}\,\mright\}\lvert z_{1}-z_{2}\rvert\quad \textrm{for all }z_{1},\,z_{2}\in {\mathbb R}^{n}\setminus\{0\},
\end{equation}
which is easy to deduce by the triangle inequality.
\begin{lemma}[Uniform Lipschitz continuity of ${\mathcal G}_{2\delta,\,\varepsilon}$]\label{Lemma: Lipschitz estimate on relaxed vector fields}
Let \(\delta,\,\varepsilon\) satisfy
\[0<\varepsilon<h\delta\quad \textrm{with }h\in(0,\,2).\]
Then, the mapping \({\mathcal G}_{2\delta,\,\varepsilon}\) satisfies
\begin{equation}\label{Eq: Lipschitz estimate on relaxed vector fields}
\mleft\lvert{\mathcal G}_{2\delta,\,\varepsilon}(z_{1})-{\mathcal G}_{2\delta,\,\varepsilon}(z_{2}) \mright\rvert \le c_{\dagger}\lvert z_{1}-z_{2}\rvert
\end{equation}
for all \(z_{1},\,z_{2}\in{\mathbb R}^{n}\) with 
\[c_{\dagger}(h)\coloneqq 1+\frac{8}{\sqrt{4-h^{2}}}.\]
In particular, if \(\delta,\,\varepsilon\) satisfies (\ref{Eq: Range of delta-epsilon}), then (\ref{Eq: Lipschitz estimate on relaxed vector fields}) holds with $c_{\dagger}=1+64/\sqrt{255}$.
\end{lemma}
\begin{proof}
For preliminary, we introduce a function \(\gamma_{\varepsilon}\in C^{\infty}({\mathbb R})\) by
\(\gamma_{\varepsilon}(t)\coloneqq \sqrt{\varepsilon^{2}+t^{2}}\,(t\in{\mathbb R})\) for each fixed \(0<\varepsilon<1\).
By direct calculations, we can easily check that \(\lVert \gamma_{\varepsilon}^{\prime}\rVert_{L^{\infty}({\mathbb R})}\le 1\), and hence it follows that
\begin{equation}\label{Eq: Mean value theorem}
\lvert \gamma_{\varepsilon}(t_{1})-\gamma_{\varepsilon}(t_{2})\rvert\le \lvert t_{1}-t_{2}\rvert\quad \textrm{for all }t_{1},\,t_{2}\in{\mathbb R}
\end{equation}
by the classical mean value theorem.

We prove (\ref{Eq: Lipschitz estimate on relaxed vector fields}) by considering three possible cases. If \(\lvert z_{1}\rvert,\,\lvert z_{2}\rvert\le \sqrt{(2\delta)^{2}-\varepsilon^{2}}\), then (\ref{Eq: Lipschitz estimate on relaxed vector fields}) is clear by \({\mathcal G}_{2\delta,\,\varepsilon}(z_{1})={\mathcal G}_{2\delta,\,\varepsilon}(z_{2})=0\). When \(\lvert z_{1}\rvert,\,\lvert z_{2}\rvert>\sqrt{(2\delta)^{2}-\varepsilon^{2}}\), we use (\ref{Eq: Local C11 estimate of E1 special case}) and (\ref{Eq: Mean value theorem}) to get
\begin{align*}
\mleft\lvert{\mathcal G}_{2\delta,\,\varepsilon}(z_{1})-{\mathcal G}_{2\delta,\,\varepsilon}(z_{2}) \mright\rvert &= \mleft\lvert \mleft(\sqrt{\varepsilon^{2}+\lvert z_{1}\rvert^{2}}-2\delta\mright) \frac{z_{1}}{\lvert z_{1}\rvert}- \mleft(\sqrt{\varepsilon^{2}+\lvert z_{2}\rvert^{2}}-2\delta\mright) \frac{z_{2}}{\lvert z_{2}\rvert} \mright\rvert\\& \le 2\delta \mleft\lvert \frac{z_{1}}{\lvert z_{1}\rvert}- \frac{z_{2}}{\lvert z_{2}\rvert} \mright\rvert+\mleft\lvert \mleft[\gamma_{\varepsilon}(\lvert z_{1}\rvert)-\gamma_{\varepsilon}(\lvert z_{2}\rvert)\mright]\frac{z_{1}}{\lvert z_{1}\rvert}\mright\rvert\\&+ \gamma_{\varepsilon}(\lvert z_{2}\rvert)\mleft\lvert \frac{z_{1}}{\lvert z_{1}\rvert}- \frac{z_{2}}{\lvert z_{2}\rvert} \mright\rvert\\ &\le \mleft(\frac{4\delta}{\lvert z_{2}\rvert}+1+\frac{2\sqrt{\varepsilon^{2}+\lvert z_{2}\rvert^{2}}}{\lvert z_{2}\rvert}\mright)\lvert z_{1}-z_{2}\rvert.
\end{align*}
Combining with \(\lvert z_{2}\rvert>\sqrt{(2\delta)^{2}-\varepsilon^{2}}\ge \sqrt{4-h^{2}}\delta\ge \sqrt{4-h^{2}}(\varepsilon/h)\), we obtain (\ref{Eq: Lipschitz estimate on relaxed vector fields}).
In the remaining case, without loss of generality we may assume that \(\lvert z_{1}\rvert>\sqrt{(2\delta)^{2}-\varepsilon^{2}}\ge \lvert z_{2}\rvert\).
Then, \(\gamma_{\varepsilon}(\lvert z_{2}\rvert)\le 2\delta\) is clear.
Hence, (\ref{Eq: Mean value theorem}) implies
\begin{align*}
\mleft\lvert{\mathcal G}_{2\delta,\,\varepsilon}(z_{1})-{\mathcal G}_{2\delta,\,\varepsilon}(z_{2}) \mright\rvert &=\sqrt{\varepsilon^{2}+\lvert z_{1}\rvert^{2}}-2\delta\\& \le \gamma_{\varepsilon}(\lvert z_{1}\rvert)-\gamma_{\varepsilon}(\lvert z_{2}\rvert)\\&\le \lvert\lvert z_{1}\rvert-\lvert z_{2}\rvert \rvert\le \lvert z_{1}-z_{2}\rvert\\&\le c_{\dagger}\lvert z_{1}-z_{2}\rvert,
\end{align*}
which completes the proof.
\end{proof}
\subsection{Basic facts of positively one-homogeneous convex functions}\label{Subsect: Basic facts on convex integrands}
In Section \ref{Subsect: Basic facts on convex integrands}, we mention some basic properties of the positively one-homogeneous convex function \(E_{1}\) without proofs.
All of the proofs are already given in \cite[Section 1.3]{MR2033382}, \cite[\S 13]{MR1451876}, \cite[Sections 6.1 \& A.3]{giga2021continuity}.
Also, we briefly give some estimates that immediately follows from (\ref{Eq: positive hom of deg 1}) and Lemma \ref{Lemma: Local C11 estimate}.

We define a closed convex set \(C_{E_{1}}\subset{\mathbb R}^{n}\) by
\[C_{E_{1}}\coloneqq \mleft\{ z\in{\mathbb R}^{n}\mathrel{}\middle| \mathrel{}E_{1}(z)\le 1\mright\},\]
and a function \({\tilde E_{1}}\colon{\mathbb R}^{n}\rightarrow [0,\,\infty]\) by
\[{\tilde E_{1}}(w)\coloneqq \sup\mleft\{ \langle w\mid z \rangle \mathrel{}\middle| \mathrel{} z\in C_{E_{1}} \mright\}\quad\textrm{for }w\in{\mathbb R}^{n},\]
often called the support function of the convex set \(C_{E_{1}}\). The inequality \({\tilde E_{1}}\ge 0\) is clear by the inclusion \(0\in C_{E_{1}}\).
By (\ref{Eq: positive hom of deg 1}) and the definition of \({\tilde E_{1}}\), it is easy to check  the Cauchy--Schwarz-type inequality:
\begin{equation}\label{Eq: Cauchy-Schwarz-type-ineq}
\langle z\mid w \rangle\le E_{1}(z){\tilde E_{1}}(w)\quad \textrm{for all }z\in{\mathbb R}^{n},
\end{equation}
provided $w\in{\mathbb R}^{n}$ satisfies \({\tilde E_{1}}(w)<\infty\).
Also, under the condition (\ref{Eq: positive hom of deg 1}), it is well-known that \(\partial E_{1}\) is given by
\begin{equation}\label{Eq: Representation of subdifferental set}
\partial E_{1}(z)=\mleft\{w\in{\mathbb R}^{n}\mathrel{}\middle| \mathrel{} {\tilde E_{1}}(w)\le 1\textrm{ and }\langle z\mid w\rangle=E_{1}(z)\mright\}\quad \textrm{for all }z\in{\mathbb R}^{n}.
\end{equation}
In particular,
\(\partial E_{1}(0)=\mleft\{w\in{\mathbb R}^{n}\mathrel{}\middle| \mathrel{} {\tilde E_{1}}(w)\le 1\mright\}\) holds. 
The identity
\begin{equation}\label{Eq: Euler's identity}
\langle z\mid w\rangle=E_{1}(z)\quad \textrm{for all }z\in{\mathbb R}^{n},\,w\in\partial E_{1}(z)
\end{equation}
is often called Euler's identity.

We briefly give some estimates related to derivatives of \(E_{1}\) outside the origin.

First, when \(E_{1}\in C^{1}({\mathbb R}^{n}\setminus\{ 0\})\), it is easy to check
\begin{equation}\label{Eq: gradient equality for one-hom}
\nabla E_{1}(\lambda z)=\nabla E_{1} (z)\quad \textrm{for all }z\in{\mathbb R}^{n}\setminus\{ 0\},\,\lambda>0
\end{equation}
by (\ref{Eq: positive hom of deg 1}).
In particular, we have \(\nabla E_{1}\in L^{\infty}({\mathbb R}^{n})\) and
\begin{equation}\label{Eq: Lipschitz bound on E1}
\lVert \nabla E_{1}\rVert_{L^{\infty}({\mathbb R}^{n})}=K_{1}\coloneqq \max\,\mleft\{\lvert \nabla E_{1}(z) \rvert \mathrel{}\middle| \mathrel{} z\in{\mathbb R}^{n},\,\lvert z\rvert=1\mright\}\in\lbrack 0,\,\infty).
\end{equation}
Moreover, from Euler's identity (\ref{Eq: Euler's identity}), it follows that
\begin{equation}\label{Eq: Inclusion on Subdifferential}
\partial E_{1}(z)\subset \overline{B_{K_{1}}(0)}\quad \textrm{for all }z\in{\mathbb R}^{n}.
\end{equation}
Also, by $E_{1}\in C^{1}({\mathbb R}^{n}\setminus\{ 0\})$, there holds $\partial E_{1}(z)=\{\nabla E_{1}(z)\}$ for every $z\in{\mathbb R}^{n}\setminus\{0\}$.

Secondly, when $E_{1}$ satisfies \(E_{1}\in C^{2}({\mathbb R}^{n}\setminus\{ 0\})\), from (\ref{Eq: positive hom of deg 1}) it follows that
\begin{equation}\label{Eq: Hess of E-1}
\nabla^{2} E_{1}(\lambda z)=\lambda^{-1}\nabla^{2} E_{1} (z)\quad \textrm{for all }z_{0}\in{\mathbb R}^{n}\setminus\{ 0\},\,\lambda>0.
\end{equation}
Therefore, the constant
\[K_{2}\coloneqq \max\,\mleft\{ \mleft\lVert\nabla^{2}E_{1}(z)\mright\rVert \mathrel{}\middle|\mathrel{}z\in{\mathbb R}^{n},\,\lvert z\rvert=1\mright\}\in\lbrack 0,\,\infty),\]
satisfies
\begin{equation}\label{Eq: Bound of Hessian of E1}
O\leqslant \nabla^{2} E_{1}(z) \leqslant \frac{K_{2}}{\lvert z\rvert}\mathrm{id}\quad \textrm{for all }z\in{\mathbb R}^{n}\setminus\{ 0\}.
\end{equation}

Finally, it is noted that when $E_{1}\in C_{\mathrm{loc}}^{2,\,\beta_{0}}({\mathbb R}^{n}\setminus\{0\})$ with $\beta_{0}\in(0,\,1\rbrack$, we are able to define a constant
\[K_{2,\,\beta_{0}}\coloneqq\mleft\{\frac{\mleft\lVert \nabla^{2} E_{1}(z_{1})-\nabla^{2}E_{1}(z_{2}) \mright\rVert}{\lvert z_{1}-z_{2}\rvert^{\beta_{0}}}\mathrel{}\middle|\mathrel{} \frac{1}{16}\le\lvert z_{1}\rvert\le 3\textrm{ and }0<\lvert z_{1}-z_{2}\rvert\le \frac{1}{32}\mright\}\in\lbrack 0,\,\infty).\]
Then, by (\ref{Eq: Hess of E-1}), it is easy to check that 
\begin{equation}
\mleft\lVert \nabla^{2}E_{1}(z_{1})-\nabla^{2}E_{1}(z_{2}) \mright\rVert\le K_{2,\,\beta_{0}}\mu^{-1-\beta_{0}}\lvert z_{1}-z_{2} \rvert^{\beta_{0}}.
\end{equation}
for all $\mu\in(0,\,\infty)$ and $z_{1},\,z_{2}\in{\mathbb R}^{n}\setminus\{ 0\}$ enjoying (\ref{Eq: C-2-beta variable condition}).


\subsection{Approximation based on mollifiers}\label{Subsect: approximation of density}
Section \ref{Subsect: approximation of density} is focused on the approximation of \(E=E_{1}+E_{p}\), which is based on the convolution by the Friedrichs mollifier. We would like to check that the relaxed density $E_{\varepsilon}=j_{\varepsilon}\ast E$ satisfies some quantitative estimates.

For the density \(E_{p}\in C^{1}({\mathbb R}^{n})\cap C^{2}({\mathbb R}^{n}\setminus\{ 0\})\) satisfying (\ref{Eq: Gradient bound condition for Ep})--(\ref{Eq: Hessian condition for Ep}), we define an approximated function \(E_{p,\,\varepsilon}\in C^{\infty}({\mathbb R}^{n})\) by
\begin{equation}\label{Eq: General approximation scheme of density and operator p}
E_{p,\,\varepsilon}\coloneqq j_{\varepsilon}\ast E_{p},\quad\textrm{so that}\quad \nabla E_{p,\,\varepsilon}=j_{\varepsilon}\ast\nabla E_{p}
\end{equation}
for each \(\varepsilon\in(0,\,1)\). By applying Lemmata \ref{Lemma: Mollified estimates}--\ref{Lemma: Coerciveness keeps under convolution} (see also \cite[Lemma 2]{MR1634641}, \cite[Lemma 2.4]{MR1612389}), we can check that \(E_{p,\,\varepsilon}\) defined by (\ref{Eq: General approximation scheme of density and operator p}) satisfies 
\begin{equation}\label{Eq: Gradient bound for nabla Ep}
\lvert \nabla E_{p,\,\varepsilon}(z_{0})\rvert\le \Lambda_{0}^{\prime} \mleft(\varepsilon^{2}+\lvert z_{0}\rvert^{2}\mright)^{(p-1)/2},
\end{equation}
\begin{equation}\label{Eq: Hessian estimate for Ep}
\lambda_{0}^{\prime}\mleft(\varepsilon^{2}+\lvert z_{0}\rvert^{2}\mright)^{p/2-1}\mathrm{id}\leqslant \nabla^{2}E_{p,\,\varepsilon}(z_{0})\leqslant \Lambda_{0}^{\prime}\mleft(\varepsilon^{2}+\lvert z_{0}\rvert^{2}\mright)^{p/2-1}\mathrm{id}
\end{equation}
for all \(z_{0}\in{\mathbb R}^{n}\) and \(\varepsilon\in(0,\,1)\). Here the constants \(0<\lambda_{0}^{\prime}\le \Lambda_{0}^{\prime}<\infty\) may depend on \(n,\,p,\,\lambda_{0},\,\Lambda_{0}\), but are independent of \(\varepsilon\in(0,\,1)\).
From (\ref{Eq: Hessian estimate for Ep}), it is possible to get growth and monotonicity estimates for $\nabla E_{p,\,\varepsilon}$ (see \cite{MR1634641}, \cite[Lemma 3]{MR4201656} for detailed computations). 
For growth estimates, for all \(z_{1},\,z_{2}\in{\mathbb R}^{n}\) and \(\varepsilon\in(0,\,1)\), we have
\begin{equation}\label{Eq: Local Hoelder convergence estimate on Ep with p>2}
\lvert \nabla E_{p,\,\varepsilon}(z_{1})-\nabla E_{p,\,\varepsilon}(z_{2})\rvert\le \Lambda_{0}^{\prime}C(p)\mleft(\varepsilon^{p-2}+\lvert z_{1}\rvert^{p-2}+\lvert z_{2}\rvert^{p-2}\mright)\lvert z_{1}-z_{2}\rvert
\end{equation}
provided \(2\le p<\infty\), and
\begin{equation}\label{Eq: Local Hoelder convergence estimate on Ep with p<2}
\lvert \nabla E_{p,\,\varepsilon}(z_{1})-\nabla E_{p,\,\varepsilon}(z_{2})\rvert\le \Lambda_{0}^{\prime} C(p)\lvert z_{1}-z_{2}\rvert^{p-1}
\end{equation}
provided \(1<p<2\).
Hence, from (\ref{Eq: Gradient bound for nabla Ep}) and (\ref{Eq: Local Hoelder convergence estimate on Ep with p>2})--(\ref{Eq: Local Hoelder convergence estimate on Ep with p<2}), it follows that
\begin{equation}\label{Eq: compact convergence of nabla Ep}
\nabla E_{p,\,\varepsilon}(z_{0})\rightarrow \nabla E_{p}(z_{0})\quad\textrm{as $\varepsilon\to 0$, locally uniformly for }z_{0}\in{\mathbb R}^{n}.
\end{equation}
For monotonicity estimates, we can find a constant \(C=C(p)\in(0,\,\infty)\) such that
\begin{equation}\label{Eq: Lp lower bound lemma case p>2}
\mleft\langle \nabla E_{p,\,\varepsilon}(z_{1})-\nabla E_{p,\,\varepsilon}(z_{2})\mid z_{1}-z_{2} \mright\rangle\ge \lambda C(p)\lvert z_{1}-z_{2}\rvert^{p}
\end{equation}
holds for all \(z_{1},\,z_{2}\in{\mathbb R}^{n},\,\varepsilon\in(0,\,1)\) when $2\le p<\infty$. When \(1<p<2\), we have 
\begin{equation}\label{Eq: L2 lower bound lemma case p<2}
\mleft\langle \nabla E_{p,\,\varepsilon}(z_{1})-\nabla E_{p,\,\varepsilon}(z_{2})\mid z_{1}-z_{2} \mright\rangle\ge \lambda(\varepsilon^{2}+\lvert z_{1}\rvert^{2}+\lvert z_{2}\rvert^{2})^{p/2-1} \lvert z_{1}-z_{2}\rvert^{2}
\end{equation}
for all \(z_{1},\,z_{2}\in{\mathbb R}^{n},\,\varepsilon\in(0,\,1)\).

For the positively one-homogeneous convex function \(E_{1}\in C({\mathbb R}) \cap C^{2}({\mathbb R}^{n}\setminus\{0\})\) and for each fixed \(\varepsilon\in(0,\,1)\), we aim to set a smooth convex function \(E_{1,\,\varepsilon}\in C^{\infty}({\mathbb R}^{n})\) such that 
\begin{equation}\label{Eq: approximated vector field bound in L-infty, E11}
\lvert \nabla E_{1,\,\varepsilon}(z_{0})\rvert\le K_{1}\quad\textrm{for all }z_{0}\in{\mathbb R}^{n},\,\varepsilon\in(0,\,1);
\end{equation}
\begin{equation}\label{Eq: smooth subgradient at 0}
\nabla E_{1,\,\varepsilon}(0) \textrm{ is independent of }\varepsilon\in(0,\,1),\textrm{ written by $c_{0}\in{\mathbb R}^{n}$;}
\end{equation}
\begin{equation}\label{Eq: Jensen-type inequality}
{\tilde E_{1}}\mleft(\nabla E_{1,\,\varepsilon}(z_{0})\mright)\le 1\quad \textrm{for all $z_{0}\in{\mathbb R}^{n},\,\varepsilon\in(0,\,1)$; and}
\end{equation}
\begin{equation}\label{Eq: Hessian bound of approximated E1}
O\leqslant \nabla^{2}E_{1,\,\varepsilon}(z_{0})\leqslant \frac{K_{2}^{\prime}}{\sqrt{\varepsilon^{2}+\lvert z_{0}\rvert^{2}}}\mathrm{id}\quad \textrm{for all $z_{0}\in{\mathbb R}^{n},\,\varepsilon\in(0,\,1)$.}
\end{equation}
Here the constant \(K_{2}^{\prime}\in(0,\,\infty)\) in (\ref{Eq: Hessian bound of approximated E1}) is independent of \(\varepsilon\in(0,\,1)\). 
This is possible by constructing a regularized function \(E_{1,\,\varepsilon}\in C^{\infty}({\mathbb R}^{n})\) by
\begin{equation}\label{Eq: general approximation scheme of density and operator 1}
E_{1,\,\varepsilon}\coloneqq j_{\varepsilon}\ast E_{1},\quad\textrm{so that}\quad \nabla E_{1,\,\varepsilon}=j_{\varepsilon}\ast\nabla E_{1}
\end{equation}
for each \(\varepsilon\in(0,\,1)\). 
In Lemma \ref{Lemma: mollifier approximation}, we check that this \(E_{1,\,\varepsilon}\) surely satisfies the desired properties (\ref{Eq: approximated vector field bound in L-infty, E11})--(\ref{Eq: Hessian bound of approximated E1}). 
\begin{lemma}\label{Lemma: mollifier approximation} Let \(E_{1}\in C({\mathbb R}^{n})\cap C^{1}({\mathbb R}^{n}\setminus\{ 0\})\) be a positively one-homogeneous convex function. Then, for each \(\varepsilon\in(0,\,1)\), the relaxed function \(E_{1,\,\varepsilon}\) defined by (\ref{Eq: general approximation scheme of density and operator 1}) satisfies (\ref{Eq: approximated vector field bound in L-infty, E11})--(\ref{Eq: Jensen-type inequality}). Moreover, if $E_{1}$ is in $E_{1}\in C^{2}({\mathbb R}^{n}\setminus\{ 0\})$, then there exists a constant \(K_{2}^{\prime}=K_{2}^{\prime}(K_{1},\,K_{2},\,n)\) such that (\ref{Eq: Hessian bound of approximated E1}) holds.
\end{lemma}
\begin{proof}
It is easy to get
\[\lVert\nabla (E_{1}\ast j_{\varepsilon})\rVert_{L^{\infty}({\mathbb R}^{n})}\le \lVert\nabla E_{1}\rVert_{L^{\infty}({\mathbb R}^{n})}\lVert j_{\varepsilon}\rVert_{L^{1}({\mathbb R}^{n})}\le K_{1},\]
which yields (\ref{Eq: approximated vector field bound in L-infty, E11}).

By change of variables and (\ref{Eq: gradient equality for one-hom}), we can check that
\begin{align*}
(j_{\varepsilon}\ast \nabla E_{1})(0)&=\int_{{\mathbb R}^{n}}\varepsilon^{-n}j(-y/\varepsilon)\nabla E_{1}(y)\,{\mathrm d}y\\& =\int_{{\mathbb R}^{n}}j(-x)\nabla E_{1}(\varepsilon x)\,{\mathrm d}x\\&=\int_{{\mathbb R}^{n}}j(-x)\nabla E_{1}(x)\,{\mathrm d}x\\&=(j\ast\nabla E_{1})(0)\eqqcolon c_{0}
\end{align*}
for all \(\varepsilon\in(0,\,1)\). Clearly, this \(c_{0}\in{\mathbb R}^{n}\) is independent of \(\varepsilon\in(0,\,1)\), which yields (\ref{Eq: smooth subgradient at 0}).

We take arbitrary \(w\in C_{E_{1}}\) and \(z_{0}\in{\mathbb R}^{n}\).
We recall that there holds \(\partial E_{1}(z_{0}-y)=\{\nabla E_{1}(z_{0}-y)\}\) for all \(y\in{\mathbb R}^{n}\setminus\{ z_{0}\}\) by \(E_{1}\in C^{1}({\mathbb R}^{n}\setminus\{ 0\})\), and hence we obtain \({\tilde E_{1}}(\nabla E_{1}(z_{0}-y))\le 1<\infty\). 
Combining with (\ref{Eq: Cauchy-Schwarz-type-ineq}), we compute
\begin{align*}
\mleft\langle \nabla(j_{\varepsilon}\ast E_{1})(z_{0})\mathrel{}\middle| \mathrel{}w \mright\rangle&=\mleft\langle (j_{\varepsilon}\ast\nabla E_{1})(z_{0})\mathrel{}\middle| \mathrel{}w \mright\rangle\\&=\mleft\langle \int_{{\mathbb R}^{n}}j_{\varepsilon}(y)\nabla E_{1}(z_{0}-y)\,{\mathrm d}y\mathrel{}\middle| \mathrel{}w \mright\rangle\\&=\int_{{\mathbb R}^{n}\setminus\{ z_{0}\}}\mleft\langle \nabla E_{1}(z_{0}-y)\mathrel{}\middle| \mathrel{}w \mright\rangle\,\cdot\,j_{\varepsilon}(y)\,{\mathrm d}y\\&\le E_{1}(w)\int_{{\mathbb R}^{n}\setminus\{ z_{0}\}}j_{\varepsilon}(y)\,{\mathrm d}y\le 1.  
\end{align*} 
Since \(w\in C_{E_{1}}\) is arbitrary, this completes the proof of (\ref{Eq: Jensen-type inequality}).

We let $E_{1}\in C^{2}({\mathbb R}^{n}\setminus\{0\})$, so that (\ref{Eq: Bound of Hessian of E1}) holds. Let \(z_{0}\in {\mathbb R}^{n}\), and \(\lambda\) be an arbitrary eigenvalue of \(\nabla^{2} E_{1,\,\varepsilon}(z_{0})\). Then, \(\lambda\ge 0\) is clear by convexity of \(E_{1,\,\varepsilon}\in C^{\infty}({\mathbb R}^{n})\).
Moreover, by (\ref{Eq: Lipschitz bound on E1})--(\ref{Eq: Bound of Hessian of E1}), we are able to apply Lemma \ref{Lemma: Mollified estimates} to obtain
\[\lambda=\lvert \lambda\rvert \le \mleft\lVert \nabla^{2}E_{1,\,\varepsilon}(z_{0}) \mright\rVert \le K_{2}^{\prime}\mleft(\varepsilon^{2}+\lvert z_{0}\rvert^{2}\mright)^{-1/2}\]
with \(K_{2}^{\prime}\) depending at most on \(K_{1},\,K_{2},\,n\).
This result yields (\ref{Eq: Hessian bound of approximated E1}).
\end{proof}
It should be noted that (\ref{Eq: Hessian bound of approximated E1}) clearly yields
\[\mleft\lVert \nabla^{2}E_{1,\,\varepsilon}(z_{0})\mright\rVert\le\frac{K_{2}^{\prime}}{\lvert z_{0}\rvert}\quad \textrm{for all \(\varepsilon\in(0,\,1)\) and \(z_{0}\in{\mathbb R}^{n}\setminus\{0\}\)}.\]
By this and (\ref{Eq: approximated vector field bound in L-infty, E11}), we are able to apply Lemma \ref{Lemma: Local C11 estimate} with $s=1,\,\sigma_{1}=\sigma_{2}=0$. Then, there exists a constant \(C=C(K_{1},\,K_{2}^{\prime})\) such that
\begin{equation}\label{Eq: Local Lipschitz estimate on A-1-epsilon}
\lvert \nabla E_{1,\,\varepsilon}(z_{1})-\nabla E_{1,\,\varepsilon}(z_{2})\rvert\le C\min\mleft\{\,\lvert z_{1}\rvert^{-1},\,\lvert z_{2}\rvert^{-1} \,\mright\}\lvert z_{1}-z_{2}\rvert
\end{equation}
for all \(z_{1},\,z_{2}\in{\mathbb R}^{n}\setminus\{0\}\) and \(\varepsilon\in(0,\,1)\).
Hence from (\ref{Eq: approximated vector field bound in L-infty, E11}) and (\ref{Eq: Local Lipschitz estimate on A-1-epsilon}), it follows that
\begin{equation}\label{Eq: uniformly local convergence}
\nabla E_{1,\,\varepsilon}(z_{0})\rightarrow \nabla E_{1}(z_{0})\quad\textrm{as $\varepsilon\to 0$, locally uniformly for }z_{0}\in{\mathbb R}^{n}\setminus\{ 0\}.
\end{equation}
By $E_{1,\,\varepsilon}\in C^{1}({\mathbb R}^{n})$, it should be noted that the inequality (\ref{Eq: Local Lipschitz estimate on A-1-epsilon}) is valid as long as $(z_{1},\,z_{2})\neq (0,\,0)$. Also, by convexity of $E_{1,\,\varepsilon}\in C^{1}({\mathbb R}^{n})$, it is easy to get a monotonicity estimate
\begin{equation}\label{Eq: non-quantitative monotonicity}
\mleft\langle \nabla E_{1,\,\varepsilon}(z_{1})-\nabla E_{1,\,\varepsilon}(z_{2})\mid z_{1}-z_{2} \mright\rangle\ge 0
\end{equation}for all $z_{1},\,z_{2}\in{\mathbb R}^{n}$. 

Hereinafter we set constants \(0<\lambda\le \Lambda<\infty,\,0<K<\infty\) by
\[\lambda\coloneqq \min\mleft\{\,\lambda_{0},\,\lambda_{0}^{\prime}\,\mright\},\quad \Lambda\coloneqq \max\mleft\{\,\Lambda_{0},\,\Lambda_{0}^{\prime}\,\mright\},\quad K\coloneqq\max \mleft\{\,K_{1},\,K_{2},\,K_{2}^{\prime},\,K_{2,\,\beta_{0}}\,\mright\}\]
for notational simplicity.
Finally, for an approximation parameter \(\varepsilon\in(0,\,1)\), the regularized operator ${\mathcal L}^{\varepsilon}$ is given by
\begin{equation}\label{Eq: Regularized equation}
{\mathcal L}^{\varepsilon}u_{\varepsilon}\coloneqq -\mathrm{div}\,(\nabla E_{1,\,\varepsilon}(\nabla u_{\varepsilon})+\nabla E_{p,\,\varepsilon}(\nabla u_{\varepsilon}))=-\mathrm{div}\,(\nabla E_{\varepsilon}(\nabla u_{\varepsilon})),
\end{equation}
with \[E_{\varepsilon}\coloneqq E_{1,\,\varepsilon}+E_{p,\,\varepsilon}\in C^{\infty}({\mathbb R}^{n}).\]
By (\ref{Eq: compact convergence of nabla Ep}), (\ref{Eq: smooth subgradient at 0}), and (\ref{Eq: uniformly local convergence}), we can easily check that
\begin{equation}\label{Eq: A-0 pointwise convergence to dummy vector field}
\nabla E_{\varepsilon}(z_{0})\rightarrow A_{0}(z_{0})\coloneqq \mleft\{\begin{array}{cc}\nabla E_{1}(z_{0})+\nabla E_{p}(z_{0}) & (z_{0}\not= 0),\\ c_{0}& (z_{0}=0),\end{array}\mright.\quad \textrm{as }\varepsilon\to 0
\end{equation}
for each fixed \(z_{0}\in{\mathbb R}^{n}\).
The limit vector field \(A_{0}\colon{\mathbb R}^{n}\rightarrow {\mathbb R}^{n}\) is in general not continuous at the origin, but Borel measurable in ${\mathbb R}^{n}$. In particular, for every \(v\in L^{p}(\Omega;\,{\mathbb R}^{n})\), we can define a Lebesgue measurable vector field \(A_{0}(v)\), which is in \(L^{p^{\prime}}(\Omega;\,{\mathbb R}^{n})\) by (\ref{Eq: Gradient bound condition for Ep}) and (\ref{Eq: Lipschitz bound on E1}).

By (\ref{Eq: Hessian estimate for Ep}) and (\ref{Eq: Hessian bound of approximated E1}), it is clear that \(E_{\varepsilon}\) satisfies
\begin{equation}\label{Eq: Hessian estimate for approximated E}
\lambda\mleft(\varepsilon^{2}+\lvert z_{0}\rvert^{2}\mright)^{p/2-1}\mathrm{id}\leqslant \nabla^{2}E_{\varepsilon}(z_{0})\leqslant\mleft[\Lambda \mleft(\varepsilon^{2}+\lvert z_{0}\rvert^{2}\mright)^{p/2-1}+ \frac{K}{\sqrt{\varepsilon^{2}+\lvert z_{0}\rvert^{2}}}\mright]\mathrm{id}
\end{equation}
for all $z_{0}\in{\mathbb R}^{n},\,\varepsilon\in(0,\,1)$. In particular, the ellipticity ratio of $\nabla^{2}E_{\varepsilon}(z_{0})$ can be measured by $\sqrt{\varepsilon^{2}+\lvert z_{0}\rvert^{2}}$.

From the assumptions \(E_{1},\,E_{p}\in C_{\mathrm{loc}}^{2,\,\beta_{0}}({\mathbb R}^{n}\setminus\{ 0\})\) and Lemma \ref{Lemma: Mollified estimates}, we would like to deduce an error estimate on the Hessian matrix \(\nabla^{2} E_{\varepsilon}\). The following Lemma \ref{Lemma: Error estimate on Hess E-epsilon} will be applied in our freezing coefficient argument. 
\begin{lemma}\label{Lemma: Error estimate on Hess E-epsilon}
Let positive numbers $\delta,\,\varepsilon$ satisfy (\ref{Eq: Range of delta-epsilon}). Assume that the convex functions $E_{1}\in C^{0}({\mathbb R}^{n})\cap C_{\mathrm{loc}}^{2,\,\beta_{0}}({\mathbb R}^{n}\setminus\{ 0\})$ and $E_{p}\in C^{1}({\mathbb R}^{n})\cap C_{\mathrm{loc}}^{2,\,\beta_{0}}({\mathbb R}^{n}\setminus\{ 0\})$ satisfy (\ref{Eq: positive hom of deg 1})--(\ref{Eq: C-2-beta growth for Ep}). Then, the relaxed function $E_{\varepsilon}=E_{1,\,\varepsilon}+E_{p,\,\varepsilon}$, defined by (\ref{Eq: General approximation scheme of density and operator p}) and (\ref{Eq: general approximation scheme of density and operator 1}) for each $\varepsilon\in(0,\,1)$, satisfies the following:
\begin{enumerate}
\item \label{Item 1/2: Monotonicity and Growth estimate} Let $M\in(\delta,\,\infty)$ be a fixed constant. Then, for all $z_{1},\,z_{2}\in{\mathbb R}^{n}$ satisfying
\[\lvert z_{1}\rvert\le M\quad \textrm{and}\quad \delta\le \lvert z_{2}\rvert\le M,\]
we have
\begin{equation}\label{Eq: Monotonicity outside}
\mleft\langle \nabla E_{\varepsilon}(z_{1})-\nabla E_{\varepsilon}(z_{2}) \mathrel{}\middle|\mathrel{}z_{1}-z_{2} \mright\rangle\ge C_{1}\lvert z_{1}-z_{2}\rvert^{2},
\end{equation}
and
\begin{equation}\label{Eq: Growth outside}
\mleft\lvert \nabla E_{\varepsilon}(z_{1})-\nabla E_{\varepsilon}(z_{2}) \mright\rvert\le C_{2}\lvert z_{1}-z_{2}\rvert.
\end{equation}
Here the constants $C_{1},\,C_{2}\in(0,\,\infty)$ depend at most on $p$, $\lambda$, $\Lambda$, $K$, $\delta$, and $M$.
\item \label{Item 2/2: Hessian errors on E-epsilon} For all $\mu\in(\delta,\,\infty)$, and \(z_{1},\,z_{2}\in{\mathbb R}^{n}\) satisfying (\ref{Eq: Outside condition on variable}), we have
\begin{equation}\label{Eq: Error of Hess}
\mleft\lvert \nabla^{2} E_{\varepsilon}(z_{1})(z_{2}-z_{1})-(\nabla E_{\varepsilon}(z_{2})-\nabla E_{\varepsilon}(z_{1})) \mright\rvert\le C\mu^{p-2-\beta_{0}}\lvert z_{1}-z_{2}\rvert^{1+\beta_{0}},
\end{equation}
Here the constant \(C\in(0,\,\infty)\) depends at most on $n$, $p$, $\beta_{0}$, $\lambda$, $\Lambda$, $K$ and $\delta$.
\end{enumerate}
\end{lemma}
\begin{proof}
\ref{Item 1/2: Monotonicity and Growth estimate}. It is easy to check
\[\varepsilon^{l}+\lvert z_{1}\rvert^{l}+\lvert z_{2}\rvert^{l}\le 3M^{l}\quad \textrm{for every }l\in\lbrack 0,\,\infty),\]
and
\[\mleft\{\begin{array}{cccc}\lvert z_{2}\rvert^{p-2}&\le&\delta^{p-2}&\textrm{for }1< p\le 2,\\
\lvert z_{2}\rvert^{p-2}&\ge& \delta^{p-2}& \textrm{for }2\le p<\infty,\end{array}\mright.\]
by our settings of $z_{1}$ and $z_{2}$. 
With these results in mind, we are able to deduce (\ref{Eq: Monotonicity outside})--(\ref{Eq: Growth outside}) by applying (\ref{Eq: L2 upper bound lemma case p<2}) or (\ref{Eq: L2 lower bound lemma case p>2}) with $H_{\varepsilon}=E_{s,\,\varepsilon}\,(1\le s<\infty)$, or simply using (\ref{Eq: L2 lower bound lemma case p<2}) and (\ref{Eq: non-quantitative monotonicity}).

\ref{Item 2/2: Hessian errors on E-epsilon}.
We apply Lemma \ref{Lemma: Mollified estimates} \ref{Item 3/3: Hessian errors} to the functions \(E_{1,\,\varepsilon}\) and \(E_{p,\,\varepsilon}\). It should be mentioned that the condition \(0<\delta<\mu\) clearly yields \(\mu^{-1-\beta_{0}}=\mu^{1-p}\cdot \mu^{p-2-\beta_{0}}\le \delta^{1-p}\mu^{p-2-\beta_{0}}\), from which (\ref{Eq: Error of Hess}) easily follows.
\end{proof}

\subsection{An alternative approximation scheme}\label{Rmk: Uhlenbeck structure case} 
When the densities \(E_{1}\) and \(E_{p}\) have so called the Uhlenbeck structure, it is possible to give relaxed densities $E_{1,\,\varepsilon}$ and $E_{p,\,\varepsilon}$ without convolution by standard mollifiers. For this alternative approximation scheme, we are able to prove rather easily that the properties (\ref{Eq: Gradient bound for nabla Ep})--(\ref{Eq: Hessian estimate for Ep}), (\ref{Eq: approximated vector field bound in L-infty, E11})--(\ref{Eq: Hessian bound of approximated E1}) and (\ref{Eq: Error of Hess}), which are shown in Section \ref{Subsect: approximation of density}, are valid. In Section \ref{Rmk: Uhlenbeck structure case}, we would like to explain briefly how to justify these desired properties. For simplicity, we let the densities $E_{1}$ and $E_{p}$ be in $C^{3}$ outside the origin.

We consider the special case where the function \(E_{p}\) is of the form
\[E_{p}(z)=\frac{1}{2}g_{p}(\lvert z\rvert^{2})\quad \textrm{and therefore}\quad \nabla E_{p}(z)=g_{p}^{\prime}(\lvert z\rvert^{2})z\quad \textrm{for }z\in{\mathbb R}^{n}\setminus\{0\}.\]
Here \(g_{p}\colon [0,\,\infty)\rightarrow [0,\,\infty)\) is a continuous function satisfying \(g_{p}\in C^{3}((0,\,\infty))\), and we let this $g_{p}$ admit positive constants \(\gamma_{p},\,\Gamma_{p}\in(0,\,\infty)\) such that
\begin{equation}\label{Eq: Condition of gp}
\mleft\{\begin{array}{rcl}
\mleft\lvert g_{p}^{\prime}(\sigma)\mright\rvert &\le & \Gamma_{p}\sigma^{p/2-1},\\ \mleft\lvert g_{p}^{\prime\prime}(\sigma) \mright\rvert & \le & \Gamma_{p}\sigma^{p/2-2},\\ \gamma_{p}(\tau+\sigma)^{p/2-1} & \le & g_{p}^{\prime}(\tau+\sigma)+2\sigma\min\mleft\{\,0,\,g_{p}^{\prime\prime}(\tau+\sigma)\,\mright\},\\
\mleft\lvert g^{\prime\prime\prime}(\sigma)\mright\rvert  & \le & \Gamma_{p}\sigma^{p/2-3},\end{array} \mright. 
\end{equation}
for all $\sigma\in(0,\,\infty),\,\tau\in(0,\,1)$.
Then, for each \(z=(z_{1},\,\dots,\,z_{n})\in {\mathbb R}^{n}\setminus\{ 0\}\), we compute
\[\mleft\{\begin{array}{rcl}\nabla^{2}E_{p}(z)&=&g_{p}^{\prime}(\lvert z\rvert^{2})\mathrm{id}+2g_{p}^{\prime\prime}(\lvert z\rvert^{2})z\otimes z,\\ \nabla^{3}E_{p}(z)&=&4g_{p}^{\prime\prime\prime}(\lvert z\rvert^{2})(z_{i}z_{j}z_{k})_{i,\,j,\,k}+2g_{p}^{\prime\prime}(\lvert z\rvert^{2})(z_{j}\delta_{ki}+z_{k}\delta_{ij})_{i,\,j,\,k},\end{array}\mright.\]
where $\delta_{ij}$ denotes Kronecker's delta.
From this result, we can check that \(E_{p}=g_{p}(\lvert z\rvert^{2})/2\) satisfies (\ref{Eq: Gradient bound condition for Ep})--(\ref{Eq: Hessian condition for Ep}) with \(0<\lambda_{0}=\gamma_{p}\le \Lambda_{0}=\Lambda_{0}(\Gamma_{p})<\infty\). Moreover, we have
\begin{equation}\label{Eq: Third-order derivative bound for Ep}
\mleft\lvert \nabla^{3} E_{p}(z) \mright\rvert\le \Lambda_{1}\lvert z_{0}\rvert^{p-3}\quad \textrm{for all }z\in{\mathbb R}^{n}\setminus\{ 0\}.
\end{equation}
for some constant $\Lambda_{1}\in(0,\,\infty)$ depending at most on $\Gamma_{p}$ and $n$. From (\ref{Eq: Third-order derivative bound for Ep}), it is easy to conclude that (\ref{Eq: C-2-beta growth for Ep}) holds for all $z_{1},\,z_{2}\in{\mathbb R}^{n}$ with (\ref{Eq: C-2-beta variable condition}) and $\beta_{0}=1$. Indeed, the triangle inequality and (\ref{Eq: C-2-beta variable condition}) imply
\[0<\frac{\mu}{32}\le\lvert z_{1}+t(z_{2}-z_{1})\rvert\le 4\mu\quad \textrm{for all }t\in\lbrack0,\,1\rbrack.\]
Hence by $E_{p}\in C^{3}({\mathbb R}^{n}\setminus\{ 0\})$, (\ref{Eq: Third-order derivative bound for Ep}) and the Cauchy--Schwarz inequality, we have
\begin{align*}
\mleft\lVert \nabla^{2}E_{p}(z_{2})-\nabla^{2}E_{p}(z_{1})\mright\rVert&\le \mleft\lvert \nabla^{2}E_{p}(z_{2})-\nabla^{2}E_{p}(z_{1})\mright\rvert\\&\le \mleft(\int_{0}^{1}\mleft\lvert \nabla^{3}E_{p}(z_{1}+t(z_{2}-z_{1})) \mright\rvert^{2}\,{\mathrm{d}}t\mright)^{1/2}\lvert z_{1}-z_{2}\rvert\\&\le \mleft(32^{3-p}\vee 4^{p-3} \mright)\Lambda_{1}\lvert z_{1}-z_{2}\rvert.
\end{align*}
We mention that this symmetric setting generalizes our model case of the $p$-Laplace operator. In fact, it is easy to check that a special function \(g_{p}(\sigma)\coloneqq 2\sigma^{p/2}/p\) satisfies (\ref{Eq: Condition of gp}) with
\[\gamma_{p}\coloneqq \min\{\,1,\,p-1\,\},\quad \Gamma_{p}\coloneqq \max\mleft\{\,1,\,\frac{\lvert p-2\rvert}{2},\,\frac{\lvert (p-2)(p-4)\rvert}{4} \,\mright\}\]
In this case, the density \(E_{p}\) and the operator \({\mathcal L}_{p}\) become \(E_{p}(z)=\lvert z\rvert^{p}/p\) and \({\mathcal L}_{p}=-\Delta_{p}\) respectively. In particular, the density $E_{p}$ given by (\ref{Eq: Special case of E1 and Ep}) surely satisfies (\ref{Eq: C-2-beta growth for Ep}) with $\beta_{0}=1$. In this setting, a relaxed density \(E_{p,\,\varepsilon}\) is alternatively given by
\[E_{p,\,\varepsilon}(z)\coloneqq g_{p}(\varepsilon^{2}+\lvert z\rvert^{2}), \quad \textrm{so that}\quad \nabla E_{p,\,\varepsilon}(z)=g_{p}^{\prime}(\varepsilon^{2}+\lvert z\rvert^{2})z\]
for each \(\varepsilon\in(0,\,1)\).
By direct computations, we can check that this \(E_{p,\,\varepsilon}\) satisfies (\ref{Eq: Gradient bound for nabla Ep})--(\ref{Eq: Hessian estimate for Ep}) with \(\lambda_{0}^{\prime}=\lambda_{0},\,\Lambda_{0}^{\prime}=\Lambda_{0}\). Moreover, there holds
\begin{equation}\label{Eq: Third derivative growth for special E-p-epsilon}
\mleft\lvert \nabla^{3}E_{p,\,\varepsilon}(z)\mright\rvert\le \Lambda_{0}\mleft(\varepsilon^{2}+\lvert z\rvert^{2} \mright)^{(p-3)/2}\quad \textrm{for all }z\in{\mathbb R}^{n}.
\end{equation}

Similarly, we let \(E_{1}\) be of the form
\[E_{1}(z)=\frac{1}{2}g_{1}(\lvert z\rvert^{2})\quad \textrm{and therefore}\quad \nabla E_{1}(z)=g_{1}^{\prime}(\lvert z\rvert^{2})z\quad \textrm{for }z\in{\mathbb R}^{n}\setminus\{0\},\]
where \(g_{1}\colon [0,\,\infty)\rightarrow [0,\,\infty)\) is a non-decreasing continuous function with \(g_{1}\in C^{2}((0,\,\infty))\).
Since $E_{1}$ is positively one-homogeneous, this forces us to determine \(g_{1}\) explicitly. Indeed, (\ref{Eq: gradient equality for one-hom}) and the Uhlenbeck structure imply that \(\lvert \nabla E_{1}(z)\rvert=g_{1}^{\prime}(\lvert z\rvert^{2})\lvert z\rvert\) is constant for \(z\in{\mathbb R}^{n}\setminus\{ 0\}\). For this reason, \(g_{1}\) must satisfy \(\sqrt{\sigma}g_{1}^{\prime}(\sigma)=b\) for all \(\sigma\in(0,\,\infty)\), where $b\in\lbrack 0,\,\infty)$ is constant.
Combining with \(g_{1}(0)=E_{1}(0)=0\), we have \(g_{1}(\sigma)=2b\sigma^{1/2}\,(\sigma\ge 0)\), which yields \(E_{1}(z)=b\lvert z\rvert\) and \({\mathcal L}_{1}=-b\Delta_{1}\). In this special case, it is possible to give a relaxed density \(E_{1,\,\varepsilon}\) by
\[E_{1,\,\varepsilon}(z)\coloneqq b\sqrt{\varepsilon^{2}+\lvert z\rvert^{2}}, \quad \textrm{so that}\quad \nabla E_{1,\,\varepsilon}(z)=\frac{bz}{\sqrt{\varepsilon^{2}+\lvert z\rvert^{2}}}\]
for each \(\varepsilon\in(0,\,1)\), similarly to $E_{p,\,\varepsilon}$.  By direct computations, it is easy to check that this \(E_{1,\,\varepsilon}\) satisfies (\ref{Eq: approximated vector field bound in L-infty, E11})--(\ref{Eq: Hessian bound of approximated E1}) with \(c_{0}=0,\,K_{1}=K_{2}=K_{2}^{\prime}=b\). Moreover, there holds
\begin{equation}\label{Eq: Third derivative growth for special E-1-epsilon}
\mleft\lvert \nabla^{3}E_{1,\,\varepsilon}(z)\mright\rvert\le K_{3}\mleft(\varepsilon^{2}+\lvert z\rvert^{2} \mright)^{-1}\quad\textmd{for all }z\in{\mathbb R}^{n}.
\end{equation}
Here the constant $K_{3}\in(0,\,\infty)$ depends at most on $b$ and $n$.

It will be worth mentioning that under these settings, the inequality (\ref{Eq: Error of Hess}) in Lemma \ref{Lemma: Error estimate on Hess E-epsilon} can be deduced rather easily with $\beta_{0}=1$. This is possible by applying the following Lemma \ref{Lemma: Quantitative error estimate on Hessian} with $H_{\varepsilon}=E_{s,\,\varepsilon}\,(s\in\{\,1,\,p\,\})$. There it should be noted that these relaxed densities satisfy the assumption (\ref{Eq: H-epsilon-condition 3/3}) by (\ref{Eq: Third derivative growth for special E-p-epsilon})--(\ref{Eq: Third derivative growth for special E-1-epsilon}).
\begin{lemma}\label{Lemma: Quantitative error estimate on Hessian}
Let \(s\in\lbrack 1,\,\infty)\), and positive numbers \(\delta,\,\varepsilon\) satisfy (\ref{Eq: Range of delta-epsilon}).
Assume that a real-valued function \(H_{\varepsilon}\in C^{3}({\mathbb R}^{n})\) admits a constant \(L\in(0,\,\infty)\), independent of \(\varepsilon\), such that
\begin{equation}\label{Eq: H-epsilon-condition 1/3}
\mleft\lvert \nabla H_{\varepsilon}(z)\mright\rvert\le L\mleft(\varepsilon^{2}+\lvert z\rvert^{2}\mright)^{(s-1)/2},
\end{equation}
\begin{equation}\label{Eq: H-epsilon-condition 2/3}
\mleft\lVert \nabla^{2}H_{\varepsilon}(z)\mright\rVert\le L\mleft(\varepsilon^{2}+\lvert z\rvert^{2}\mright)^{s/2-1},
\end{equation}
\begin{equation}\label{Eq: H-epsilon-condition 3/3}
\mleft\lvert \nabla^{3}H_{\varepsilon}(z)\mright\rvert\le L\mleft(\varepsilon^{2}+\lvert z\rvert^{2}\mright)^{(s-3)/2},
\end{equation}
for all \(z\in{\mathbb R}^{n}\).
Then, for all \(z_{1},\,z_{2}\in{\mathbb R}^{n}\) satisfying (\ref{Eq: Outside condition on variable}), we have (\ref{Eq: Error estimate on Hessian}) with $\beta_{0}=1$
.\end{lemma}
\begin{proof}
We recall that there holds
\[(s+t)^{\gamma}\le \mleft\{\begin{array}{cc} t^{\gamma}& (\gamma\le 0), \\ \mleft(1\vee 2^{\gamma-1}\mright)\mleft(s^{\gamma}+t^{\gamma}\mright) & (0<\gamma<\infty), \end{array}\mright.\quad \textrm{for all }s\in\lbrack 0,\,\infty),\,t\in(0,\,\infty).\]
Hence from (\ref{Eq: H-epsilon-condition 1/3})--(\ref{Eq: H-epsilon-condition 3/3}), it follows that
\[\mleft\{\begin{array}{rcl}
\lvert \nabla H_{\varepsilon}(z)\rvert&\le&\mleft(1\vee 2^{(s-3)/2}\mright)L\cdot \mleft(\sigma_{1}+\lvert z\rvert^{s-1}\mright),\\ \mleft\lVert\nabla^{2}H_{\varepsilon}(z) \mright\rVert &\le & \mleft(1\vee 2^{(s-4)/2} \mright)L\cdot \mleft(\sigma_{2}+\lvert z\rvert^{s-2}\mright),\\ \mleft\lvert\nabla^{3}H_{\varepsilon}(z) \mright\rvert &\le & \mleft(1\vee 2^{(s-5)/2} \mright)L\cdot \mleft(\sigma_{3}+\lvert z\rvert^{s-3}\mright),
\end{array} \mright. \quad \textrm{for all }z\in{\mathbb R}^{n}\setminus\{ 0\},\]
where the constants \(\sigma_{1},\,\sigma_{2},\,\sigma_{3}\ge 0\) are defined by
\[\sigma_{k}\coloneqq \mleft\{\begin{array}{cc}
0 & (1\le s\le k),\\ \varepsilon^{s-k} &(k<s<\infty),
\end{array}\mright.\quad \textrm{for each }k\in\{\,1,\,2,\,3\,\}.\]
By Lemma \ref{Lemma: Local C11 estimate} and \(\lvert z_{1}\rvert\ge \mu/4>0\), we are able to obtain 
\begin{align*}
&\mleft\lvert \nabla^{2} H_{\varepsilon}(z_{1})(z_{2}-z_{1})-(\nabla H_{\varepsilon}(z_{2})-\nabla H_{\varepsilon}(z_{1}))\mright\rvert\\&=C(s)L\cdot \frac{\mleft(\lvert z_{1}\rvert^{s-1}+\lvert z_{2}\rvert^{s-1}\mright)+\sigma_{1}+\sigma_{2}\mleft(\lvert z_{1}\rvert+\lvert z_{2}\rvert\mright)+\sigma_{3}\lvert z_{1}\rvert^{2}}{\lvert z_{1}\rvert^{2}}\cdot\lvert z_{1}-z_{2}\rvert^{2}\\&\le C(s)L\mu^{s-3}\lvert z_{1}-z_{2}\rvert^{2} \quad \textrm{for all }z_{2}\in{\mathbb R}^{n}\setminus\{ 0\}.
\end{align*}
Here we have used \(\lvert z_{1}\rvert,\,\lvert z_{2}\rvert\le 2\mu\) and \(\sigma_{k}\le \mu^{s-k}\,(k<s<\infty)\), which immediately follow from (\ref{Eq: Range of delta-epsilon}) and (\ref{Eq: Outside condition on variable}). By \(H_{\varepsilon}\in C^{3}({\mathbb R}^{n})\), the estimate above is valid even for \(z_{2}=0\), and this completes the proof.
\end{proof}
Finally, when considering the special case where the operators are of the form
\[{\mathcal L}_{1}u=-\divx\mleft(\frac{\nabla u}{\lvert \nabla u\rvert}\mright),\quad {\mathcal L}_{p}u=-\divx\mleft(g_{p}^{\prime}(\lvert \nabla u \rvert^{2})\nabla u\mright),\]
where $g_{p}$ satisfies (\ref{Eq: Condition of gp}), we may introduce relaxed operators alternatively by
\[{\mathcal L}_{1}^{\varepsilon}u_{\varepsilon}\coloneqq \divx\mleft(\frac{\nabla u_{\varepsilon}}{\sqrt{\varepsilon^{2}+\lvert \nabla u_{\varepsilon}\rvert^{2}}}\mright),\quad {\mathcal L}_{p}^{\varepsilon}u_{\varepsilon}=-\divx\mleft(g_{p}^{\prime}(\varepsilon^{2}+\lvert \nabla u_{\varepsilon} \rvert^{2})\nabla u_{\varepsilon}\mright).\]
We note that all of the proofs after Section \ref{Rmk: Uhlenbeck structure case} work for this different approximation scheme. Also, it should be worth mentioning that, under this setting, everywhere $C^{1}$-regularity of weak solutions for elliptic system problems have been established in another recent paper \cite{T-system} by the author. There our method works under the assumption $g_{p}\in C_{\mathrm{loc}}^{2,\,\beta_{0}}((0,\,\infty))$ with $\beta_{0}\in(0,\,1\rbrack$.

\subsection{A fundamental result on convergences of vector fields}\label{Subsect: Basic results on convergence of vector fields}
In this section, we prove Lemma \ref{Lemma: a key lemma in justifications of convergence} below, which plays important roles in a mathematical justification of convergence for approximated solutions.
\begin{lemma}\label{Lemma: a key lemma in justifications of convergence}
Assume that $E_{1}\in C^{0}({\mathbb R}^{n})\cap C^{1}({\mathbb R}^{n}\setminus\{ 0\})$ is a positively one-homogeneous convex function, and $E_{p}\in C^{1}({\mathbb R}^{n})$ satisfies (\ref{Eq: Gradient bound condition for Ep}).
Let \(U\subset{\mathbb R}^{m}\) be a measurable set, and let \(\{\varepsilon_{k}\}_{k=1}^{\infty}\subset(0,\,1)\) be a non-increasing sequence such that \(\varepsilon_{k}\to 0\) as \(k\to\infty\). Assume that relaxed functions \(E_{1,\,\varepsilon},\,E_{p,\,\varepsilon}\in C^{1}({\mathbb R}^{n})\) satisfy (\ref{Eq: Gradient bound for nabla Ep}), (\ref{Eq: compact convergence of nabla Ep})--(\ref{Eq: Jensen-type inequality}), and (\ref{Eq: uniformly local convergence}). Then, we have the following:
\begin{enumerate}
\item\label{Item 1/3 strong convergence of dummy vector fields} For each fixed \(v\in L^{p}(U;\,{\mathbb R}^{n})\), we have
\begin{equation}\label{Eq: Strong convergence of dummy vector fields}
\nabla E_{\varepsilon_{k}}(v)\rightarrow A_{0}(v)\quad \textrm{in }L^{p^{\prime}}(U;\,{\mathbb R}^{n}),
\end{equation}
where \(A_{0}\) is defined by (\ref{Eq: A-0 pointwise convergence to dummy vector field}).
\item\label{Item 2/3 strong convergence of p-th growth term} 
Assume that a vector field \(v_{0}\in L^{p}(U;\,{\mathbb R}^{n})\) and a sequence of vector fields \(\{v_{\varepsilon_{k}}\}_{k=1}^{\infty}\subset L^{p}(U;\,{\mathbb R}^{n})\) satisfy \(v_{\varepsilon_{k}}\to v_{0}\) in \(L^{p}(U;\,{\mathbb R}^{n})\) as \(k\to\infty\). Then, there exists a subsequence \(\{\varepsilon_{k_{j}}\}_{j=1}^{\infty}\) such that
\begin{equation}\label{Eq: Lp-convergence of relaxed vector fields}
\nabla E_{p,\,\varepsilon_{k_{j}}}(v_{\varepsilon_{k_{j}}})\to \nabla E_{p}(v_{0})\quad \textrm{in } L^{p^{\prime}}(U;\,{\mathbb R}^{n}).
\end{equation}
\item\label{Item 3/3 constructions of subgradient vector fields} Assume that a sequence of measurable vector fields \(\{v_{\varepsilon_{k}}\colon U\rightarrow{\mathbb R}^{n}\}_{k=1}^{\infty}\) satisfies
\begin{equation}\label{Eq: a.e. conv on v.f.}
v_{\varepsilon_{k}}(x)\rightarrow v_{0}(x)\quad\textrm{for a.e. }x\in U
\end{equation}
for some measurable vector field \(v_{0}\colon U\rightarrow {\mathbb R}^{n}\).
Then, we have
\begin{equation}\label{Eq: weak star conv outside}
Z_{k}(x)\coloneqq \nabla E_{1,\,\varepsilon_{k}}(v_{\varepsilon_{k}}(x)) \overset{\ast}{\rightharpoonup}  \nabla E_{1}(v_{0}(x))\eqqcolon Z_{0} \quad \textrm{in}\quad L^{\infty}(D;\, {\mathbb R}^{n}),\end{equation}
where \(D\coloneqq \{ x\in U\mid v_{0}(x)\not= 0\}\).
Moreover, there exists a vector fields \(Z\in L^{\infty}(U,\,{\mathbb R}^{n})\) and a subsequence \(\{\varepsilon_{k_{j}}\}_{j=1}^{\infty}\) such that
\begin{equation}\label{Eq: weak star conv}
Z_{k_{j}}\overset{\ast}{\rightharpoonup}  Z\quad \textrm{in}\quad L^{\infty}(U;\,{\mathbb R}^{n}),
\end{equation}
\begin{equation}\label{Eq: limit is subgradient}
Z(x)\in \partial E_{1}(v_{0}(x))\quad \textrm{for a.e.}\quad x\in U.
\end{equation}
In particular, if a sequence of vector fields \(\{v_{\varepsilon_{k}}\}_{k=1}^{\infty}\subset L^{s}(U;\,{\mathbb R}^{n})\,(1\le s<\infty)\) converges to \(v_{0}\in L^{s}(U;\,{\mathbb R}^{n})\) with respect to the strong topology in \(L^{s}(U;\,{\mathbb R}^{n})\), then there exist a subsequence \(\{\varepsilon_{k_{j}}\}_{j=1}^{\infty}\) and a vector field \(Z\in L^{\infty}(U;\,{\mathbb R}^{n})\) satisfying (\ref{Eq: weak star conv})--(\ref{Eq: limit is subgradient}).
\end{enumerate}
\end{lemma}
\begin{proof}
\ref{Item 1/3 strong convergence of dummy vector fields} We fix \(v\in L^{p}(U;\,{\mathbb R}^{n})\). We note that for all \(\varepsilon\in(0,\,1)\),
\[\lvert \nabla E_{\varepsilon}(z)-A_{0}(z)\rvert\le \lvert \nabla E_{\varepsilon}(z)\rvert+\lvert A_{0}(z)\rvert\le C(1+\lvert z\rvert^{p-1})\quad \textrm{for all }z\in{\mathbb R}^{n}.\]
Here the constant \(C=C(p,\,\Lambda,\,K)\) is independent of \(\varepsilon\).
This clearly yields
\[\lvert \nabla E_{\varepsilon_{k}}(v(x))-A_{0}(v(x))\rvert^{p^{\prime}}\le C(p,\,\Lambda,\,K)(1+\lvert v(x)\rvert^{p})\eqqcolon g(x)\quad \textrm{for a.e. }x\in U,\]
where \(g\in L^{1}(U)\) is independent of \(k\in{\mathbb N}\). From (\ref{Eq: A-0 pointwise convergence to dummy vector field}), we have already known that \(\nabla E_{\varepsilon_{k}}(v(x))\rightarrow A_{0}(v(x))\) for a.e. \(x\in U\). Therefore by Lebesgue's dominated convergence theorem, we obtain (\ref{Eq: Strong convergence of dummy vector fields}).

\ref{Item 2/3 strong convergence of p-th growth term} By \cite[Theorem 4.9]{MR2759829}, we may take a subsequence \(\{v_{\varepsilon_{k_{j}}}\}_{j=1}^{\infty}\) such that
\begin{equation}\label{Eq: a.e. conv of v-epsilon}
v_{\varepsilon_{k_{j}}}(x)\rightarrow v_{0}(x)\quad\textrm{for a.e. }x\in U,
\end{equation}
\begin{equation}\label{Eq: a.e. bound of v-epsilon}
\lvert v_{\varepsilon_{k_{j}}}(x)\rvert\le h(x)\quad\textrm{for a.e. }x\in U,
\end{equation}
where \(h\in L^{p}(U)\) is independent of \(j\in{\mathbb N}\). From (\ref{Eq: compact convergence of nabla Ep}) and (\ref{Eq: a.e. conv of v-epsilon}), we can check that \(\nabla E_{p,\,\varepsilon_{k_{j}}}(v_{\varepsilon_{k_{j}}}(x))\to \nabla E_{p}(v_{0}(x))\) for a.e. \(x\in U\). Also, as in the proof of \ref{Item 1/3 strong convergence of dummy vector fields}, we can check that for a.e. \(x\in U\) there holds
\[\lvert \nabla E_{p,\,\varepsilon_{k_{j}}}(v_{\varepsilon_{k_{j}}}(x))-\nabla E_{p}(v_{0}(x))\rvert^{p^{\prime}}\le C\mleft(1+\lvert v_{0}(x)\rvert^{p}+h(x)^{p}\mright)\eqqcolon g(x).\]
Here the constant \(C=C(p,\,\Lambda,\,K)>0\) is independent of \(j\in{\mathbb N}\), and so is \(g\in L^{1}(U)\). We can conclude (\ref{Eq: Lp-convergence of relaxed vector fields}) from Lebesgue's dominated convergence theorem.

\ref{Item 3/3 constructions of subgradient vector fields} By (\ref{Eq: uniformly local convergence}) and (\ref{Eq: a.e. conv on v.f.}), we can check that \(Z_{k}(x)\to Z_{0}(x)\) for a.e. \(x\in D\). 
By (\ref{Eq: approximated vector field bound in L-infty, E11}) the vector field $Z_{k}$ satisfies \(\lVert Z_{k}\rVert_{L^{\infty}(U)}\le K\) for all \(k\in{\mathbb N}\). Hence, the weak convergence (\ref{Eq: weak star conv outside}),
\[\textrm{i.e.,}\int_{D}\langle Z_{k}(x)\mid\phi(x)\rangle\,{\mathrm d}x\rightarrow \int_{D}\langle Z_{0}(x)\mid\phi(x)\rangle\,{\mathrm d}x\quad\textrm{for all }\phi\in L^{1}(D;\,{\mathbb R}^{n})\]
easily follows from Lebesgue's dominated convergence theorem.

We should recall that the sequence \(\{Z_{k}\}_{k}\subset L^{\infty}(U;\,{\mathbb R}^{n})\) is bounded, and that the function space \(L^{1}(U;\,{\mathbb R}^{n})\), which is the predual of \(L^{\infty}(U;\,{\mathbb R}^{n})\), is separable. Hence by \cite[Corollary 3.30]{MR2759829}, we may take a vector field \(Z\in L^{\infty}(U;\,{\mathbb R}^{n})\) and a subsequence \(\{\varepsilon_{k_{j}}\}_{j}\) such that (\ref{Eq: weak star conv}) holds.
We are left to show that this weak$^{\ast}$ limit \(Z\) satisfies (\ref{Eq: limit is subgradient}).
Since it is clear that \(Z_{k_{j}} \overset{\ast}{\rightharpoonup}  Z\) in \(L^{\infty}(D,\,{\mathbb R}^{n})\), we have already known that \(Z(x)=Z_{0}(x)\in\partial E_{1}(v_{0}(x))\) for a.e. \(x\in D\). From (\ref{Eq: weak star conv}), it follows that
\begin{equation}\label{Eq: lsc}
\esssup\limits_{x\in U} {\tilde E_{1}}(Z(x))\le \liminf_{j\to\infty} \esssup\limits_{x\in U} {\tilde E_{1}}(Z_{k_{j}}(x)).
\end{equation}
For the proof of (\ref{Eq: lsc}), see \cite[Lemma 4]{giga2021continuity}.
Recall (\ref{Eq: Jensen-type inequality}), and we have
\begin{equation}\label{Eq: lsc result}
Z(x)\in \mleft\{w\in{\mathbb R}^{n}\mathrel{}\middle|\mathrel{} {\tilde E_{1}}(w)\le 1 \mright\}=\partial E_{1}(0)
\end{equation}
for a.e. \(x\in U\). This result implies that (\ref{Eq: limit is subgradient}) holds for a.e. $x\in U\setminus D$. The last statement is a consequence from \cite[Theorem 4.9]{MR2759829}.
\end{proof}
\begin{Remark}\label{Rmk: Euler's id}\upshape
In the proof of Lemma \ref{Lemma: a key lemma in justifications of convergence}.\ref{Item 3/3 constructions of subgradient vector fields}, it is also possible to check that \(Z\in L^{\infty}(\Omega;\,{\mathbb R}^{n})\) satisfies the Euler's identity;
\begin{equation}\label{Eq: Euler's id of vector field}
\langle v_{0}(x)\mid Z(x)\rangle=E_{1}(v_{0}(x))
\end{equation}
for a.e. \(x\in U\). In fact, for a.e. \(x\in U\setminus D\), (\ref{Eq: Euler's id of vector field}) is clear by \(v_{0}(x)=0\) and \(E_{1}(0)=0\). In the remaining case \(x\in D\), we have for all \(k\in {\mathbb N}\)
\[E_{1,\,\varepsilon_{k}}(v_{\varepsilon_{k}}(x))=E_{1,\,\varepsilon_{k}}(0)+\int_{0}^{1}\mleft\langle \nabla E_{1,\,\varepsilon_{k}}(t v_{\varepsilon_{k}}(x))\mathrel{}\middle|\mathrel{} v_{\varepsilon_{k}}(x)\mright\rangle\,{\mathrm d}t,\]
applying the fundamental theorem of calculus for \(E_{1,\,\varepsilon_{k}}\in C^{\infty}({\mathbb R}^{n})\).
Since it is clear that \(E_{1,\,\varepsilon_{k}}\) converges to \(E_{1}\) locally uniformly in \({\mathbb R}^{n}\), it follows that \(E_{1,\,\varepsilon_{k}}(v_{\varepsilon_{k}}(x))\rightarrow E_{1}(v_{0}(x))\) for a.e. \(x\in U\). Also, for a.e. \(0<t<1\) and a.e. \(x\in D\), we have
\[\mleft\{\begin{array}{rclcc}
\langle \nabla E_{1,\,\varepsilon_{k}}(t v_{\varepsilon_{k}}(x))\mid v_{\varepsilon_{k}}(x)\rangle & \rightarrow & \langle \nabla E_{1}(tv_{0}(x))\mid v_{0}(x) \rangle & \textrm{as}&k\to\infty,\\ \lvert \langle \nabla E_{1,\,\varepsilon_{k}}(t v_{\varepsilon_{k}}(x))\mid v_{\varepsilon_{k}}(x)\rangle  \rvert & \le & K\sup\limits_{k}\,\lvert v_{\varepsilon_{k}}(x)\rvert\eqqcolon C_{x}& \textrm{for all}&k\in{\mathbb N}.
\end{array} \mright.\]
Here we have used (\ref{Eq: gradient equality for one-hom}), (\ref{Eq: uniformly local convergence}), (\ref{Eq: a.e. conv on v.f.}).
Since \(C_{x}\) is a finite constant for a.e. \(x\in D\), which is independent of \(k\in{\mathbb N}\), we can apply Lebesgue's dominated convergence theorem.
Letting \(k\to \infty\) and noting that $\langle \nabla E_{1}(tv_{0}(x))\mid v_{0}(x) \rangle=\langle Z_{0}(x)\mid v_{0}(x) \rangle=\langle Z(x)\mid v_{0}(x) \rangle$ for a.e. \(x\in D\) and a.e. \(t\in(0,\,1)\), we are able to obtain
\[E_{1}(v_{0}(x))=E_{1}(0)+\int_{0}^{1} \langle Z(x)\mid v_{0}(x)\rangle\,{\mathrm d}t=\langle Z(x)\mid v_{0}(x)\rangle\quad\textrm{for a.e. }x\in D.\]
From (\ref{Eq: Representation of subdifferental set}) and (\ref{Eq: lsc result})--(\ref{Eq: Euler's id of vector field}), it follows that \(Z\) satisfies (\ref{Eq: limit is subgradient}).
\end{Remark}

\subsection{Convergence results on approximated variational problems}\label{Subsect: Convergence and solvability}
Only in Section \ref{Subsect: Convergence and solvability}, the exponent \(q\) is assumed to satisfy (\ref{Eq: Condition of q general}).
Similarly to Definition \ref{Def of Weak sol}, we define a weak solution of the Dirichlet boundary value problem
\begin{equation}\label{Eq: Dirichlet boundary value problem}
\mleft\{\begin{array}{ccccc}
{\mathcal L}u & = & f & \textrm{in} & \Omega,\\ u & = & u_{0} & \textrm{on} & \partial\Omega.
\end{array} \mright.
\end{equation}
\begin{definition}\label{Def: Weak sol}\upshape
Let \(p\in(1,\,\infty)\) and \(q\in \lbrack 1,\,\infty\rbrack\) satisfy (\ref{Eq: Condition of q general}). Let functions \(f\in L^{q}(\Omega)\) $u_{0}\in W^{1,\,p}(\Omega)$ be given. A function \(u\in u_{0}W^{1,\,p}(\Omega)\) is called the \textit{weak} solution of the Dirichlet boundary value problem (\ref{Eq: Dirichlet boundary value problem}), when there exists a vector field \(Z\in L^{\infty}(\Omega;\,{\mathbb R}^{n})\) such that the pair \((u,\,Z)\) satisfies (\ref{Eq: Weak formulation on the equation})--(\ref{Eq: Subgradient Z}).
\end{definition}

The main aim of Section \ref{Subsect: Convergence and solvability} is to prove that there uniquely exists a solution of the Dirichlet problem (\ref{Eq: Dirichlet boundary value problem}), and this solution can be obtained as a limit function of the approximation problem
\begin{equation}\label{Eq: Dirichlet boundary value problem approximated}
\mleft\{\begin{array}{ccccc}
{\mathcal L}^{\varepsilon}u_{\varepsilon} & = & f_{\varepsilon} & \textrm{in} & \Omega,\\ u_{\varepsilon} & = & u_{0} & \textrm{on} & \partial\Omega.
\end{array} \mright.
\end{equation}
Here for each $\varepsilon\in(0,\,1)$, the divergence operator ${\mathcal L}^{\varepsilon}$ is given by (\ref{Eq: Regularized equation}), and the function $f_{\varepsilon}\in L^{q}(\Omega)$ satisfies
\begin{equation}\label{Eq: Weak convergence on external force term}
f_{\varepsilon}\rightharpoonup f\quad \textrm{in }\sigma (L^{q}(\Omega),\,L^{q^{\prime}}(\Omega))\quad \textrm{as }\varepsilon\to 0.
\end{equation}
In other words, we assume that the approximated external force term $f_{\varepsilon}\in L^{q}(\Omega)$ converges to $f$ with respect to the weak and the weak$^{\ast}$ topology of $L^{q}(\Omega)$, respectively when $q<\infty$ and $q=\infty$. 

We set 
\begin{equation}\label{Eq: Energy}
{\mathcal F}_{0}(v)\coloneqq \int_{\Omega}E_{1}(\nabla v)\,{\mathrm d}x+\int_{\Omega}E_{p}(\nabla v)\,{\mathrm d}x-\int_{\Omega}fv\,{\mathrm d}x
\end{equation}
for each \(v\in W^{1,\,p}(\Omega)\). We also set
\begin{equation}\label{Eq: Energy approximated}
{\mathcal F}_{\varepsilon}(v)\coloneqq \int_{\Omega}E_{1,\,\varepsilon}(\nabla v)\,{\mathrm d}x+\int_{\Omega}E_{p,\,\varepsilon}(\nabla v)\,{\mathrm d}x-\int_{\Omega}f_{\varepsilon}v\,{\mathrm d}x
\end{equation}
for each \(v\in W^{1,\,p}(\Omega),\,\varepsilon\in(0,\,1)\).
As mentioned in Section \ref{Sect: Introduction}, the equation (\ref{Eq: main eq elliptic}) is connected with a minimizing problem of the functional ${\mathcal F}_{0}$ given by (\ref{Eq: Energy}). However, it should be noted that this functional is, in general, neither G\^{a}teaux differentiable nor Fr\'{e}chet differentiable. This is substantially due to non-smoothness of the density function $E_{1}$ at the origin. Therefore, it is natural to regard the term $\nabla E_{1}(\nabla u)$ as a vector field that satisfies (\ref{Eq: Subgradient Z}), as in Definition \ref{Def: Weak sol}. There we face to check whether it is possible to construct a vector field $Z$ satisfying (\ref{Eq: Subgradient Z}). We can overcome this problem by introducing regularized functionals ${\mathcal F}_{\varepsilon}\,(0<\varepsilon<1)$ given by (\ref{Eq: Energy approximated}), and by applying Lemma \ref{Lemma: a key lemma in justifications of convergence}. Our strategy works even for a variational inequality problem (Proposition \ref{Prop: Elliptic variational inequality}).
\begin{proposition}\label{Prop: Elliptic variational inequality}
Let \(p\in(1,\,\infty)\) and $q\in\lbrack 1,\,\infty\rbrack$ satisfy (\ref{Eq: Condition of q general}). Assume that $f\in L^{q}(\Omega)$ and $\{f_{\varepsilon}\}_{0<\varepsilon<1}\subset L^{q}(\Omega)$ satisfy (\ref{Eq: Weak convergence on external force term}). We define functionals in \({\mathcal F}_{\varepsilon}\,(0\le \varepsilon<1)\) by (\ref{Eq: Energy})--(\ref{Eq: Energy approximated}), where the density $E_{\varepsilon}$ is given by $E_{\varepsilon}=j_{\varepsilon}\ast E$ with $E=E_{1}+E_{p}$ satisfying (\ref{Eq: positive hom of deg 1})--(\ref{Eq: Hessian condition for Ep}).  Assume that ${\mathcal K}\subset W^{1,\,p}(\Omega)$ is a non-empty closed convex set, and there exists a positive constant \(C_{\mathcal K}\in(0,\,\infty)\) such that
\begin{equation}\label{Eq: inclusion on K-K}
\lVert v_{1}-v_{2}\rVert_{L^{p}(\Omega)}\le C_{\mathcal K}\lVert\nabla v_{1}-\nabla v_{2}\rVert_{L^{p}(\Omega)}\quad \textrm{for all }v_{1},\,v_{2}\in {\mathcal K}.
\end{equation}
Then, for a function \(u\in {\mathcal K}\), these following are equivalent:
\begin{enumerate}
\item\label{Item 1/3 minimizer} The function $u$ satisfies
\begin{equation}\label{Eq: minimizer in K}
u=\argmin\mleft\{{\mathcal F}_{0}(v)\mathrel{}\middle|\mathrel{}v\in {\mathcal K}\mright\}.
\end{equation}
In other words, \(u\) is the unique minimizer of the functional \({\mathcal F}_{0}\colon {\mathcal K}\rightarrow {\mathbb R}\). 
\item\label{Item 2/3 variational inequality new} There exists \(Z\in L^{\infty}(\Omega;\,{\mathbb R}^{n})\) such that the pair \((u,\,Z)\) satisfies (\ref{Eq: Subgradient Z}), and 
\begin{equation}\label{Eq: variational inequality; new type}
\int_{\Omega}\langle Z\mid \nabla (\psi-u)\rangle\,{\mathrm d}x+\int_{\Omega}\langle \nabla E_{p}(\nabla u)\mid \nabla (\psi-u) \rangle\,{\mathrm d}x \ge \int_{\Omega}f(\psi-u)\,{\mathrm d}x
\end{equation}
for all \(\psi\in {\mathcal K}\).
\end{enumerate}
Moreover, there exists a unique function \(u\in {\mathcal K}\) satisfying these equivalent properties \ref{Item 1/3 minimizer}--\ref{Item 2/3 variational inequality new}, and there holds \(u_{\varepsilon}\rightarrow u\) in \(W^{1,\,p}(\Omega)\) up to a subsequence. Here \(u_{\varepsilon}\) is a function defined by
\begin{equation}\label{Eq: minimizer in K approximated}
u_{\varepsilon}\coloneqq \argmin\mleft\{{\mathcal F}_{\varepsilon}(v)\mathrel{}\middle|\mathrel{} v\in {\mathcal K}\mright\}\in {\mathcal K}
\end{equation}
for each \(\varepsilon\in(0,\,1)\). 
\end{proposition}

\begin{proof}
We first mention that the right hand sides of (\ref{Eq: minimizer in K}) and (\ref{Eq: minimizer in K approximated}) are well-defined. To prove this, we will check coerciveness of ${\mathcal F}_{\varepsilon}$.
For each fixed $\varepsilon\in (0,\,1)$, $E_{p,\,\varepsilon}$ satisfies
\[E_{p,\,\varepsilon}(z)\ge E_{p,\,\varepsilon}(0)+\mleft\langle\nabla E_{p,\,\varepsilon}(0)\mathrel{}\middle|\mathrel{}z\mright\rangle+\lambda C(p)\mleft(\lvert z\rvert^{p}-\varepsilon^{p}\mright)\quad \textrm{for all }z\in{\mathbb R}^{n}.\]
In fact, by (\ref{Eq: Lp lower bound lemma case p>2})--(\ref{Eq: L2 lower bound lemma case p<2}) we can compute
\begin{align*}
E_{p,\,\varepsilon}(z)-E_{p,\,\varepsilon}(0)-\mleft\langle \nabla E_{p,\,\varepsilon}(0)\mathrel{}\middle|\mathrel{} z\mright\rangle&=\int_{0}^{1}\mleft\langle \nabla E_{p,\,\varepsilon}(tz)-\nabla E_{p,\,\varepsilon}(0)\mathrel{}\middle|\mathrel{}tz \mright\rangle\,\frac{{\mathrm{d}}t}{t}\\&\ge \lambda C(p)\mleft(\lvert z\rvert^{p}\int_{0}^{1}t^{p}\,{\mathrm{d}}t-\varepsilon^{p}\mright).
\end{align*}
It is easy to get the last inequality in the case $2\le p<\infty$. When $1<p<2$, we should note that  
\[\mleft(\varepsilon^{2}+\lvert tz\rvert^{2}\mright)^{p/2-1}\lvert tz\rvert^{2}\ge\mleft[\mleft(\varepsilon^{2}+\lvert tz\rvert^{2}\mright)^{p/2}-\varepsilon^{2}\mleft(\varepsilon^{2}+\lvert tz\rvert^{2}\mright)^{p/2-1}\mright]t\ge \mleft(t^{p}\lvert z\rvert^{p}-\varepsilon^{p}\mright)t\]
for all $t\in\lbrack0,\,1\rbrack$, which yields the desired inequality. 
Hence, for each \(v\in W^{1,\,p}(\Omega)\), \({\mathcal F}_{\varepsilon}\) satisfies
\begin{align*}
{\mathcal F}_{\varepsilon}(v)&\ge \int_{\Omega}E_{p,\,\varepsilon}(\nabla v)\,{\mathrm d}x-\int_{\Omega}f_{\varepsilon}v\,{\mathrm d}x\\&\ge \int_{\Omega}\mleft\langle\nabla E_{p,\,\varepsilon}(0)\mathrel{}\middle|\mathrel{}\nabla v\mright\rangle\,{\mathrm d}x+\lambda C(p)\mleft(\lVert\nabla v \rVert_{L^{p}(\Omega)}^{p}-\lvert\Omega\rvert\varepsilon^{p}\mright)\\&\quad -C(n,\,p,\,q,\,\Omega)\lVert f_{\varepsilon}\rVert_{L^{q}(\Omega)}\lVert v\rVert_{W^{1,\,p}(\Omega)}
\end{align*}
by $E_{p,\,\varepsilon}(0)\ge 0$, H\"{o}lder's inequality and the continuous embedding $W^{1,\,p}(\Omega)\hookrightarrow L^{q^{\prime}}(\Omega)$.
We may take and fix a function \(v_{0}\in {\mathcal K}\). Then, by (\ref{Eq: inclusion on K-K}), we have
\[\lVert v\rVert_{W^{1,\,p}(\Omega)}\le \lVert v-v_{0}\rVert_{W^{1,\,p}(\Omega)}+\lVert v_{0}\rVert_{W^{1,\,p}(\Omega)}\le C\cdot\mleft(\lVert\nabla v\rVert_{L^{p}(\Omega)}+\lVert v_{0}\rVert_{W^{1,\,p}(\Omega)}\mright)\]
with a new constant $C$ depending on $C_{\mathcal K}$.
With this in mind, by applying Young's inequality, we are able to find constants $\gamma=\gamma(\lambda,\,p)\in(0,\,1)$ and $\Gamma=\Gamma(n,\,p,\,q,\,\lambda,\,\Omega,\,C_{\mathcal K})\in(1,\,\infty)$ such that
\begin{equation}\label{Eq: Coercivity of F-epsilon}
{\mathcal F}_{\varepsilon}(v)\ge \gamma\lVert \nabla v\rVert_{L^{p}(\Omega)}^{p}-\Gamma\cdot\mleft( 1+\mleft\lvert \nabla E_{p,\,\varepsilon}(0)\mright\rvert^{p^{\prime}}+\lVert v_{0}\rVert_{W^{1,\,p}(\Omega)}\lVert f_{\varepsilon}\rVert_{L^{q}(\Omega)}+\lVert f_{\varepsilon}\rVert_{L^{q}(\Omega)}^{p^{\prime}}\mright)
\end{equation}
for all $v\in{\mathcal K}$. By the estimate (\ref{Eq: Coercivity of F-epsilon}), which implies that ${\mathcal F}_{\varepsilon}$ is a coercive functional in ${\mathcal K}$, it is easy to construct a minimizer of ${\mathcal F}_{\varepsilon}$ in ${\mathcal K}$ by direct methods. Here we note that the convex functional ${\mathcal F}_{\varepsilon}$ is continuous with respect the strong topology of $W^{1,\,p}(\Omega)$, and hence sequentially lower-semicontinuous with respect to the weak topology of $W^{1,\,p}(\Omega)$.
We note that the density $E_{\varepsilon}$ is strictly convex by (\ref{Eq: Hessian estimate for approximated E}). From this fact and (\ref{Eq: inclusion on K-K}), we can easily conclude uniqueness of a minimizer of the functional ${\mathcal F}_{\varepsilon}$. Therefore, the right hand side of (\ref{Eq: minimizer in K approximated}) is well-defined. This result similarly holds for ${\mathcal F}_{0}$. In fact, (\ref{Eq: Coercivity of F-epsilon}) is valid even when $\varepsilon=0$, and strict convexity of $E$ also follows from (\ref{Eq: Hessian condition for Ep}). 
For detailed arguments and computations omitted in the proof, see e.g., \cite[Chapter 8]{MR1625845}, \cite[Chapter 4]{MR1962933}, \cite[\S 3]{MR4201656}.

Let $(u,\,Z)\in {\mathcal K}\times L^{\infty}(\Omega;\,{\mathbb R}^{n})$ satisfy (\ref{Eq: Subgradient Z}) and (\ref{Eq: variational inequality; new type}). Then, for each \(\psi\in {\mathcal K}\), it is easy to get
\[\langle Z\mid \nabla (\psi- u) \rangle\le E_{1}(\nabla\psi)-E_{1}(\nabla u)\quad\textrm{a.e. in }\Omega,\]
which is so called a subgradient inequality.
Similarly, for each \(\psi\in {\mathcal K}\), we have another subgradient inequality
\[\langle \nabla E_{p}(\nabla u)\mid \nabla(\psi-u)\rangle \le E_{p}(\nabla\psi)-E_{p}(\nabla u)\quad \textrm{a.e. in }\Omega\]
by \(\partial E_{p}(z)=\{\nabla E_{p}(z)\}\) for all \(z\in{\mathbb R}^{n}\). From these, we can easily check \(0\le {\mathcal F}(\psi)-{\mathcal F}(u)\) for all \(\psi\in {\mathcal K}\), which implies that \(u\) satisfies (\ref{Eq: minimizer in K}).

We have already proved the implication \ref{Item 2/3 variational inequality new} $\Rightarrow$ \ref{Item 1/3 minimizer}, and unique existence of the minimizer (\ref{Eq: minimizer in K}). Hence, in order to complete the proof, it suffices to show the following two claims.
The first is that the function \(u_{\varepsilon}\in{\mathcal K}\,(0<\varepsilon<1)\) defined by (\ref{Eq: minimizer in K approximated}) converges to a certain function \(u\in{\mathcal K}\) strongly in \(W^{1,\,p}(\Omega)\) up to a sequence.
The second is that this limit function $u$ satisfies \ref{Item 2/3 variational inequality new}. We would like to prove these assertions by applying Lemma \ref{Lemma: a key lemma in justifications of convergence}.

For each \(\varepsilon\in(0,\,1)\), we set a function \(u_{\varepsilon}\in {\mathcal K}\) by (\ref{Eq: minimizer in K approximated}).
Then, for each \(\psi\in {\mathcal K}\), we have
\[0\le \frac{{\mathcal F}_{\varepsilon}(u_{\varepsilon}+t(\psi-u_{\varepsilon}))-{\mathcal F}_{\varepsilon}(u_{\varepsilon})}{t}\quad\textrm{for all }t\in(0,\,1).\]
Letting \(t\to 0\) and noting \(E_{\varepsilon}\in C^{\infty}({\mathbb R}^{n})\), we are able to obtain
\begin{align}\label{Eq: approximated variational inequality}
&\int_{\Omega}\langle \nabla E_{1,\,\varepsilon}(\nabla u_{\varepsilon})\mid \nabla(\psi-u_{\varepsilon}) \rangle\,{\mathrm d}x+\int_{\Omega}\langle \nabla E_{p,\,\varepsilon}(\nabla u_{\varepsilon})\mid \nabla(\psi-u_{\varepsilon}) \rangle\,{\mathrm d}x\nonumber\\&\quad \ge \int_{\Omega}f_{\varepsilon}(\psi-u_{\varepsilon})\,{\mathrm d}x
\end{align}
for all \(\psi\in {\mathcal K}\). 

We choose the number $\varepsilon_{0}\in(0,\,1)$ satisfying
\[\sup_{0<\varepsilon\le \varepsilon_{0}}\,\lvert\nabla E_{p,\,\varepsilon}(0)\rvert\le 1\quad \textrm{and}\quad \sup_{0<\varepsilon\le \varepsilon_{0}}\,\lVert f_{\varepsilon}\rVert_{L^{q}(\Omega)}\le F_{\ast}\]
for some constant $F_{\ast}\in(0,\,\infty)$. This is possible by (\ref{Eq: Gradient bound condition for Ep}), (\ref{Eq: compact convergence of nabla Ep}) and (\ref{Eq: Weak convergence on external force term}).
We will show that the net $\{u_{\varepsilon}\}_{0<\varepsilon\le \varepsilon_{0}}\subset W^{1,\,p}(\Omega)$ is bounded.
To show this, by (\ref{Eq: inclusion on K-K}), it suffices to find a constant $C\in(0,\,\infty)$, independent of $\varepsilon\in(0,\,\varepsilon_{0}\rbrack$, such that there holds
\begin{equation}\label{Eq: Claim on Lp bound of minima}
\sup_{0<\varepsilon\le \varepsilon_{0}}\,\lVert \nabla v_{\varepsilon}\rVert_{L^{p}(\Omega)}\le C<\infty.
\end{equation}
Without loss of generality, we may let $E_{p}(0)=0$, then we get
\[\lvert E_{p}(z)\rvert\le C(p,\,\Lambda)\lvert z\rvert^{p}\quad \textrm{for all }z\in{\mathbb R}^{n}\]
by elementary computations based on $E_{p}\in C^{1}({\mathbb R}^{n})$ and (\ref{Eq: Gradient bound condition for Ep}) (see e.g., \cite[Lemma 3]{MR4201656}).
On the other hand, by Euler's identity (\ref{Eq: Euler's identity}), it is clear that $\lvert E_{1}(z)\rvert \le K\lvert z\rvert$ for all $z\in{\mathbb R}^{n}$.
We apply (\ref{Eq: Mollified estimate on absolute value functions}) with $\sigma=1$ and $\sigma=p$, so that we have
\[\lvert E_{\varepsilon}(z)\rvert\le C(n,\,p,\,\Lambda,\,K)\mleft(1+\lvert z\rvert^{p}\mright)\quad \textrm{for all }z\in{\mathbb R}^{n}\]
by Young's inequality.
Since $v_{\varepsilon}$ satisfies (\ref{Eq: minimizer in K approximated}), we have
\[{\mathcal F}_{\varepsilon}(v_{\varepsilon})\le {\mathcal F}_{\varepsilon}(v_{0})\le C(n,\,p,\,q,\,\Lambda,\,K,\,\Omega)\mleft( 1+\lVert \nabla v_{0}\rVert_{L^{p}(\Omega)}^{p}+\lVert f_{\varepsilon}\rVert_{L^{q}(\Omega)}\lVert v_{0}\rVert_{W^{1,\,p}(\Omega)} \mright)\]
by H\"{o}lder's inequality and the continuous embedding $W^{1,\,p}(\Omega)\hookrightarrow L^{q^{\prime}}(\Omega)$.
By applying (\ref{Eq: Coercivity of F-epsilon}) with $v=v_{\varepsilon}$, we can find a finite constant $C\in(0,\,\infty)$, depending at most on $n$, $p$, $q$, $\lambda$, $\Lambda$, $K$, $\Omega$, $C_{{\mathcal K}}$, and $F_{\ast}$, such that (\ref{Eq: Claim on Lp bound of minima}) holds.

Hence by a weak compactness argument, we may take a sequence \(\{\varepsilon_{N}\}_{N=1}^{\infty}\subset (0,\,\varepsilon_{0})\) and a function \(u\in {\mathcal K}\) such that \(\varepsilon_{N}\to 0\) and \(u_{\varepsilon_{N}}\rightharpoonup u\) in \(W^{1,\,p}(\Omega)\). By the compact embedding \(W^{1,\,p}(\Omega)\hookrightarrow L^{q^{\prime}}(\Omega)\), we may also assume that \(u_{\varepsilon_{N}}\to u\) in \(L^{q^{\prime}}(\Omega)\) by taking a subsequence. For this relabeled sequence, which we write again as \(\{u_{\varepsilon_{N}}\}_{N=1}^{\infty}\), we will show that \(\nabla u_{\varepsilon_{N}}\to \nabla u\) in \(L^{p}(\Omega;\,{\mathbb R}^{n})\). Here we set and compute an integral \(I(\varepsilon)\) as following;
\begin{align*}
I(\varepsilon)&\coloneqq \int_{\Omega}\mleft\langle \nabla E_{\varepsilon} (\nabla u)-\nabla E_{\varepsilon}(\nabla u_{\varepsilon})\mid \nabla u-\nabla u_{\varepsilon} \mright\rangle\,{\mathrm d}x\\&= -\int_{\Omega}\mleft\langle \nabla E_{\varepsilon} (\nabla u_{\varepsilon})\mathrel{}\middle|\mathrel{} \nabla u-\nabla u_{\varepsilon} \mright\rangle\,{\mathrm d}x+\int_{\Omega}\mleft\langle \nabla E_{\varepsilon} (\nabla u)\mathrel{}\middle|\mathrel{} \nabla u-\nabla u_{\varepsilon} \mright\rangle\,{\mathrm d}x \\&\eqqcolon I_{1}(\varepsilon)+I_{2}(\varepsilon).
\end{align*}
We note that \(I(\varepsilon)\ge 0\) follows from (\ref{Eq: Lp lower bound lemma case p>2})--(\ref{Eq: L2 lower bound lemma case p<2}) and (\ref{Eq: non-quantitative monotonicity}).
For \(I_{1}(\varepsilon_{N})\), by testing \(\psi\coloneqq u\in {\mathcal K}\) in (\ref{Eq: approximated variational inequality}), we obtain
\[I_{1}(\varepsilon_{N})\le \int_{\Omega}f_{\varepsilon_{N}}(u_{\varepsilon_{N}}-u)\,{\mathrm d}x\le F_{\ast}\lVert u_{\varepsilon_{N}}-u\rVert_{L^{q^{\prime}}(\Omega)} \to 0\]
as \(N\to\infty\). Here we have used H\"{o}lder's inequality and \(u_{\varepsilon_{N}}\to u\) in \(L^{q^{\prime}}(\Omega)\). For \(I_{2}(\varepsilon_{N})\), we have already known that \(\nabla E_{\varepsilon_{N}}(\nabla u)\to A_{0}(\nabla u)\) in \(L^{p^{\prime}}(\Omega;\,{\mathbb R}^{n})\) from Lemma \ref{Lemma: a key lemma in justifications of convergence}.\ref{Item 1/3 strong convergence of dummy vector fields}. Combining this result with \(\nabla u_{\varepsilon_{N}}\rightharpoonup \nabla u\) in \(L^{p}(\Omega;\,{\mathbb R}^{n})\), we have \(I_{2}(\varepsilon_{N})\to 0\) as \(N\to\infty\). Finally, we obtain
\[0\le\liminf_{N\to\infty}I(\varepsilon_{N})\le\limsup_{N\to\infty}I(\varepsilon_{N})\le 0,\]
which yields \(I(\varepsilon_{N})\to 0\) as \(N\to\infty\).
From this convergence result, we are able to conclude that \(\nabla u_{\varepsilon_{N}}\to \nabla u\) in \(L^{p}(\Omega;\,{\mathbb R}^{n})\). 
In fact, when \(1<p<2\), we use H\"{o}lder's inequality, (\ref{Eq: L2 lower bound lemma case p<2}) and (\ref{Eq: non-quantitative monotonicity}) to compute
\begin{align*}\label{Eq: Lp norm conv elliptic}
\lVert \nabla u_{\varepsilon_{N}}-\nabla u\rVert_{L^{p}(\Omega)}^{p}&\le \mleft(\int_{\Omega}\mleft(\varepsilon_{N}^{2}+\lvert \nabla u_{\varepsilon_{N}}\rvert^{2}+\lvert\nabla u\rvert^{2}\mright)^{p/2-1}\lvert\nabla u_{\varepsilon_{N}}-\nabla u \rvert^{2}\,{\mathrm d}x \mright)^{2/p}\nonumber \\&\quad \cdot\underbrace{\sup\limits_{N}\,\mleft(\int_{\Omega}(\varepsilon_{N}^{2}+\lvert \nabla u_{\varepsilon_{N}}\rvert^{2}+\lvert \nabla u\rvert^{2})^{p/2}\,{\mathrm d}x\mright)^{(2-p)/p}}_{\eqqcolon C}\\ &\le C\lambda^{-2/p}\cdot I(\varepsilon_{N})^{2/p}\rightarrow 0\quad \textrm{as }N\to\infty.
\end{align*}
Here we note that \(C\) is finite by \(\nabla u_{\varepsilon_{N}}\rightharpoonup \nabla u\) in \(L^{p}(\Omega;\,{\mathbb R}^{n})\). Similarly, when \(2\le p<\infty\), we use (\ref{Eq: Lp lower bound lemma case p>2}) and (\ref{Eq: non-quantitative monotonicity}) to get \[\lVert \nabla u_{\varepsilon_{N}}-\nabla u\rVert_{L^{p}(\Omega)}^{p}\le (\lambda C(p))^{-1}\cdot I(\varepsilon_{N})\to 0\quad \textrm{as }N\to\infty.\]
Therefore, we are able to apply Lemma \ref{Lemma: a key lemma in justifications of convergence}.\ref{Item 2/3 strong convergence of p-th growth term}--\ref{Item 3/3 constructions of subgradient vector fields}. By taking a subsequence, we may assume that 
\[\mleft\{\begin{array}{ccccc}
\nabla E_{1,\,\varepsilon_{N}}(\nabla u_{\varepsilon_{N}})& \overset{\ast}{\rightharpoonup} & Z& \textrm{in}& L^{\infty}(\Omega;\,{\mathbb R}^{n}),\\ \nabla E_{p,\,\varepsilon_{N}}(\nabla u_{\varepsilon_{N}}) &\rightarrow & \nabla E_{p}(\nabla u)& \textrm{in}& L^{p^{\prime}}(\Omega;\,{\mathbb R}^{n}),
\end{array}\mright.\]
for some \(Z\in L^{\infty}(\Omega;\,{\mathbb R}^{n})\) satisfying (\ref{Eq: Subgradient Z}). Combining these results with (\ref{Eq: approximated variational inequality}) and
\[\mleft\{\begin{array}{ccccc}
\nabla (\psi -u_{\varepsilon_{N}})& \rightarrow &\nabla(\psi-u)& \textrm{in}& L^{p}(\Omega;\,{\mathbb R}^{n}),\\ \psi-u_{\varepsilon_{N}} &\rightarrow & \psi-u& \textrm{in}& L^{q^{\prime}}(\Omega),
\end{array}\mright.\]
for each fixed \(\psi\in {\mathcal K}\), we are able to check that the pair $(u,\,Z)\in {\mathcal K}\times L^{\infty}(\Omega;\,{\mathbb R}^{n})$ satisfies (\ref{Eq: variational inequality; new type}) for all \(\psi\in {\mathcal K}\). This completes the proof.
\end{proof}
From Proposition \ref{Prop: Elliptic variational inequality}, we are able to show unique existence of a solution to the Dirichlet boundary value problem, as stated in Corollary \ref{Cor: unique existence of elliptic Dirichlet prob} below.
\begin{corollary}\label{Cor: unique existence of elliptic Dirichlet prob}
Let \(p\in(1,\,\infty)\) and $q\in\lbrack 1,\,\infty\rbrack$ satisfy (\ref{Eq: Condition of q general}). Assume that $f\in L^{q}(\Omega)$ and $\{f_{\varepsilon}\}_{0<\varepsilon<1}\subset L^{q}(\Omega)$ satisfy (\ref{Eq: Weak convergence on external force term}). We define functionals in \({\mathcal F}_{\varepsilon}\,(0\le \varepsilon<1)\) by (\ref{Eq: Energy})--(\ref{Eq: Energy approximated}), where the density $E_{\varepsilon}$ is given by $E_{\varepsilon}=j_{\varepsilon}\ast E$ with $E=E_{1}+E_{p}$ satisfying (\ref{Eq: positive hom of deg 1})--(\ref{Eq: Hessian condition for Ep}).  Then, the Dirichlet problem (\ref{Eq: Dirichlet boundary value problem}) has a unique solution for each \(u_{0}\in W^{1,\,p}(\Omega)\). Moreover, this solution \(u\in u_{0}+W_{0}^{1,\,p}(\Omega)\) is characterized by the minimizing property
\[u=\argmin\mleft\{{\mathcal F}_{0}(v)\mathrel{}\middle|\mathrel{}v\in u_{0}+W_{0}^{1,\,p}(\Omega)\mright\},\]
and there holds $u_{\varepsilon}\to u$ in $W^{1,\,p}(\Omega)$ up to a subsequence. Here $u_{\varepsilon}$ is the unique solution of (\ref{Eq: Dirichlet boundary value problem approximated}). In other words, $u_{\varepsilon}$ satisfies $u_{\varepsilon}\in u_{0}+W_{0}^{1,\,p}(\Omega)$ and 
\begin{equation}\label{Eq: Weak formulation regularized ver}
\int_{\Omega}\mleft\langle\nabla E_{1,\,\varepsilon}(\nabla u_{\varepsilon})\mathrel{}\middle|\mathrel{}\nabla \phi\mright\rangle\,{\mathrm d}x+ \int_{\Omega}\mleft\langle\nabla E_{p,\,\varepsilon}(\nabla u_{\varepsilon})\mathrel{}\middle|\mathrel{}\nabla \phi\mright\rangle\,{\mathrm d}x=\int_{\Omega}f_{\varepsilon}\phi\,{\mathrm d}x
\end{equation}
for all $\phi\in W_{0}^{1,\,p}(\Omega)$.
\end{corollary}
\begin{proof}
We apply Proposition \ref{Prop: Elliptic variational inequality} with \({\mathcal K}=u_{0}+W_{0}^{1,\,p}(\Omega)\). Under this setting, it is easy to check that for each $\varepsilon\in(0,\,1)$, the unique minimizer $u_{\varepsilon}$ of the functional ${\mathcal F}_{\varepsilon}\colon u_{0}+W_{0}^{1,\,p}(\Omega)\to{\mathbb R}$ is characterized by the weak formulation (\ref{Eq: Weak formulation regularized ver}). Also, it should be mentioned that when a pair $(u,\,Z)\in W^{1,\,p}(\Omega)\times L^{\infty}(\Omega;\,{\mathbb R}^{n})$ satisfies (\ref{Eq: variational inequality; new type}) for all $\psi\in u_{0}+W_{0}^{1,\,p}(\Omega)$, then it also satisfies (\ref{Eq: Weak formulation on the equation}). In fact, for each $\phi\in W_{0}^{1,\,p}(\Omega)$, we can test $\psi\coloneqq u_{0}\pm \phi$ into (\ref{Eq: variational inequality; new type}), from which (\ref{Eq: Weak formulation on the equation}) is easy to deduce.
\end{proof}
\begin{Remark}\upshape
The proofs of Proposition \ref{Prop: Elliptic variational inequality} and Corollary \ref{Cor: unique existence of elliptic Dirichlet prob} work even when $E_{1}\in C^{0}({\mathbb R}^{n})\cap C^{1}({\mathbb R}^{n}\setminus \{ 0\}),\, E_{p}\in C^{1}({\mathbb R}^{n})\cap C^{2}({\mathbb R}^{n}\setminus \{ 0\})$. In fact, all of the inequalities applied in the proof of Proposition \ref{Prop: Elliptic variational inequality} follow from these regularity assumptions. It should be noted that (\ref{Eq: non-quantitative monotonicity}) follows from convexity of $E_{1,\,\varepsilon}=j_{\varepsilon}\ast E_{1}$.
\end{Remark}

\subsection{Proof of main theorem}\label{Subsect: Proof of main theorem}
Corollary \ref{Cor: unique existence of elliptic Dirichlet prob} in the previous Section \ref{Subsect: Convergence and solvability} implies that a weak solution to (\ref{Eq: main eq elliptic}) can be approximated by a weak solution to 
\begin{equation}\label{Eq: Reg Eq}
-\divx\mleft(\nabla E_{\varepsilon}(\nabla u_{\varepsilon})\mright)=f_{\varepsilon}\quad \textrm{with}\quad E_{\varepsilon}=j_{\varepsilon}\ast (E_{1}+E_{p}),
\end{equation}
under a suitable Dirichlet boundary condition, where $f_{\varepsilon}$ satisfies (\ref{Eq: Weak convergence on external force term}). In the proof of Theorem \ref{Theorem: C1-regularity}, we aim to prove ${\mathcal G}_{\delta}(\nabla u)$ is H\"{o}lder continuous. There Theorem \ref{Prop: A priori Hoelder estimate} below plays an important role.  
\begin{theorem}\label{Prop: A priori Hoelder estimate}
In addition to assumptions of Theorem \ref{Theorem: C1-regularity}, we let positive numbers \(\delta,\,\varepsilon\) satisfy (\ref{Eq: Range of delta-epsilon}), and a function \(u_{\varepsilon}\in W^{1,\,p}(\Omega)\) be a weak solution to (\ref{Eq: Reg Eq}) in \(\Omega\) with
\begin{equation}\label{Eq: Bound on external force term et al.}
\lVert f_{\varepsilon}\rVert_{L^{q}(\Omega)}\le F,
\end{equation}
and
\begin{equation}\label{Eq: Bound on nabla u-epsilon}
\lVert\nabla u_{\varepsilon}\rVert_{L^{p}(\Omega)}\le L
\end{equation}
for some constants $F,\,L\in(0,\,\infty)$.
Then, for each fixed \(x_{\ast}\in\Omega\), there exist a sufficiently small ball \(B_{\rho_{0}}(x_{\ast})\Subset\Omega\) and a sufficiently small number \(\alpha\in(0,\,1)\), such that \({\mathcal G}_{2\delta,\,\varepsilon}(\nabla u_{\varepsilon})\) is in \(C^{\alpha}\mleft(B_{\rho_{0}/2}(x_{\ast});\,{\mathbb R}^{n}\mright)\). Moreover, we have
\begin{equation}\label{Eq: a priori bounds of G-2delta-epsilon}
\mleft\lvert {\mathcal G}_{2\delta,\,\varepsilon}(\nabla u_{\varepsilon}(x)) \mright\rvert\le \mu_{0}\quad \textrm{for all }x\in B_{\rho_{0}/2}(x_{\ast})
\end{equation}
and
\begin{equation}\label{Eq: a priori Hoelder bounds of G-2delta-epsilon}
\mleft\lvert {\mathcal G}_{2\delta,\,\varepsilon}(\nabla u_{\varepsilon}(x_{1}))-{\mathcal G}_{2\delta,\,\varepsilon}(\nabla u_{\varepsilon}(x_{2}))\mright\rvert\le C\lvert x_{1}-x_{2}\rvert^{\alpha}\quad \textrm{for all }x_{1},\,x_{2}\in B_{\rho_{0}/2}(x_{\ast}).
\end{equation}
Here the exponent \(\alpha\in(0,\,1)\), the radius \(\rho_{0}\in(0,\,1)\) and the constants \(C,\,\mu_{0}\in(0,\,\infty)\) depend at most on $n$, $p$, $q$, $\beta_{0}$, $\lambda$, $\Lambda$, $K$, $F$, $L$, $\mathop{\mathrm{dist}}\,(x_{\ast},\,\partial\Omega)$, and $\delta$, but are independent of \(\varepsilon\). Moreover, the constant $\mu_{0}$ does not depend on $\delta$.
\end{theorem}
Theorem \ref{Prop: A priori Hoelder estimate} is proved in Section \ref{Sect: A priori Hoelder estimate}.
We would like to conclude Section \ref{Sect: Approximation} by giving the proof of Theorem \ref{Theorem: C1-regularity}.
\begin{proof}
Let \(u\in W^{1,\,p}(\Omega)\) be a weak solution to (\ref{Eq: main eq elliptic}). For each \(\varepsilon\in(0,\,1)\), we put the unique function \(u_{\varepsilon}\in u+W_{0}^{1,\,p}(\Omega)\) that solves the Dirichlet boundary value problem
\[\mleft\{\begin{array}{ccccc}
{\mathcal L}^{\varepsilon}u_{\varepsilon} & = & f_{\varepsilon} & \textrm{in} & \Omega,\\ u_{\varepsilon} & = & u & \textrm{on} & \partial\Omega.
\end{array} \mright.\]
By Corollary \ref{Cor: unique existence of elliptic Dirichlet prob}, we may take a subsequence \(\{u_{\varepsilon_{j}}\}_{j=1}^{\infty}\subset u+W_{0}^{1,\,p}(\Omega)\) such that \(u_{\varepsilon_{j}}\rightarrow u\) in \(W^{1,\,p}(\Omega)\). Hence, we may take finite constants $F,\,L\in(0,\,\infty)$, which are independent of $j\in{\mathbb N}$, such that (\ref{Eq: Bound on external force term et al.}) holds for every $\varepsilon=\varepsilon_{j}$. Moreover, by taking a subsequence if necessary, we may assume that
\begin{equation}\label{Eq: Regularized solutions a.e. convergence}
\nabla u_{\varepsilon_{j}}(x)\rightarrow \nabla u(x)\quad \textrm{for a.e. }x\in\Omega.
\end{equation}

We fix arbitrary \(\delta\in(0,\,1)\). Without loss of generality we let \(\varepsilon_{j}\in(0,\,\delta/8)\) for all \(j\in{\mathbb N}\).
By Theorem \ref{Prop: A priori Hoelder estimate}, we can choose a sufficiently small ball \(B_{\rho_{0}}(x_{\ast})\Subset \Omega\) such that the sequence \(\{{\mathcal G}_{2\delta,\,\varepsilon_{j}}(\nabla u_{\varepsilon_{j}})\}_{j=1}^{\infty}\subset C^{\alpha}\mleft(B_{\rho_{0}/2}(x_{\ast}),\,{\mathbb R}^{n}\mright)\) is bounded. Here the H\"{o}lder exponent \(\alpha\in(0,\,1)\) depends on \(\delta\) but is independent of \(\varepsilon_{j}\). Hence, by taking a subsequence, we can find a continuous vector field \(v_{\delta}\in C^{\alpha}\mleft(B_{\rho_{0}/2}(x_{\ast});\,{\mathbb R}^{n}\mright)\) such that \({\mathcal G}_{2\delta,\,\varepsilon_{j}}(\nabla u_{\varepsilon_{j}})\rightarrow v_{\delta}\) uniformly in \({B_{\rho_{0}/2}}(x_{\ast})\).
On the other hand, from (\ref{Eq: Regularized solutions a.e. convergence}) we have already known that \({\mathcal G}_{2\delta,\,\varepsilon_{j}}(\nabla u_{\varepsilon_{j}})\to {\mathcal G}_{2\delta}(\nabla u)\) a.e. in \(\Omega\) as \(j\to\infty\).
This implies that \(v_{\delta}={\mathcal G}_{2\delta}(\nabla u)\in C^{\alpha}\mleft(B_{\rho_{0}/2}(x_{0});\,{\mathbb R}^{n}\mright)\) holds for each fixed \(\delta\in(0,\,1)\) with \(\alpha=\alpha(\delta)\in(0,\,1)\). Since \(x_{\ast}\in\Omega\) is arbitrary, this completes the proof of \({\mathcal G}_{2\delta}(\nabla u)\in C^{0}\mleft(\Omega;\,{\mathbb R}^{n}\mright)\).

By the definition of ${\mathcal G}_{\delta}$, it is clear that
\[\sup_{\Omega}\,\mleft\lvert{\mathcal G}_{\delta_{1}}(\nabla u)-{\mathcal G}_{\delta_{2}}(\nabla u)\mright\rvert\le \lvert \delta_{1}-\delta_{2}\rvert\quad \textrm{for all }\delta_{1},\,\delta_{2}\in(0,\,1).\]
In particular, there exists a continuous vector field \(v_{0}\in C^{0}\mleft(\Omega;\,{\mathbb R}^{n}\mright)\) such that \({\mathcal G}_{\delta}(\nabla u)\rightarrow v_{0}\) uniformly in \(\Omega\).
On the other hand, there clearly holds \({\mathcal G}_{\delta}(\nabla u)\rightarrow \nabla u\) a.e. in \(\Omega\) as $\delta\to 0+$. Thus, \(\nabla u=v_{0}\in C^{0}\mleft(\Omega;\,{\mathbb R}^{n}\mright)\) is realized, and this completes the proof of Theorem \ref{Theorem: C1-regularity}.
\end{proof}

\section{A priori H\"{o}lder estimates of the mapping ${\mathcal G}_{2\delta,\,\varepsilon}(\nabla u_{\varepsilon})$}\label{Sect: A priori Hoelder estimate}
In Section \ref{Sect: A priori Hoelder estimate}, we consider weak solutions to the regularized equation (\ref{Eq: Reg Eq}), and show a priori H\"{o}lder bounds of truncated gradients (Theorem \ref{Prop: A priori Hoelder estimate}).

\subsection{Higher regularity on approximated solutions}\label{Subsect: A priori Lipschitz bound}
In Section \ref{Subsect: A priori Lipschitz bound}, we briefly describe inner regularity on weak solutions to (\ref{Eq: Reg Eq}).

We first note that it is not restrictive to assume that \(u_{\varepsilon}\in W^{1,\,p}(\Omega)\), a weak solution to (\ref{Eq: Reg Eq}) in $\Omega$, satisfies $u_{\varepsilon}\in W_{\mathrm{loc}}^{1,\,\infty}(\Omega)\cap W_{\mathrm{loc}}^{2,\,2}(\Omega)$ for each \(\varepsilon\in(0,\,1)\). This is possible by appealing to existing elliptic regularity theory, since the ellipticity ratio of (\ref{Eq: Reg Eq}) is bounded for each fixed $\varepsilon\in(0,\,1)$.
In fact, (\ref{Eq: Hessian estimate for approximated E}) implies that the regularized density $E_{\varepsilon}\in C^{\infty}(\Omega)$ satisfies
\[c_{\varepsilon}\mleft(1+\lvert z\rvert^{2}\mright)^{p/2-1}\mathrm{id}\leqslant \nabla^{2}E_{\varepsilon}(z)\leqslant C_{\varepsilon}\mleft(1+\lvert z\rvert^{2}\mright)^{p/2-1}\mathrm{id}\quad \textrm{for all }z\in{\mathbb R}^{n}\]
for some constants $0<c_{\varepsilon}<C_{\varepsilon}<\infty$ that may depend on an approximation parameter $\varepsilon$. Moreover, if $f_{\varepsilon}\in C^{\infty}(\Omega)$, then $u_{\varepsilon}\in C^{\infty}(\Omega)$ follows from bootstrap arguments \cite[Chapters 4--5]{MR0244627}. We recall that our approximation arguments work as long as (\ref{Eq: Weak convergence on external force term}) holds. Under this setting, we may choose $f_{\varepsilon}\in C^{\infty}(\Omega)$, so that the solution $u_{\varepsilon}$ satisfies (\ref{Eq: Reg Eq}) even in the classical sense. Also, even when $f_{\varepsilon}$ is not smooth but in $L^{q}(\Omega)\,(n<q\le \infty)$, by standard arguments as in \cite[Chapter 8]{MR1962933}, it is possible to get $u_{\varepsilon}\in C_{\mathrm{loc}}^{1,\,\alpha}(\Omega)\cap W_{\mathrm{loc}}^{2,\,2}(\Omega)$ for some $\alpha=\alpha(n,\,p,\,q,\,c_{\varepsilon},\,C_{\varepsilon})\in(0,\,1)$. Here we note that the ratio $C_{\varepsilon}/c_{\varepsilon}$ substantially depends on $\varepsilon$, and hence the exponent $\alpha$ may tend to $0$ as $\varepsilon\to 0$.

Since our arguments are local, we often let $u_{\varepsilon}\in W^{1,\,\infty}(B_{\rho}(x_{0}))\cap W^{2,\,2}(B_{\rho}(x_{0}))$ for some fixed open ball $B_{\rho}(x_{0})\Subset \Omega$. Under this setting, integrating by parts, we are able to deduce a weak formulation 
\begin{equation}\label{Eq: Weak formulation from EL}
\int_{B_{\rho}(x_{0})}\mleft\langle\nabla^{2}E_{\varepsilon}(\nabla u_{\varepsilon})\nabla \partial_{x_{j}}u_{\varepsilon}\mathrel{}\middle|\mathrel{}\nabla\phi\mright\rangle\,{\mathrm d}x=-\int_{B_{\rho}(x_{0})}f\partial_{x_{j}}\phi\,{\mathrm d}x
\end{equation}
for all \(j\in\{\,1,\,\dots,\,n\,\}\) and \(\phi\in W_{0}^{1,\,2}(B_{\rho}(x_{0}))\). Here we mention that the inclusion $\nabla^{2}E_{\varepsilon}(\nabla u_{\varepsilon})\in L^{2}(B_{\rho}(x_{0});\,{\mathbb R}^{n})$ follows from $u_{\varepsilon}\in W^{1,\,\infty}(B_{\rho}(x_{0}))\cap W^{2,\,2}(B_{\rho}(x_{0}))$.

Without proofs, we use local Lipschitz bounds of $u_{\varepsilon}$, uniformly for an approximation parameter $\varepsilon\in(0,\,1)$ (Proposition \ref{Prop: a priori Lipschitz}). Based on Moser's iterations, these estimates were already shown in the author's work \cite[Proposition 4]{MR4201656} by choosing test functions whose supports never intersects the facets of regularized solutions. There higher regularity assumptions \(E_{\varepsilon},\,f_{\varepsilon}\in C^{\infty}(\Omega)\) are imposed, so that \(u_{\varepsilon}\in C^{\infty}(\Omega)\). However, the proof therein works as long as \(u_{\varepsilon}\in W_{\mathrm{loc}}^{1,\,\infty}(\Omega)\cap W_{\mathrm{loc}}^{2,\,2}(\Omega)\), since the test functions chosen in \cite[Proposition 4]{MR4201656} become admissible under this regularity.
\begin{proposition}\label{Prop: a priori Lipschitz}
Let \(u_{\varepsilon}\in W^{1,\,p}(\Omega)\) be a weak solution to (\ref{Eq: Reg Eq}) in \(\Omega\) with $\varepsilon\in(0,\,1)$. Fix an open ball \(B_{r}\Subset\Omega\) with \(r\in(0,\,1\rbrack\). Then, for $n\ge 3$, we have
\[\esssup_{B_{\theta r}}\,\lvert \nabla u_{\varepsilon}\rvert\le \frac{C(n,\,p,\,q,\,\lambda,\,\Lambda,\,K)}{(1-\theta)^{n/p}}\mleft(1+\lVert f_{\varepsilon}\rVert_{L^{q}(B_{r})}^{1/(p-1)} +r^{-n/p}\lvert \nabla u_{\varepsilon}\rVert_{L^{p}(B_{r})}\mright)\]
for all \(\theta\in(0,\,1)\). For $n=2$, we have 
\[\esssup_{B_{\theta r}}\,\lvert \nabla u_{\varepsilon}\rvert\le \frac{C(n,\,p,\,q,\,\lambda,\,\Lambda,\,K,\,\chi)}{(1-\theta)^{2\chi/p(\chi-1)}}\mleft(1+\lVert f_{\varepsilon}\rVert_{L^{q}(B_{r})}^{1/(p-1)} +r^{-2/p}\lVert \nabla u_{\varepsilon}\rVert_{L^{p}(B_{r})}\mright)\]
for all \(\theta\in(0,\,1),\,\chi\in(1,\,\infty)\). 
\end{proposition}
It is noted that even when the external force term $f_{\varepsilon}$ is less regular than $f_{\varepsilon}\in L^{q}$, local Lipschitz estimates of $u_{\varepsilon}$ can be obtained by the recent result of \cite[Theorems 1.9 \& 1.11]{MR4078712}. There, external force terms are assumed to be in a Lorentz space or an Orlicz space, and the computations therein are based on De Giorgi's truncation.

\subsection{Three basic propositions and the proof of Theorem \ref{Prop: A priori Hoelder estimate}}\label{Subsect: A priori Hoelder estimate completed}
We first fix some notations in Section \ref{Sect: A priori Hoelder estimate}.
Throughout Section \ref{Sect: A priori Hoelder estimate}, we assume that the positive numbers \(\delta,\,\varepsilon\) satisfy (\ref{Eq: Range of delta-epsilon}).
For each \(u_{\varepsilon}\in W_{\mathrm{loc}}^{1,\,\infty}(\Omega)\cap W_{\mathrm{loc}}^{2,\,2}(\Omega)\), a weak solution to the regularized equation (\ref{Eq: Reg Eq}) in \(\Omega\), we define
\[V_{\varepsilon}\coloneqq \sqrt{\varepsilon^{2}+\lvert \nabla u_{\varepsilon}\rvert^{2}}\in L_{\mathrm{loc}}^{\infty}(\Omega)\cap W_{\mathrm{loc}}^{1,\,2}(\Omega),\]
and 
\[U_{\delta,\,\varepsilon}\coloneqq \mleft(V_{\varepsilon}-\delta\mright)_{+}^{2}\in L_{\mathrm{loc}}^{\infty}(\Omega)\cap W_{\mathrm{loc}}^{1,\,2}(\Omega).\]
For given numbers \(\mu\in(0,\,\infty),\,\nu\in(0,\,1)\) and an open ball \(B_{\rho}(x_{0})\Subset \Omega\), we define a superlevel set
\[S_{\rho,\,\mu,\,\nu}(x_{0})\coloneqq \mleft\{x\in B_{\rho}(x_{0})\mathrel{}\middle|\mathrel{}V_{\varepsilon}(x)-\delta>(1-\nu)\mu\mright\}.\]
For $f\in L^{1}(B_{\rho}(x_{0});\,{\mathbb R}^{m})$, we define an average integral
\[\fint_{B_{\rho}(x_{0})}f\,{\mathrm{d}x} \coloneqq \frac{1}{\lvert B_{\rho}(x_{0})\rvert}\int_{B_{\rho}(x_{0})}f\,{\mathrm{d}x}\in{\mathbb R}^{m},\] which is often written by $(f)_{x_{0},\,\rho}$ for notational simplicity.

We set an exponent $\beta\in(0,\,1)$ by
\begin{equation}\label{Eq: Definition of betas}
\beta\coloneqq \mleft\{\begin{array}{cc}
1-n/q & (n<q<\infty),\\ {\hat\beta}_{0} & (q=\infty),
\end{array}\mright.
\end{equation}
where \({\hat\beta}_{0}\) is an arbitrary number satisfying \(0<{\hat\beta}_{0}<1\). The number $\beta$ often appears when one considers regularity of weak solutions to the Poisson equation $-\Delta v=f\in L^{q}$. It is well-known that this weak solution $v$ is locally $\beta$-H\"{o}lder continuous, which can be proved by the standard freezing coefficient method (see e.g., \cite[Theorem 3.13]{MR2777537}; see also \cite[Chapter 3]{MR3887613} and \cite[Chapter 5]{MR3099262} as related items).

To show local a priori H\"{o}lder estimates (Theorem \ref{Prop: A priori Hoelder estimate}), we apply Propositions \ref{Prop: Perturbation result}--\ref{Prop: De Giorgi Oscillation lemma} below.
\begin{proposition}\label{Prop: Perturbation result}
Let \(u_{\varepsilon}\) be a weak solution to (\ref{Eq: Reg Eq}) in \(\Omega\).
Assume that positive numbers $\delta$, $\varepsilon$, $\mu$, $F$, $M$, and an open ball $B_{\rho}(x_{0})\Subset\Omega$ satisfy (\ref{Eq: Range of delta-epsilon}), (\ref{Eq: Bound on external force term et al.}),
\begin{equation}\label{Eq: mu>delta}
0<\delta<\mu,
\end{equation}
and
\begin{equation}\label{Eq: Gradient bound in Growth estimate}
\esssup_{B_{\rho}(x_{0})}\,V_{\varepsilon}\le \delta+\mu\le M.
\end{equation}
Then, there exist sufficiently small numbers \(\nu\in(0,\,1/6)\), \(\rho_{\star}\in(0,\,1)\), which depend at most on $n$, $p$, $q$, $\beta_{0}$, $\lambda$, $\Lambda$, $K$, $F$, $M$, and $\delta$, but are independent of \(\varepsilon\), such that the following statement holds true.
If there hold \(0<\rho\le \rho_{\star}\) and
\begin{equation}\label{Eq: Levelset assumption 2}
\lvert S_{\rho,\,\mu,\,\nu}(x_{0})\rvert >(1-\nu)\lvert B_{\rho}(x_{0})\rvert,
\end{equation}
then the limit
\[\Gamma_{2\delta,\,\varepsilon}(x_{0})\coloneqq \lim\limits_{r\to 0}\mleft({\mathcal G}_{2\delta,\,\varepsilon}(\nabla u_{\varepsilon})\mright)_{x_{0},\,r}\in{\mathbb R}^{n}\]
exists. Moreover, this limit satisfies 
\begin{equation}\label{Eq: G-delta-epsilon average-limit bound}
\mleft\lvert \Gamma_{2\delta,\,\varepsilon}(x_{0})\mright\rvert\le \mu,
\end{equation}
and we have the following Campanato-type growth estimate
\begin{equation}\label{Eq: Campanato-type beta-growth estimate}
\fint_{B_{r}(x_{0})}\mleft\lvert {\mathcal G}_{2\delta,\,\varepsilon}(\nabla u_{\varepsilon})-\Gamma_{2\delta,\,\varepsilon}(x_{0})\mright\rvert^{2}\,{\mathrm d}x\le\mu^{2} r^{2\beta}\quad\textrm{for all }r\in(0,\,\rho\rbrack.
\end{equation}
Here the exponent \(\beta\in(0,\,1)\) is defined by (\ref{Eq: Definition of betas}).
\end{proposition}
\begin{proposition}\label{Prop: De Giorgi Oscillation lemma}
Let \(u_{\varepsilon}\) be a weak solution to (\ref{Eq: Reg Eq}) in \(\Omega\).
Assume that positive numbers $\delta$, $\varepsilon$, $\mu$, $F$, $M$, and an open ball $B_{\rho}(x_{0})\Subset\Omega$ satisfy $0<\rho\le 1$, (\ref{Eq: Range of delta-epsilon}), (\ref{Eq: Bound on external force term et al.}),
\begin{equation}\label{Eq: Gradient bound in De Giorgi estimate}
\esssup_{B_{\rho}(x_{0})}\,\mleft\lvert {\mathcal G}_{\delta,\,\varepsilon}(\nabla u_{\varepsilon})\mright\rvert\le \mu\le \mu+\delta\le M,
\end{equation}
and 
\begin{equation}\label{Eq: Levelset assumption}
\mleft\lvert S_{\rho/2,\,\mu,\,\nu}(x_{0})\mright\rvert\le (1-\nu)\lvert B_{\rho/2}(x_{0})\rvert
\end{equation}
for some constant \(\nu\in(0,\,1/6)\). Then, there exist constants \(\kappa\in(2^{-\beta},1)\) and \(C_{\star}\in\lbrack1,\,\infty)\), which depend at most on $n$, $p$, $q$, $F$, $\lambda$, $\Lambda$, $K$, $M$, $\delta$, and $\nu$, but are independent of \(\varepsilon\), such that we have either 
\begin{equation}\label{Eq: Case 1}
\mu^{2}<C_{\star}\rho^{\beta},
\end{equation}
or
\begin{equation}\label{Eq: Case 2}
\esssup_{B_{\rho/4}(x_{0})}\,\mleft\lvert{\mathcal G}_{\delta,\,\varepsilon}(\nabla u_{\varepsilon})\mright\rvert\le \kappa\mu.
\end{equation}
Here the exponent \(\beta\in(0,\,1)\) is defined by (\ref{Eq: Definition of betas}).
\end{proposition}

Our analysis broadly depends on whether a scalar function $V_{\varepsilon}$ degenerates or not, which can be judged by measure assumptions as in (\ref{Eq: Levelset assumption 2}) or (\ref{Eq: Levelset assumption}). We would like to describe each individual part.

The assumption (\ref{Eq: Levelset assumption 2}) in Proposition \ref{Prop: Perturbation result} suggests that $V_{\varepsilon}$ should be non-degenerate near the point $x_{0}$. This expectation enables us to apply freezing coefficient methods to show Campanato-type growth estimates as in (\ref{Eq: Campanato-type beta-growth estimate}). In other words, Proposition \ref{Prop: Perturbation result} is based on analysis over non-degenerate points of a scalar function $V_{\varepsilon}$ or a gradient $\nabla u_{\varepsilon}$. To justify this, however, we will face to check that the average integral \((\nabla u_{\varepsilon})_{x_{0},\,r}\in{\mathbb R}^{n}\) never vanishes even when the radius $r$ tends to $0$. To answer this affirmatively, we would like to find the suitable numbers $\nu$ and $\rho_{\star}$ as in Proposition \ref{Prop: Perturbation result}. These numbers are the key criteria for judging whether it is possible to separate degenerate and non-degenerate points of a gradient $\nabla u_{\varepsilon}$. When choosing these important numbers, we will need a variety of energy estimates. There, the assumptions (\ref{Eq: mu>delta}) and (\ref{Eq: Levelset assumption 2}) are used to show these energy estimates.

When (\ref{Eq: Levelset assumption 2}) fails but instead a reversed estimate like (\ref{Eq: Levelset assumption}) holds, we will have to consider a case where $V_{\varepsilon}$ may degenerate. There we appeal to a truncation method to deduce a De Giorgi-type oscillation lemma (Proposition \ref{Prop: De Giorgi Oscillation lemma}). There we have to check that the scalar function $U_{\delta,\,\varepsilon}$ is a subsolution to a certain uniformly elliptic equation. This is possible since $U_{\delta,\,\varepsilon}$ is supported in $\{\,V_{\varepsilon}>\delta\,\}$, where the ellipticity ratio of $\nabla^{2} E_{\varepsilon}(\nabla u_{\varepsilon})$ is bounded. 

\begin{Remark}\label{Rmk: Remark on truncated composition of vector fields}\upshape
Let \(\delta,\,\varepsilon\) satisfy (\ref{Eq: Range of delta-epsilon}).
For \(z\in{\mathbb R}^{n}\) and \(\mu\in(0,\,\infty)\), the inequalities $\mleft\lvert{\mathcal G}_{\delta,\,\varepsilon}(z)\mright\rvert\le \mu$ and $\varepsilon^{2}+\lvert z\rvert^{2}\le (\delta+\mu)^{2}$ are equivalent.
In particular, (\ref{Eq: Gradient bound in Growth estimate}) is equivalent to (\ref{Eq: Gradient bound in De Giorgi estimate}). Also, under these equivalent conditions, we can easily check that \(u_{\varepsilon}\) satisfies
\begin{equation}\label{Eq: Bound on G-2delta-epsilon}
\esssup_{B_{\rho}(x_{0})}\,\mleft\lvert {\mathcal G}_{2\delta,\,\varepsilon}(\nabla u_{\varepsilon})\mright\rvert\le (\mu-\delta)_{+}\le \mu.
\end{equation}
\end{Remark}

From Propositions \ref{Prop: a priori Lipschitz}--\ref{Prop: De Giorgi Oscillation lemma}, we would like to give the proof of Theorem \ref{Prop: A priori Hoelder estimate}.
\begin{proof}
For each fixed \(x_{\ast}\in\Omega\), we first fix
\[R\coloneqq \min\mleft\{\,\frac{1}{2},\,\frac{1}{3}\mathop{\mathrm{dist}}\,(x_{\ast},\,\partial\Omega)\mright\}>0,\]
so that \(B_{2R}(x_{\ast})\Subset \Omega\) holds. By Proposition \ref{Prop: a priori Lipschitz}, we may take a finite constant \(\mu_{0}\in(1,\,\infty)\), depending at most on $n$, $p$, $q$, $\lambda$, $\Lambda$, $K$, $F$, $L$ and $R$, such that we have
\begin{equation}\label{Eq: Lipschitz bound}
\esssup_{B_{R}(x_{\ast})}\,V_{\varepsilon}\le \mu_{0}.
\end{equation}
We set \(M\coloneqq 1+\mu_{0}\), so that \(\mu_{0}+\delta\le M\) clearly holds. 

We choose and fix the numbers \(\nu\in(0,\,1/6),\,\rho_{\star}\in(0,\,1)\) as in Proposition \ref{Prop: Perturbation result}, which depend at most on $n$, $p$, $q$, $\beta_{0}$, $\lambda$, $\Lambda$, $K$, $F$, $M$, and \(\delta\). Corresponding to this \(\nu\), we choose finite constants \(\kappa\in(2^{-\beta},\,1),\,C_{\star}\in\lbrack1,\,\infty)\) as in Proposition \ref{Prop: De Giorgi Oscillation lemma}.
We define the desired H\"{o}lder exponent \(\alpha\in(0,\,\beta/2)\) by \(\alpha\coloneqq -\log\kappa/\log 4\), so that the identity \(4^{-\alpha}=\kappa\) holds.
We also put the radius \(\rho_{0}\) such that it satisfies
\begin{equation}\label{Eq: determination of radius}
0<\rho_{0}\le \min\mleft\{\,\frac{R}{2},\,\rho_{\star}\,\mright\}<1\quad\textrm{and}\quad C_{\star}\rho_{0}^{\beta}\le \kappa^{2}\mu_{0}^{2},
\end{equation}
which depends at most on $n$, $p$, $q$, $\beta_{0}$, $\lambda$, $\Lambda$, $K$, $F$, $M$, and \(\delta\).
We set non-negative decreasing sequences \(\{\rho_{k}\}_{k=1}^{\infty},\,\{\mu_{k}\}_{k=1}^{\infty}\) by \(\rho_{k}\coloneqq 4^{-k}\rho_{0},\,\mu_{k}\coloneqq \kappa^{k}\mu_{0}\) for \(k\in{\mathbb N}\).
By \(2^{-\beta}<\kappa=4^{-\alpha}<1\) and (\ref{Eq: determination of radius}), we can easily check that 
\begin{equation}\label{Eq: An estimate for induction}
\sqrt{ C_{\star}\rho_{k}^{\beta}}\le 2^{-\beta k}\kappa\mu_{0}\le\kappa^{k+1}\mu_{0}= \mu_{k+1},
\end{equation}
and
\begin{equation}\label{Eq: Estimate on radius ratio}
\mu_{k}=4^{-\alpha k}\mu_{0}=\mleft(\frac{\rho_{k}}{\rho_{0}}\mright)^{\alpha}\mu_{0}
\end{equation}
for every \(k\in{\mathbb Z}_{\ge 0}\).

We claim that for every \(x_{0}\in B_{\rho_{0}}(x_{\ast})\), the limit
\[\Gamma_{2\delta,\,\varepsilon}(x_{0})\coloneqq\lim_{r\to 0}\mleft({\mathcal G}_{2\delta,\,\varepsilon}(\nabla u_{\varepsilon})\mright)_{x_{0},\,r}\in{\mathbb R}^{n}\]
exists, and this limit satisfies
\begin{equation}\label{Eq: Campanato-type alpha-growth estimate}
\fint_{B_{r}(x_{0})}\mleft\lvert {\mathcal G}_{2\delta,\,\varepsilon}(\nabla u_{\varepsilon})-\Gamma_{2\delta,\,\varepsilon}(x_{0}) \mright\rvert^{2}\,{\mathrm d}x\le 4^{2\alpha+1}\mleft(\frac{r}{\rho_{0}}\mright)^{2\alpha}\mu_{0}^{2}\quad \textrm{for all }r\in(0,\,\rho_{0}\rbrack.
\end{equation}
In the proof of (\ref{Eq: Campanato-type alpha-growth estimate}), we introduce a set
\[{\mathcal N}\coloneqq \mleft\{k\in{\mathbb Z}_{\ge 0}\mathrel{}\middle|\mathrel{} \lvert S_{\rho_{k}/2,\,\mu_{k},\,\nu}(x_{0})\rvert>(1-\nu) \lvert B_{\rho_{k}/2}(x_{0})\rvert\mright\},\]
and define a number \(k_{\star}\coloneqq \min{\mathcal N}\in{\mathbb Z}_{\ge 0}\) when \({\mathcal N}\neq\emptyset\).
To show (\ref{Eq: Campanato-type alpha-growth estimate}), we consider the three possible cases:
\settasks{counter-format=(\arabic*)}
\begin{tasks}(3)
\task \label{CASE1} \({\mathcal N}\neq \emptyset\) and \(\mu_{k_{\star}}> \delta\).
\task \label{CASE2}\({\mathcal N}\neq \emptyset\) and \(\mu_{k_{\star}}\le \delta\).
\task \label{CASE3}\({\mathcal N}=\emptyset\).
\end{tasks}
It should be mentioned that when \({\mathcal N}\neq \emptyset\), there clearly holds \(\lvert S_{\rho_{k}/2,\,\mu_{k},\,\nu_{k}}(x_{0})\rvert\le (1-\nu)\lvert B_{\rho_{k}/2}(x_{0})\rvert\) for every \(k\in\{\,0,\,1,\,\dots\,,\,k_{\star}-1\,\}\), since \(k_{\star}\) is the minimum number of \({\mathcal N}\). 
Thus, in the cases \ref{CASE1}--\ref{CASE2}, we are able to obtain
\begin{equation}\label{Eq: DG}
\esssup_{B_{k}}\,\mleft\lvert{\mathcal G}_{\delta,\,\varepsilon}(\nabla u_{\varepsilon})\mright\rvert\le \mu_{k}\quad \textrm{for every }k\in\{\,0,\,1,\,\dots\,,\,k_{\star}\,\}
\end{equation}
by repeatedly applying (\ref{Eq: An estimate for induction}) and Proposition \ref{Prop: De Giorgi Oscillation lemma} with \((\rho,\,\mu)=(\rho_{k},\,\mu_{k})\) for \(k\in\{\,0,\,1,\,\dots\,,\,k_{\star}-1\,\}\). Here we write $B_{k}\coloneqq B_{\rho_{k}}(x_{0})\,(k\in{\mathbb Z}_{\ge 0})$ for notational simplicity.

\ref{CASE1}. By \(k_{\star}\in {\mathcal N}\), we are able to apply Proposition \ref{Prop: Perturbation result} in the open ball \(B_{\rho_{k_{\star}}/2}(x_{0})\) with \(\mu=\mu_{k_{\star}}\) (see Remark \ref{Rmk: Remark on truncated composition of vector fields}). In particular, the limit \(\Gamma_{2\delta,\,\varepsilon}(x_{0})\) exists and it satisfies
\begin{equation}\label{Eq: Campanato-type beta-growth in main theorem}
\fint_{B_{r}(x_{0})}\mleft\lvert{\mathcal G}_{2\delta,\,\varepsilon}(\nabla u_{\varepsilon})-\Gamma_{2\delta,\,\varepsilon}(x_{0})\mright\rvert^{2}\,{\mathrm d}x\le \mleft(\frac{2r}{\rho_{k_{\star}}}\mright)^{2\beta}\mu_{k_{\star}}^{2}\quad \textrm{for all }r\in\mleft( 0,\,\frac{\rho_{k_{\star}}}{2}\mright],
\end{equation}
and
\begin{equation}
\mleft\lvert \Gamma_{2\delta,\,\varepsilon}(x_{0})\mright\rvert\le \mu_{k_{\star}}.
\end{equation}
When \(0<r\le \rho_{k_{\star}}/2\), we use (\ref{Eq: Estimate on radius ratio}), (\ref{Eq: Campanato-type beta-growth in main theorem}) and \(\alpha<\beta\) to get
\[\fint_{B_{r}(x_{0})}\mleft\lvert{\mathcal G}_{2\delta,\,\varepsilon}(\nabla u_{\varepsilon})-\Gamma_{2\delta,\,\varepsilon}(x_{0})\mright\rvert^{2}\,{\mathrm d}x\le \mleft(\frac{2r}{\rho_{k_{\star}}}\mright)^{2\alpha}\mleft(\frac{\rho_{k_{\star}}}{\rho_{0}}\mright)^{2\alpha}\mu_{0}^{2}= 4^{\alpha}\mleft(\frac{r}{\rho_{0}}\mright)^{2\alpha}\mu_{0}^{2}.\]
To every \(r\in(\rho_{k_{\star}}/2,\,\rho_{0}\rbrack\), there corresponds a unique integer \(k\in\{\,0,\,\dots\,,\,k_{\star}\,\}\) such that \(\rho_{k+1}<r\le \rho_{k}\). By (\ref{Eq: DG}), we compute
\begin{align*}
\fint_{B_{r}(x_{0})}\mleft\lvert{\mathcal G}_{2\delta,\,\varepsilon}(\nabla u_{\varepsilon})-\Gamma_{2\delta,\,\varepsilon}(x_{0})\mright\rvert^{2}\,{\mathrm d}x&\le 2\fint_{B_{r}(x_{0})}\mleft(\mleft\lvert{\mathcal G}_{2\delta,\,\varepsilon}(\nabla u_{\varepsilon})\mright\rvert^{2}+\mleft\lvert\Gamma_{2\delta,\,\varepsilon}(x_{0})\mright\rvert^{2}\mright)\,{\mathrm d}x\\&\le 2\mleft(\esssup_{B_{k}}\,\mleft\lvert {\mathcal G}_{\delta,\,\varepsilon}(\nabla u_{\varepsilon})\mright\rvert^{2}+\mleft\lvert\Gamma_{2\delta,\,\varepsilon} \mright\rvert^{2}\mright)\\&\le 4\mu_{k}^{2}\le 4 \mleft(\frac{\rho_{k}}{\rho_{0}}\mright)^{2\alpha}\mu_{0}^{2}\le 4 \mleft(\frac{4r}{\rho_{0}}\mright)^{2\alpha}\mu_{0}^{2}.
\end{align*}

\ref{CASE2}. We recall (\ref{Eq: Bound on G-2delta-epsilon}) in Remark \ref{Rmk: Remark on truncated composition of vector fields}, which yields \({\mathcal G}_{2\delta,\,\varepsilon}(\nabla u_{\varepsilon})=0\) a.e. in \(B_{k_{\star}}\), and
\begin{equation}\label{Eq: G-2delta-epsilon bounded by G-delta-epsilon}
\mleft\lvert{\mathcal G}_{2\delta,\,\varepsilon}(\nabla u_{\varepsilon})\mright\rvert\le\mleft\lvert {\mathcal G}_{\delta,\,\varepsilon}(\nabla u_{\varepsilon})\mright\rvert\quad \textrm{a.e. in }\Omega.
\end{equation}
Combining with (\ref{Eq: DG}), we have
\begin{equation}\label{Eq: G-2delta-epsilon decay in digital}
\esssup_{B_{k}}\,\mleft\lvert{\mathcal G}_{2\delta,\,\varepsilon}(\nabla u_{\varepsilon})\mright\rvert\le \mu_{k}\quad \textrm{for every }k\in{\mathbb Z}_{\ge 0}.
\end{equation}
This clearly yields \(\Gamma_{2\delta,\,\varepsilon}(x_{0})=0\). To every \(r\in (0,\, \rho_{0}\rbrack\), there corresponds a unique \(k\in{\mathbb Z}_{\ge 0}\) such that \(\rho_{k+1}<r\le \rho_{k}\). By (\ref{Eq: G-2delta-epsilon decay in digital}) and \(\kappa=4^{-\alpha}\), we have
\begin{align*}
\fint_{B_{r}(x_{0})}\mleft\lvert {\mathcal G}_{2\delta,\,\varepsilon}(\nabla u_{\varepsilon})-\Gamma_{2\delta,\,\varepsilon}(x_{0})\mright\rvert^{2}\,{\mathrm d}x&=\fint_{B_{r}(x_{0})}\mleft\lvert {\mathcal G}_{2\delta,\,\varepsilon}(\nabla u_{\varepsilon})\mright\rvert^{2}\,{\mathrm d}x\rvert\\&\le \esssup_{B_{k}}\,\mleft\lvert{\mathcal G}_{2\delta,\,\varepsilon}(\nabla u_{\varepsilon})\mright\rvert^{2}\le \mu_{k}^{2}=4^{-2\alpha k}\mu_{0}^{2}\\&\le \mleft[4\cdot\mleft(\frac{r}{\rho_{0}}\mright)\mright]^{2\alpha}\mu_{0}^{2}=16^{\alpha}\mleft(\frac{r}{\rho_{0}}\mright)^{2\alpha}\mu_{0}^{2}.
\end{align*}

\ref{CASE3}. There clearly holds \(\lvert S_{\rho_{k}/2,\,\mu_{k},\,\nu_{k}}\rvert \le (1-\nu)\lvert B_{\rho_{k}/2}(x_{0})\rvert\) for every \(k\in{\mathbb Z}_{\ge 0}\). Applying (\ref{Eq: An estimate for induction}) and Proposition \ref{Prop: De Giorgi Oscillation lemma} with \((\rho,\,\mu)=(\rho_{k},\,\mu_{k})\,(k\in{\mathbb Z}_{\ge 0})\) repeatedly, we can easily check that 
\[\esssup_{B_{k}}\,\mleft\lvert{\mathcal G}_{\delta,\,\varepsilon}(\nabla u_{\varepsilon})\mright\rvert\le \mu_{k}\quad \textrm{for every }k\in{\mathbb Z}_{\ge 0}.\]
In particular, (\ref{Eq: G-2delta-epsilon decay in digital}) clearly follows from this result and (\ref{Eq: G-2delta-epsilon bounded by G-delta-epsilon}), and therefore the proof of (\ref{Eq: Campanato-type alpha-growth estimate}) can be accomplished, similarly to \ref{CASE2}.

In all possible cases, $\Gamma_{2\delta,\,\varepsilon}(x_{0})$ exists and satisfies (\ref{Eq: Campanato-type alpha-growth estimate}). Here it should be noted that the limit \(\Gamma_{2\delta,\,\varepsilon}\) satisfies
\begin{equation}\label{Eq: Boundedness on Gamma-2delta-epsilon}
\mleft\lvert \Gamma_{2\delta,\,\varepsilon}(x_{0})\mright\rvert\le \mu_{0}\quad \textrm{for all }x_{0}\in B_{\rho_{0}}(x_{\ast})
\end{equation}
by (\ref{Eq: Lipschitz bound}) and (\ref{Eq: G-2delta-epsilon bounded by G-delta-epsilon}).
From (\ref{Eq: Campanato-type alpha-growth estimate}) and (\ref{Eq: Boundedness on Gamma-2delta-epsilon}), we would like to show that
\begin{equation}\label{Eq: Hoelder estimate on Gamma-2delta-epsilon}
\mleft\lvert \Gamma_{2\delta,\,\varepsilon}(x_{1})-\Gamma_{2\delta,\,\varepsilon}(x_{2})\mright\rvert\le \mleft(\frac{2^{2\alpha+2+n/2}}{\rho_{0}^{\alpha}}\mu_{0}\mright)\lvert x_{1}-x_{2}\rvert^{\alpha}
\end{equation}
for all \(x_{1},\,x_{2}\in B_{\rho_{0}/2}(x_{\ast})\). We prove (\ref{Eq: Hoelder estimate on Gamma-2delta-epsilon}) by dividing into the two cases. In the case \(r\coloneqq \lvert x_{1}-x_{2}\rvert\le \rho_{0}/2\), we set a point \(x_{3}\coloneqq (x_{1}+x_{2})/2\in B_{\rho_{0}}(x_{\ast})\). Noting the inclusions \(B_{r/2}(x_{3})\subset B_{r}(x_{j})\subset B_{\rho_{0}/2}(x_{j})\subset B_{\rho_{0}}(x_{\ast})\) for each \(j\in\{\,1,\,2\,\}\), we use (\ref{Eq: Campanato-type alpha-growth estimate}) to obtain
\begin{align*}
&\mleft\lvert\Gamma_{2\delta,\,\varepsilon}(x_{1})-\Gamma_{2\delta,\,\varepsilon}(x_{2}) \mright\rvert^{2}\\&=\fint_{B_{r/2}(x_{3})}\mleft\lvert\Gamma_{2\delta,\,\varepsilon}(x_{1})-\Gamma_{2\delta,\,\varepsilon}(x_{2}) \mright\rvert^{2}\,{\mathrm d}x\\&\le 2\mleft(\fint_{B_{r/2}(x_{3})}\mleft\lvert{\mathcal G}_{2\delta,\,\varepsilon}(\nabla u_{\varepsilon})-\Gamma_{2\delta,\,\varepsilon}(x_{1})\mright\rvert^{2}\,{\mathrm d}x +\fint_{B_{r/2}(x_{3})}\mleft\lvert{\mathcal G}_{2\delta,\,\varepsilon}(\nabla u_{\varepsilon})-\Gamma_{2\delta,\,\varepsilon}(x_{2})\mright\rvert^{2}\,{\mathrm d}x\mright) \\& \le 2^{n+1}\mleft(\fint_{B_{r}(x_{1})}\mleft\lvert{\mathcal G}_{2\delta,\,\varepsilon}(\nabla u_{\varepsilon})-\Gamma_{2\delta,\,\varepsilon}(x_{1})\mright\rvert^{2}\,{\mathrm d}x +\fint_{B_{r}(x_{2})}\mleft\lvert{\mathcal G}_{2\delta,\,\varepsilon}(\nabla u_{\varepsilon})-\Gamma_{2\delta,\,\varepsilon}(x_{2})\mright\rvert^{2}\,{\mathrm d}x\mright)\\&\le 2^{n+4}\cdot 16^{\alpha}\mleft(\frac{r}{\rho_{0}}\mright)^{2\alpha}\mu_{0}^{2}=\mleft(\frac{2^{2\alpha+2+n/2}}{\rho_{0}^{\alpha}}\mu_{0}\mright)^{2}\lvert x_{1}-x_{2}\rvert^{2\alpha}, 
\end{align*}
which yields (\ref{Eq: Hoelder estimate on Gamma-2delta-epsilon}). In the remaining case \(\lvert x_{1}-x_{2}\rvert>\rho_{0}/2\), we simply use (\ref{Eq: Boundedness on Gamma-2delta-epsilon}) to get
\[\mleft\lvert\Gamma_{2\delta,\,\varepsilon}(x_{1})-\Gamma_{2\delta,\,\varepsilon}(x_{2})\mright\rvert\le 2\mu_{0}\le 2\cdot \frac{2^{\alpha}\lvert x_{1}-x_{2}\rvert^{\alpha}}{\rho_{0}^{\alpha}}\mu_{0},\]
which completes the proof of (\ref{Eq: Hoelder estimate on Gamma-2delta-epsilon}).
Finally, we mention that the mapping \(\Gamma_{2\delta,\,\varepsilon}\) is a Lebesgue representative of \({\mathcal G}_{2\delta,\,\varepsilon}\in L^{p}(\Omega;\,{\mathbb R}^{n})\) by Lebesgue's differentiation theorem, and therefore the desired estimates (\ref{Eq: a priori bounds of G-2delta-epsilon})--(\ref{Eq: a priori Hoelder bounds of G-2delta-epsilon}) immediately follow from (\ref{Eq: Boundedness on Gamma-2delta-epsilon})--(\ref{Eq: Hoelder estimate on Gamma-2delta-epsilon}) with $C=C(n,\,\alpha,\,\rho_{0})\in(0,\,\infty)$.
\end{proof}

\subsection{A weak formulation}\label{Subsect: Weak formulations of regularized equations}
In Section \ref{Subsect: Weak formulations of regularized equations}, we deduce a weak formulation on regularized solutions (Lemma \ref{Lemma: Weak formulations of approximated equations}).
\begin{lemma}\label{Lemma: Weak formulations of approximated equations}
Let \(u_{\varepsilon}\) be a weak solution to (\ref{Eq: Regularized equation}) in \(\Omega\) with \(0<\varepsilon<1\). Assume that \(\psi\colon[0,\,\infty)\rightarrow [0,\,\infty)\) is a non-decreasing Lipschitz function whose non-differentiable points are finitely many.
For any non-negative function \(\zeta\in W^{1,\,\infty}(B_{\rho}(x_{0}))\) that is compactly supported in an open ball \(B_{\rho}(x_{0})\Subset \Omega\), we set
\begin{equation}\label{Eq: Definitions of Jk}
\mleft\{\begin{array}{rcl}J_{1}&\coloneqq & \displaystyle\int_{B_{\rho}(x_{0})}\mleft\langle\nabla^{2}E_{\varepsilon}(\nabla u_{\varepsilon})\nabla V_{\varepsilon}\mathrel{}\middle|\mathrel{}\nabla\zeta\mright\rangle\psi(V_{\varepsilon})V_{\varepsilon}\,{\mathrm d}x,\\ J_{2}&\coloneqq & \displaystyle\int_{B_{\rho}(x_{0})}\mleft\langle\nabla^{2}E_{\varepsilon}(\nabla u_{\varepsilon})\nabla V_{\varepsilon}\mathrel{}\middle|\mathrel{}\nabla V_{\varepsilon}\mright\rangle\zeta\psi^{\prime}(V_{\varepsilon})V_{\varepsilon}\,{\mathrm d}x, \\ J_{3}&\coloneqq & \displaystyle\sum_{j=1}^{n}\displaystyle\int_{B_{\rho}(x_{0})}\mleft\langle\nabla^{2}E_{\varepsilon}(\nabla u_{\varepsilon})\nabla\partial_{x_{j}}u_{\varepsilon}\mathrel{}\middle|\mathrel{}\nabla\partial_{x_{j}}u_{\varepsilon}\mright\rangle\zeta\psi(V_{\varepsilon})\,{\mathrm d}x,\\ J_{4}&\coloneqq & \displaystyle\int_{B_{\rho}(x_{0})}\lvert f_{\varepsilon}\rvert^{2}\psi(V_{\varepsilon})V_{\varepsilon}^{2-p}\zeta\,{\mathrm d}x, \\ J_{5}&\coloneqq & \displaystyle\int_{B_{\rho}(x_{0})}\lvert f_{\varepsilon}\rvert^{2}\psi^{\prime}(V_{\varepsilon}) V_{\varepsilon}^{3-p}\zeta\,{\mathrm d}x, \\ J_{6}&\coloneqq & \displaystyle\int_{B_{\rho}(x_{0})}\lvert f_{\varepsilon}\rvert\lvert \nabla\zeta\rvert \psi(V_{\varepsilon})V_{\varepsilon}\,{\mathrm d}x. \end{array} \mright.
\end{equation}
Then, we have
\begin{equation}\label{Eq: Weak formulation of regularized equations}
2J_{1}+J_{2}+J_{3}\le \frac{n}{\lambda}(J_{4}+J_{5}) +2J_{6}.
\end{equation}
\end{lemma}
The resulting weak formulation (\ref{Eq: Weak formulation of regularized equations}) is fully used in Sections \ref{Subsect: Energy estimates} and \ref{Subsect: De Giorgi Oscillation}.
\begin{proof}
For each $j\in\{\,1,\,\dots,\,n\,\}$, we test \(\phi\coloneqq \zeta\psi(V_{\varepsilon})\partial_{x_{j}}u_{\varepsilon}\in W_{0}^{1,\,2}(B)\) into (\ref{Eq: Weak formulation from EL}). By summing with respect to \(j\in\{\,1,\,\dots,\,n\,\}\) and using
\[\sum\limits_{j=1}^{n}\partial_{x_{j}}u_{\varepsilon}\nabla\partial_{x_{j}}u_{\varepsilon}=\frac{1}{2}\sum\limits_{j=1}^{n}\nabla\mleft(\partial_{x_{j}}u_{\varepsilon}\mright)^{2}=\frac{1}{2}\nabla V_{\varepsilon}^{2}=V_{\varepsilon}\nabla V_{\varepsilon},\]
we get
\begin{align*}
J_{1}+J_{2}+J_{3}&=-\int_{B}f_{\varepsilon}\sum_{j=1}^{n}\partial_{x_{j}}\mleft(\zeta\psi(V_{\varepsilon})\partial_{x_{j}}u_{\varepsilon}\mright)\,{\mathrm d}x\\&=-\int_{B}f_{\varepsilon}\psi(V_{\varepsilon})\langle\nabla u_{\varepsilon}\mid \nabla\zeta\rangle\,{\mathrm d}x-\int_{B}f_{\varepsilon}\zeta\psi^{\prime}(V_{\varepsilon})\langle\nabla V_{\varepsilon}\mid \nabla u_{\varepsilon}\rangle\,{\mathrm d}x\\&\quad -\int_{B}f_{\varepsilon}\zeta\psi(V_{\varepsilon})\Delta u_{\varepsilon}\,{\mathrm d}x\\&\eqqcolon -(J_{7}+J_{8}+J_{9}).
\end{align*}
For \(J_{2},\,J_{3}\), we use (\ref{Eq: Hessian estimate for approximated E}) to get
\begin{equation}\label{Eq: positivity of J2-J3}
\mleft\{\begin{array}{rcl}J_{2}&\ge & \lambda\displaystyle\int_{B}V_{\varepsilon}^{p-1}\lvert \nabla V_{\varepsilon}\rvert^{2}\zeta\psi^{\prime}(V_{\varepsilon})\,{\mathrm d}x,\\ J_{3}& \ge & \lambda \displaystyle\int_{B}V_{\varepsilon}^{p-1}\mleft\lvert\nabla^{2}u_{\varepsilon}\mright\rvert^{2}\zeta\psi(V_{\varepsilon})\,{\mathrm d}x.\end{array} \mright.
\end{equation}
With this in mind, we compute
\[\lvert J_{8}\rvert\le \frac{\lambda}{2}\int_{B}V_{\varepsilon}^{p-1}\lvert \nabla V_{\varepsilon}\rvert^{2}\zeta\psi^{\prime}(V_{\varepsilon})\,{\mathrm d}x+\frac{1}{2\lambda}\int_{B}\lvert f_{\varepsilon}\rvert^{2} V_{\varepsilon}^{1-p}\lvert\nabla u_{\varepsilon}\rvert^{2}\zeta\psi^{\prime}(V_{\varepsilon})\,{\mathrm d}x\le \frac{J_{2}}{2}+\frac{J_{5}}{2\lambda},\]
and
\begin{align*}
\lvert J_{9}\rvert&\le \sqrt{n}\int_{B}\lvert f_{\varepsilon}\rvert\zeta\psi(V_{\varepsilon})\lvert\nabla^{2}u_{\varepsilon}\rvert\,{\mathrm d}x\\& \le \frac{\lambda}{2}\int_{B}V_{\varepsilon}^{p-2}\mleft\lvert\nabla^{2}u_{\varepsilon}\mright\rvert^{2}\zeta\psi(V_{\varepsilon})\,{\mathrm d}x+\frac{n}{2\lambda}\int_{B}\lvert f_{\varepsilon}\rvert^{2}\psi(V_{\varepsilon})V_{\varepsilon}^{2-p}\zeta\,{\mathrm d}x\\&\le \frac{J_{3}}{2}+\frac{n}{2\lambda}J_{4}
\end{align*}
by Young's inequality. Combining these results with \(\lvert J_{7}\rvert\le J_{6}\), we easily conclude (\ref{Eq: Weak formulation of regularized equations}).
\end{proof}

\subsection{Perturbation results from a higher integrability lemma}\label{Subsect: Perturbation outside facets}
For a given ball $B_{\rho}(x_{0})\Subset \Omega$, we consider an $L^{2}$-mean oscillation of the gradient $\nabla u_{\varepsilon}$, which is given by
\[\Phi(x_{0},\,\rho)\coloneqq \fint_{B_{\rho}(x_{0})}\mleft\lvert \nabla u_{\varepsilon}-(\nabla u_{\varepsilon})_{x_{0},\,\rho}\mright\rvert^{2}\,{\mathrm{d}}x.\]
Sections \ref{Subsect: Perturbation outside facets}--\ref{Subsect: Average integral estimates} are devoted to prove Proposition \ref{Prop: Perturbation result} by showing a variety of quantitative estimates related to this $\Phi$. Among them, comparisons between weak solutions to (\ref{Eq: Regularized equation}) and harmonic functions play an important role. 

Before making a comparison argument, we have to verify higher integrability estimates for $\lvert \nabla u_{\varepsilon}-(\nabla u_{\varepsilon})_{x_{0},\,\rho}\rvert$. This is possible by showing Lemma \ref{Lemma: Higher integrability} below. 
\begin{lemma}[Higher integrability lemma]\label{Lemma: Higher integrability}
Let \(u_{\varepsilon}\) be a weak solution to (\ref{Eq: Regularized equation}) in \(\Omega\). Assume that positive numbers $\delta$, $\varepsilon$, $\mu$, $F$, $M$, and an open ball $B_{\rho}(x_{0})\Subset\Omega$ satisfy (\ref{Eq: Range of delta-epsilon}), (\ref{Eq: Bound on external force term et al.}) and (\ref{Eq: mu>delta})--(\ref{Eq: Gradient bound in Growth estimate}). Then, for any vector \(\zeta\in{\mathbb R}^{n}\) satisfying
\begin{equation}\label{Eq: zeta location}
\delta+\frac{\mu}{4}\le \lvert \zeta\rvert\le \delta+\mu,
\end{equation}
there exists a constant \(\vartheta=\vartheta(n,\,p,\,q,\,\lambda,\,\Lambda,\,K,\,M,\,\delta)\) such that
\begin{equation}\label{Eq: integrability up}
0<\vartheta\le\min\{\,\beta,\,\beta_{0}\,\}<1,
\end{equation}
and 
\begin{equation}\label{Eq: Higher integrability result}
\fint_{B_{\rho/2}(x_{0})}\lvert \nabla u_{\varepsilon}-\zeta\rvert^{2(1+\vartheta)}\,{\mathrm d}x\le C\mleft[\mleft(\fint_{B_{\rho}(x_{0})}\lvert \nabla u_{\varepsilon}-\zeta\rvert^{2}\,{\mathrm d}x \mright)^{1+\vartheta}+F^{2(1+\vartheta)}\rho^{2\beta(1+\vartheta)}\mright].
\end{equation}
Here the constant \(C\in(0,\,\infty)\) depends at most on $n$, $p$, $q$, $\beta_{0}$, $\lambda$, $\Lambda$, $K$, $M$, and $\delta$.
\end{lemma}
In the proof of Lemma \ref{Lemma: Higher integrability}, we use so called Gehring's lemma without proofs. 
\begin{lemma}[Gehring's lemma]\label{Lemma: Gehring-type lemma}
Let \(B=B_{R}(x_{0})\subset{\mathbb R}^{n}\) be an open ball, and non-negative function \(g,\,h\) satisfy \(g\in L^{s}(B),\,h\in L^{{\tilde s}}(B)\) with \(1<s<{\tilde s}\le \infty\). Suppose that there holds
\[\fint_{B_{r}(z_{0})}g^{s}{\mathrm d}x\le {\hat C}\mleft[\mleft(\fint_{B_{2r}(z_{0})}g\,{\mathrm d}x \mright)^{s}+\fint_{B_{2r}(z_{0})}h^{s}\,{\mathrm d}x\mright]\]
for all \(B_{2r}(z_{0})\subset B\). Here \({\hat C}\in(0,\,\infty)\) is a constant independent of \(z_{0}\) and \(r>0\). Then, there exists a sufficiently small positive number \(\vartheta=\vartheta(s,\,{\tilde s},\,n,\,{\hat C})\) such that \(g\in L_{\mathrm{loc}}^{\sigma_{0}}(B)\) with \(\sigma_{0}\coloneqq s(1+\vartheta)\in(s,\,{\tilde s})\). Moreover, for each \(\sigma\in\lbrack s,\,\sigma_{0}\rbrack\), we have
\[\mleft(\fint_{B_{R/2}(x_{0})}g^{\sigma}{\mathrm d}x\mright)^{1/\sigma}\le C\mleft[\mleft(\fint_{B_{R}(x_{0})}g^{s}\,{\mathrm d}x\mright)^{1/s}+\mleft(\fint_{B_{R}(x_{0})}h^{\sigma}\,{\mathrm d}x\mright)^{1/\sigma}\mright],\]
where the constant ${\hat C}\in(0,\,\infty)$ depends at most on $\sigma$, $n$, $s$, ${\tilde s}$, and ${\hat C}$.
\end{lemma}
The proof of Gehring's lemma is found in \cite[Theorem 3.3]{MR2173373}, which is based on ball decompositions \cite[Lemma 3.1]{MR2173373} and generally works for a metric space with a doubling measure. As related items, see also \cite[Section 6.4]{MR1962933}, where the classical Calderon--Zygmund cube decomposition is fully applied to prove Gehring's lemma.
\begin{proof}
We later prove that there holds
\begin{equation}\label{Eq: Claim for Gehring's lemma}
\fint_{B_{r/2}(z_{0})}\lvert \nabla u_{\varepsilon}-\zeta\rvert^{2}\,{\mathrm d}x\le {\hat C}\mleft[\mleft(\fint_{B_{r}(z_{0})}\lvert \nabla u_{\varepsilon}-\zeta\rvert^{\frac{2n}{n+2}}\,{\mathrm d}x\mright)^{\frac{n+2}{n}}+\fint_{B_{r}(z_{0})}\lvert \rho f_{\varepsilon}\rvert^{2}\,\mathrm{d}x\mright] 
\end{equation}
for any open ball \(B_{r}(z_{0})\subset B\coloneqq B_{\rho}(x_{0})\).
Here \({\hat C}\in(0,\,1)\) is a constant depending on $n$, $p$, $\lambda$, $\Lambda$, $K$, $M$, and $\delta$.
Then, by applying Lemma \ref{Lemma: Gehring-type lemma} with \((s,\,{\tilde s})\coloneqq (1+2/n,\,q(n+2)/2n)\), \(g\coloneqq \lvert \nabla u_{\varepsilon}-\zeta\rvert^{\frac{2n}{n+2}}\in L^{s}(B),\,h\coloneqq \lvert \rho f_{\varepsilon}\rvert^{\frac{2n}{n+2}}\in L^{\tilde s}(B)\), we are able to find a small exponent \(\vartheta=\vartheta({\hat C},\,n,\,q)>0\) and a constant \(C=C(n,\,q,\,{\hat C},\,\vartheta)>0\) such that there hold (\ref{Eq: integrability up}) and \[\fint_{B_{\rho/2}(x_{0})}\lvert \nabla u_{\varepsilon}-\zeta\rvert^{2(1+\vartheta)}\,{\mathrm d}x\le C\mleft[\mleft(\fint_{B_{\rho}(x_{0})}\lvert \nabla u_{\varepsilon}-\zeta\rvert^{2}\,{\mathrm d}x \mright)^{1+\vartheta}+\fint_{B_{\rho}(x_{0})}\lvert \rho f_{\varepsilon}\rvert^{2(1+\vartheta)}\,{\mathrm d}x \mright].\]
We note that (\ref{Eq: integrability up}) yields $2(1+\vartheta)\le q$, and therefore the function $\lvert\rho f_{\varepsilon}\rvert^{2(1+\vartheta)}$ is integrable in $B_{\rho}(x_{0})$. Moreover, by H\"{o}lder's inequality, we have
\[\fint_{B_{\rho}(x_{0})}\lvert\rho f_{\varepsilon}\rvert^{2(1+\vartheta)}\,{\mathrm d}x\le C(n,\,\vartheta)F^{2(1+\vartheta)}\rho^{2\beta(1+\vartheta)},\]
from which (\ref{Eq: Higher integrability result}) follows.

We take and fix an arbitrary open ball \(B_{r}(z_{0})\subset B\).
To show (\ref{Eq: Claim for Gehring's lemma}), we set a function \(w_{\varepsilon}\in W^{1,\,\infty}(B_{r}(z_{0}))\) by
\[w_{\varepsilon}(x)\coloneqq u_{\varepsilon}(x)-(u_{\varepsilon})_{z_{0},\,r}-\langle\zeta\mid x-z_{0}\rangle\quad \textrm{for }x\in B_{r}(z_{0}),\]
so that \(\nabla w_{\varepsilon}=\nabla u_{\varepsilon}-\zeta\) holds. We choose a cutoff function \(\eta\in C_{c}^{1}(B_{r}(z_{0}))\) satisfying 
\[\eta\equiv 1\quad\textrm{on $B_{r/2}(z_{0})$}\quad\textrm{and}\quad\lvert\nabla\eta\rvert\le \frac{4}{r}\quad\textrm{in $B_{r}(z_{0})$},\]
and test \(\phi\coloneqq \eta^{2}w_{\varepsilon}\in W_{0}^{1,\,p}(B_{r}(z_{0}))\) into (\ref{Eq: Reg Eq}). Then, we obtain
\begin{align}
0&=\int_{B_{r}(z_{0})}\langle \nabla E_{\varepsilon}(\nabla u_{\varepsilon})-\nabla E_{\varepsilon}(\zeta)\mid \nabla \phi\rangle\,{\mathrm d}x-\int_{B_{r}(z_{0})}f_{\varepsilon}\phi\,{\mathrm d}x\nonumber\\&= \int_{B_{r}(z_{0})}\eta^{2}\langle \nabla E_{\varepsilon}(\nabla u_{\varepsilon})-\nabla E_{\varepsilon}(\zeta)\mid \nabla w_{\varepsilon}\rangle\,{\mathrm d}x\nonumber\\&\quad+2\int_{B_{r}(z_{0})}\eta w_{\varepsilon}\langle \nabla E_{\varepsilon}(\nabla u_{\varepsilon})-\nabla E_{\varepsilon}(\zeta)\mid \nabla \eta\rangle\,{\mathrm d}x-\int_{B_{r}(z_{0})}\eta^{2}f_{\varepsilon}w_{\varepsilon}\,{\mathrm d}x\nonumber\\&\eqqcolon {\mathbf J}_{1}+{\mathbf J}_{2}+{\mathbf J}_{3}.\nonumber
\end{align}
The assumptions (\ref{Eq: mu>delta})--(\ref{Eq: Gradient bound in Growth estimate}) and (\ref{Eq: zeta location}) enables us to apply Lemma \ref{Lemma: Quantitative error estimate on Hessian} \ref{Item 1/2: Monotonicity and Growth estimate}. 
As a result, we have 
\[{\mathbf J}_{1}\ge C_{1}\int_{B_{r}(z_{0})}\eta^{2}\lvert \nabla w_{\varepsilon}\rvert^{2}{\mathrm d}x\]
by (\ref{Eq: Monotonicity outside}). Similarly for \({\mathbf J}_{2}\), we can use (\ref{Eq: Growth outside}) and Young's inequality to obtain
\begin{align*}
\lvert {\mathbf J}_{2}\rvert&\le C_{2}\int_{B_{r}(z_{0})}\eta \lvert w_{\varepsilon}\rvert\lvert\nabla w_{\varepsilon}\rvert\lvert\nabla\eta\rvert\,{\mathrm d}x\\&\le \frac{C_{1}}{2}\int_{B_{r}(z_{0})}\eta^{2}\lvert \nabla w_{\varepsilon}\rvert^{2}{\mathrm d}x+\frac{C_{2}^{2}}{2C_{1}}\int_{B_{r}(z_{0})}\lvert \nabla \eta\rvert^{2}w_{\varepsilon}^{2}\,{\mathrm d}x.
\end{align*}
For \({\mathbf J}_{3}\), we use Young's inequality to obtain
\[\lvert{\mathbf J}_{3}\rvert\le\frac{1}{r^{2}}\int_{B_{r}(z_{0})}\eta^{2}w_{\varepsilon}^{2}\,{\mathrm d}x+\frac{r^{2}}{4}\int_{B_{r}(z_{0})}\eta^{2}\lvert f_{\varepsilon}\rvert^{2}\,{\mathrm d}x.\]
By these estimates and our choice of $\eta$, we are able to compute
\begin{align*}
\fint_{B_{r/2}(z_{0})}\mleft\lvert\nabla w_{\varepsilon}\mright\rvert^{2}\,{\mathrm d}x&\le 2^{n}\fint_{B_{r}(z_{0})}\mleft\lvert\nabla(\eta w_{\varepsilon})\mright\rvert^{2}\,{\mathrm d}x\\&\le C\mleft[\fint_{B_{r}(z_{0})}\mleft(\lvert\nabla \eta\rvert^{2}+\frac{\eta^{2}}{r^{2}}\mright) w_{\varepsilon}^{2}\,{\mathrm d}x+r^{2}\fint_{B_{r}(z_{0})}\eta^{2}\lvert f_{\varepsilon}\rvert^{2}\,{\mathrm d}x \mright]\\&\le C\mleft[r^{-2}\fint_{B_{r}(z_{0})} w_{\varepsilon}^{2}\,{\mathrm d}x+\fint_{B_{r}(z_{0})}\lvert \rho f_{\varepsilon}\rvert^{2}\,{\mathrm d}x\mright]
\end{align*}
with \(C\in(0,\,\infty)\) depending on \(n,\,C_{1},\,C_{2}\).
By the definition, \(w_{\varepsilon}\) satisfies
\[\fint_{B_{r}(z_{0})}w_{\varepsilon}\,{\mathrm d}x=0,\]
which enables us to use the Poincar\'{e}--Sobolev inequality \cite[Chapter IX, Theorem 10.1]{MR1897317};
\[\fint_{B_{r}(z_{0})}w_{\varepsilon}^{2}\,{\mathrm d}x\le C(n)r^{2}\mleft(\fint_{B_{r}(z_{0})}\lvert\nabla w_{\varepsilon}\rvert^{\frac{2n}{n+2}}\,{\mathrm d}x\mright)^{\frac{n+2}{n}}.\]
Recall \(\nabla w_{\varepsilon}=\nabla u_{\varepsilon}-\zeta\), and this completes the proof of (\ref{Eq: Claim for Gehring's lemma}).
\end{proof}
Now we would like to give perturbation estimates based on comparisons to harmonic functions (Lemma \ref{Lemma: Excess decay estimate 1}). This result can be obtained by the regularity assumption $C_{\mathrm{loc}}^{2,\,\beta_{0}}({\mathbb R}^{n}\setminus\{0\})$, and plays an important role in the proof of Proposition \ref{Prop: Perturbation result}.
\begin{lemma}\label{Lemma: Excess decay estimate 1}
Let \(u_{\varepsilon}\) be a weak solution to (\ref{Eq: Regularized equation}) in $\Omega$. Assume that positive numbers $\delta$, $\varepsilon$, $\mu$, $F$, $M$, and an open ball $B_{\rho}(x_{0})\Subset\Omega$ satisfy $0<\rho\le 1$, (\ref{Eq: Range of delta-epsilon}), (\ref{Eq: Bound on external force term et al.}) and (\ref{Eq: mu>delta})--(\ref{Eq: Gradient bound in Growth estimate}). In addition, we let
\begin{equation}\label{Eq: Average assumption}
\mleft\lvert(\nabla u_{\varepsilon})_{x_{0},\,\rho} \mright\rvert\ge \delta+\frac{\mu}{4},
\end{equation}
and consider the Dirichlet boundary problem
\begin{equation}\label{Eq: Dirichlet boundary Poisson problem}
\mleft\{\begin{array}{rclcc}
-\mathrm{div}\,\mleft(\nabla^{2} E_{\varepsilon}\mleft((\nabla u_{\varepsilon})_{x_{0},\,\rho}\mright)\nabla v_{\varepsilon}\mright) &=&0 &\textrm{in}& B_{\rho/2}(x_{0}),\\ v_{\varepsilon}&=&u_{\varepsilon}&\textrm{on} &\partial B_{\rho/2}(x_{0}).
\end{array} \mright.
\end{equation}
Then, there uniquely exists a function \(v_{\varepsilon}\) that solves (\ref{Eq: Dirichlet boundary Poisson problem}). Moreover, we have 
\begin{equation}\label{Eq: Perturbation estimates}
\fint_{B_{\rho/2}(x_{0})}\lvert\nabla u_{\varepsilon}-\nabla v_{\varepsilon}\rvert^{2}\,{\mathrm d}x\le C\mleft\{ \mleft[\frac{\Phi(x_{0},\,\rho)}{\mu^{2}}\mright]^{\vartheta}\Phi(x_{0},\,\rho)+\mleft(F^{2}+F^{2(1+\vartheta)}\mright)\rho^{2\beta}\mright\},
\end{equation}
where \(\vartheta\) is the positive constant given in Lemma \ref{Lemma: Higher integrability}, and
\begin{equation}\label{Eq: Dirichlet Growth estimate on Poisson equations}
\fint_{B_{\tau \rho}(x_{0})}\lvert\nabla v_{\varepsilon}-(\nabla v_{\varepsilon})_{x_{0},\,\tau\rho }\rvert^{2}\,{\mathrm d}x\le C\tau^{2}\fint_{B_{\rho/2}(x_{0})}\mleft\lvert\nabla v_{\varepsilon}-(\nabla v_{\varepsilon})_{x_{0},\,\rho/2}\mright\rvert^{2}\,{\mathrm d}x
\end{equation}
for all \(\tau\in(0,\,1/2\rbrack\).
Here the exponent \(\beta\) is defined by (\ref{Eq: Definition of betas}), and the constants \(C\in(0,\,\infty)\) in (\ref{Eq: Perturbation estimates})--(\ref{Eq: Dirichlet Growth estimate on Poisson equations}) depend at most on $n$, $p$, $q$, $\beta_{0}$, $\lambda$, $\Lambda$, $K$, $M$, and $\delta$.
\end{lemma}
\begin{proof}
For notational simplicity, we write \(\zeta\coloneqq(\nabla u_{\varepsilon})_{x_{0},\,r}\in{\mathbb R}^{n}\) and \(B\coloneqq B_{\rho/2}(x_{0})\).
Then by (\ref{Eq: Gradient bound in Growth estimate}), it is easy to check that 
\begin{equation}\label{Eq: Zeta bound}
\lvert \zeta\rvert\le \fint_{B_{\rho}(x_{0})}V_{\varepsilon}\,{\mathrm{d}}x\le \delta+\mu.
\end{equation}
Combining with (\ref{Eq: Range of delta-epsilon}), (\ref{Eq: mu>delta}) and (\ref{Eq: Average assumption}), we have already known that \[0<\frac{\delta^{2}}{16}\le\frac{\mu^{2}}{16}\le \varepsilon^{2}+\lvert \zeta\rvert^{2}\le \delta^{2}+(2\mu)^{2}\le 5\mu^{2}\le 5M^{2}.\] Hence by (\ref{Eq: Hessian estimate for approximated E}), there exist a constant \(l_{0}=l_{0}(p)\in(0,\,1)\) and another constant \(m\in(0,\,1)\), which depends at most on \(p,\,\lambda,\,\Lambda,\,K,\,M,\,\delta\), such that
\[m\mathrm{id}\leqslant l_{0}\lambda \mu^{p-2}\mathrm{id} \leqslant \nabla^{2} E_{\varepsilon}(\zeta)\leqslant m^{-1}\mathrm{id}.\] 
In other words, \(\nabla^{2}E_{\varepsilon}(\zeta)\) is uniformly elliptic, which yields unique existence of the solution \(v_{\varepsilon}\in u_{\varepsilon}+W_{0}^{1,\,2}(B)\) of the problem (\ref{Eq: Dirichlet boundary Poisson problem}). Moreover, since the coefficient matrix \(\nabla^{2}E(\zeta)\) is constant, we are able to find a constant \(C=C(m,\,n)\in(0,\,\infty)\) such that (\ref{Eq: Dirichlet Growth estimate on Poisson equations}) holds (see e.g., \cite[Lemma 2.17]{MR3887613}, \cite[Proposition 5.8]{MR3099262}).

We note that functions \(u_{\varepsilon}\in W^{1,\,\infty}(B)\) and \(v_{\varepsilon}\in u_{\varepsilon}+W_{0}^{1,\,2}(B)\) respectively satisfy
\[\int_{B}\mleft\langle\nabla E_{\varepsilon}(\nabla u_{\varepsilon})\mathrel{}\middle|\mathrel{} \nabla\phi\mright\rangle\,{\mathrm d}x=\int_{B}f_{\varepsilon}\phi\,{\mathrm d}x\]
for all \(\phi\in W_{0}^{1,\,1}(B)\), and
\[\int_{B}\mleft\langle \nabla^{2} E_{\varepsilon}(\zeta)\nabla v_{\varepsilon}\mathrel{}\middle|\mathrel{} \nabla \phi\mright\rangle\,{\mathrm d}x=0\]
for all \(\phi\in W_{0}^{1,\,2}(B)\). In particular, for all \(\phi\in W_{0}^{1,\,2}(B)\), we have
\begin{align}
&\int_{B}\mleft\langle \nabla^{2} E_{\varepsilon}(\zeta)(\nabla u_{\varepsilon}-\nabla v_{\varepsilon})\mathrel{}\middle|\mathrel{} \nabla \phi\mright\rangle\,{\mathrm d}x\nonumber\\&=\int_{B}\mleft\langle \nabla^{2} E_{\varepsilon}(\zeta)\nabla u_{\varepsilon}-\nabla E_{\varepsilon}(\nabla u_{\varepsilon})\mathrel{}\middle|\mathrel{} \nabla \phi\mright\rangle\,{\mathrm d}x+\int_{B}f_{\varepsilon}\phi\,{\mathrm{d}}x\nonumber\\&=\int_{B}\mleft\langle \nabla^{2}E_{\varepsilon}(\zeta)(\nabla u_{\varepsilon}-\zeta)-\mleft(\nabla E_{\varepsilon}(\nabla u_{\varepsilon})-\nabla E_{\varepsilon}(\zeta)\mright)\mathrel{}\middle|\mathrel{}\nabla\phi\mright\rangle\,{\mathrm d}x+\int_{B}f_{\varepsilon}\phi\,{\mathrm{d}}x,\nonumber
\end{align}
where it is noted that the vectors \(\nabla^{2}E_{\varepsilon}(\zeta)\zeta,\,\,\nabla E_{\varepsilon}(\zeta)\in{\mathbb R}^{n}\) are constant.
By (\ref{Eq: Gradient bound in Growth estimate}), (\ref{Eq: Average assumption}) and (\ref{Eq: Zeta bound}), we are able to apply Lemma \ref{Lemma: Error estimate on Hess E-epsilon}. As a result, there exists a constant $C\in(0,\,\infty)$, depending at most on $n$, $p$, $\beta_{0}$, $\Lambda$, $K$ and $\delta$, such that a weak formulation
\[\int_{B}\mleft\langle \nabla^{2} E_{\varepsilon}(\zeta)(\nabla u_{\varepsilon}-\nabla v_{\varepsilon})\mathrel{}\middle|\mathrel{} \nabla \phi\mright\rangle\,{\mathrm d}x\le C\mu^{p-2-\beta_{0}}\int_{B}\lvert\nabla u_{\varepsilon}-\zeta\rvert^{1+\beta_{0}}\lvert \nabla\phi\rvert\,{\mathrm d}x+\int_{B}\lvert f_{\varepsilon}\rvert\lvert \phi\rvert\,{\mathrm{d}}x\]
holds for all \(\phi\in W_{0}^{1,\,2}(B)\).
Now we test \(\phi\coloneqq u_{\varepsilon}-v_{\varepsilon}\in W_{0}^{1,\,2}(B)\) into this weak formulation. By H\"{o}lder's inequality and the Poincar\'{e} inequality, we have
\begin{align*}
&l_{0}\lambda\mu^{p-2}\int_{B}\lvert \nabla u_{\varepsilon}-\nabla v_{\varepsilon}\rvert^{2}\,{\mathrm d}x\\&\le C(n,\,p,\,\beta_{0},\,\Lambda,\,K,\,\delta)\mu^{p-2-\beta_{0}}\int_{B}\lvert \nabla u_{\varepsilon}-\zeta\rvert^{1+\beta_{0}}\lvert \nabla u_{\varepsilon}-\nabla v_{\varepsilon}\rvert\,{\mathrm d}x+\int_{B}\lvert f_{\varepsilon}\rvert\lvert u_{\varepsilon}-v_{\varepsilon}\rvert\,{\mathrm{d}}x \\&\le C(n,\,p,\,\beta_{0},\,\Lambda,\,K,\,\delta)\mu^{p-2-\beta_{0}}\mleft(\int_{B}\lvert \nabla u_{\varepsilon}-\zeta\rvert^{2(1+\beta_{0})}\,{\mathrm d}x \mright)^{1/2}\mleft(\int_{B}\lvert \nabla u_{\varepsilon}-\nabla v_{\varepsilon}\rvert^{2}\,{\mathrm d}x\mright)^{1/2}\\&\quad+C(n,\,q)F\rho^{\beta+\frac{n}{2}}\mleft(\int_{B}\lvert \nabla u_{\varepsilon}-\nabla v_{\varepsilon}\rvert^{2}\,{\mathrm d}x \mright)^{1/2}.
\end{align*}
As a result, we obtain
\[\fint_{B}\lvert \nabla u_{\varepsilon}-\nabla v_{\varepsilon}\rvert^{2}\,{\mathrm d}x\le C\mleft[\frac{1}{\mu^{2\beta_{0}}}\fint_{B}\lvert \nabla u_{\varepsilon}-\zeta\rvert^{2(1+\beta_{0})}\,{\mathrm d}x+F^{2}\mu^{2(2-p)}\rho^{2\beta}\mright],\]
where the constant $C\in(0,\,\infty)$ depends at most on $n$, $p$, $q$, $\beta_{0}$, $\lambda$, $\Lambda$, $K$, $M$, and $\delta$. By (\ref{Eq: mu>delta}) and (\ref{Eq: Zeta bound}), it is easy to get
\[\lvert \nabla u_{\varepsilon}-\zeta\rvert\le 2(\delta+\mu) \le 4\mu\quad \textrm{a.e. in }B_{\rho}(x_{0}).\]
Since \(\zeta\) clearly satisfies (\ref{Eq: zeta location}) by (\ref{Eq: Average assumption}) and (\ref{Eq: Zeta bound}), we are able to apply Lemma \ref{Lemma: Higher integrability}. Noting $\vartheta\le \beta_{0}$, we obtain
\begin{align*}
\fint_{B}\lvert \nabla u_{\varepsilon}-\nabla v_{\varepsilon}\rvert^{2}\,{\mathrm d}x&\le C\mleft[\frac{C(n,\,\vartheta)}{\mu^{2\vartheta}}\fint_{B}\lvert\nabla u_{\varepsilon}-\zeta\rvert^{2(1+\vartheta)}\,{\mathrm d}x+F^{2}\mu^{2(2-p)}\rho^{2\beta}\mright]\\&\le C\mleft[\frac{\Phi(x_{0},\,\rho)^{1+\vartheta}}{\mu^{2\vartheta}}+\frac{F^{2(1+\vartheta)}\rho^{2\beta(1+\vartheta)}}{\mu^{2\vartheta}}+F^{2}\mu^{2(2-p)}\rho^{2\beta}\mright]
\end{align*}
for some constant \(C\in(0,\,\infty)\) depending at most on $n$, $p$, $q$, $\beta_{0}$, $\lambda$, $\Lambda$, $K$, $M$, and $\delta$.
Recalling $0<\rho\le 1$ and $\delta<\mu<M$ by (\ref{Eq: mu>delta})--(\ref{Eq: Gradient bound in Growth estimate}), we finally obtain (\ref{Eq: Perturbation estimates}).
\end{proof}

\subsection{Energy estimates}\label{Subsect: Energy estimates}
We go back to the weak formulation given in Lemma \ref{Lemma: Weak formulations of approximated equations}, and deduce some energy estimates under suitable assumptions as in Proposition \ref{Prop: Perturbation result}.

In Section \ref{Subsect: Energy estimates}, we introduce a vector field \(G_{p,\,\varepsilon}\in C^{\infty}({\mathbb R}^{n};\,{\mathbb R}^{n})\) by
\begin{equation}\label{Eq: G-p-epsilon}
G_{p,\,\varepsilon}(z)\coloneqq g_{p,\,\varepsilon}(\lvert z\rvert^{2})z\quad \textrm{for }z\in{\mathbb R}^{n}\textrm{ with } g_{p,\,\varepsilon}(t)\coloneqq \mleft(\varepsilon^{2}+t\mright)^{(p-1)/2}\,(0\le t<\infty)
\end{equation}
for each \(\varepsilon\in(0,\,1)\).
By definition, it is clear that the mapping \(G_{p,\,\varepsilon}\colon{\mathbb R}^{n}\rightarrow{\mathbb R}^{n}\) is bijective. 
By direct computations, it is easy to notice that the Jacobian matrix $DG_{p,\,\varepsilon}$ satisfies
\[\mleft(\varepsilon^{2}+\lvert z\rvert^{2}\mright)^{(p-1)/2}\mathrm{id}\leqslant DG_{p,\,\varepsilon}(z)\leqslant p\mleft(\varepsilon^{2}+\lvert z\rvert^{2}\mright)^{(p-1)/2}\mathrm{id}\]
for all $\varepsilon\in(0,\,1)$ and $z\in{\mathbb R}^{n}$.
In particular, similarly to (\ref{Eq: Lp lower bound lemma case p>2}), we can find a constant \(c=c(p)\in(0,\,\infty)\) such that there holds
\[\mleft\langle G_{p,\,\varepsilon}(z_{1})-G_{p,\,\varepsilon}(z_{2})\mathrel{}\middle|\mathrel{}z_{1}-z_{2}\mright\rangle\ge c(p)\lvert z_{1}-z_{2}\rvert^{p+1}\quad \textrm{for all }z_{1},\,z_{2}\in {\mathbb R}^{n},\]
which clearly yields
\[\mleft\lvert G_{p,\,\varepsilon}(z_{1})-G_{p,\,\varepsilon}(z_{2}) \mright\rvert\ge c(p)\lvert z_{1}-z_{2}\rvert^{p}\quad \textrm{for all } z_{1},\,z_{2}\in{\mathbb R}^{n}.\]
By this estimate and \(G_{p,\,\varepsilon}(0)=0\), the inverse mapping \(G_{p,\,\varepsilon}^{-1}\) satisfies
\begin{equation}\label{Eq: Boundedness of inversed mapping G-p-epsilon}
\mleft\lvert G_{p,\,\varepsilon}^{-1}(w)\mright\rvert\le C(p)\lvert w\rvert^{1/p}\quad \textrm{for all }w\in {\mathbb R}^{n}
\end{equation}
with \(C=c(p)^{-1}\in(0,\,\infty)\).
Also, similarly to Lemma \ref{Lemma: Coerciveness keeps under convolution} \ref{Item 1/2: Monotonicity}, it is possible to find a constant \(C=C(p)\in(0,\,\infty)\) such that there holds
\[\langle G_{p,\,\varepsilon}(z_{1})-G_{p,\,\varepsilon}(z_{2})\mid z_{1}-z_{2} \rangle\ge C(p) \max\mleft\{\,\lvert z_{1}\rvert^{p-1},\,\lvert z_{2}\rvert^{p-1}\,\mright\} \lvert z_{1}-z_{2}\rvert^{2},\]
and therefore
\begin{equation}\label{Eq: L1 lower bound lemma on G-p-epsilon}
\mleft\lvert G_{p,\,\varepsilon}(z_{1})-G_{p,\,\varepsilon}(z_{2})\mright\rvert\ge C(p) \max\mleft\{\,\lvert z_{1}\rvert^{p-1},\,\lvert z_{2}\rvert^{p-1}\,\mright\} \lvert z_{1}-z_{2}\rvert
\end{equation}
for all \(z_{1},\,z_{2}\in{\mathbb R}^{n}\). The estimates (\ref{Eq: Boundedness of inversed mapping G-p-epsilon})--(\ref{Eq: L1 lower bound lemma on G-p-epsilon}) are used in Section \ref{Subsect: Energy estimates}.
\begin{lemma}\label{Lemma: Energy estimates}
Let \(u_{\varepsilon}\) be a weak solution to (\ref{Eq: Regularized equation}) in \(\Omega\). Assume that positive numbers $\delta$, $\varepsilon$, $\mu$, $F$, $M$, and an open ball $B_{\rho}(x_{0})\Subset\Omega$ satisfy (\ref{Eq: Range of delta-epsilon}), (\ref{Eq: Bound on external force term et al.}) and (\ref{Eq: mu>delta})--(\ref{Eq: Gradient bound in Growth estimate}). Then, for a fixed constant $\nu\in(0,\,1)$, we have
\begin{equation}\label{Eq: Energy estimate 1}
\fint_{B_{\tau \rho}(x_{0})}\mleft\lvert D\mleft[G_{p,\,\varepsilon}(\nabla u_{\varepsilon})\mright]\mright\rvert^{2}\,{\mathrm d}x\le \frac{C}{\tau^{n}}\mleft[\frac{1}{(1-\tau)^{2}\rho^{2}}+F^{2}\rho^{-\frac{2n}{q}}\mright]\mu^{2p},
\end{equation}
and
\begin{equation}\label{Eq: Energy estimate 2}
\frac{1}{\lvert B_{\tau\rho}(x_{0})\rvert}\int_{S_{\tau \rho,\,\mu,\,\nu}(x_{0})}\mleft\lvert D\mleft[G_{p,\,\varepsilon}(\nabla u_{\varepsilon})\mright]\mright\rvert^{2}\,{\mathrm d}x\le \frac{C}{\tau^{n}}\mleft[\frac{\nu}{(1-\tau)^{2}\rho^{2}}+\frac{F^{2}\rho^{-\frac{2n}{q}}}{\nu}\mright]\mu^{2p}
\end{equation}
for all $\tau\in(0,\,1)$.
Here \(G_{p,\,\varepsilon}\) is the mapping defined by (\ref{Eq: G-p-epsilon}), and the constants \(C\in(0,\,\infty)\) in (\ref{Eq: Energy estimate 1})--(\ref{Eq: Energy estimate 2}) depend at most on $n$, $p$, $q$, $\lambda$, $\Lambda$, $K$, and $\delta$.
\end{lemma}
\begin{proof}
We often write \(B=B_{\rho}(x_{0})\) for notational simplicity.
By (\ref{Eq: Hessian estimate for approximated E}) and direct calculations, we have
\begin{align}
&\mleft\lvert D\mleft[G_{p,\,\varepsilon}(\nabla u_{\varepsilon})\mright]\mright\rvert^{2}\nonumber\\&=\mleft\lvert g_{p,\,\varepsilon}\mleft(\lvert\nabla u_{\varepsilon}\rvert^{2}\mright)\nabla^{2}u_{\varepsilon}+2g_{p,\,\varepsilon}^{\prime}\mleft(\lvert\nabla u_{\varepsilon}\rvert^{2}\mright)(\nabla u_{\varepsilon}\otimes\nabla^{2}u_{\varepsilon}\nabla u_{\varepsilon})\mright\rvert^{2}\nonumber\\&\le 2V_{\varepsilon}^{2(p-1)}\lvert\nabla^{2} u_{\varepsilon}\rvert^{2}+2(p-1)^{2}V_{\varepsilon}^{2(p-3)}\lvert\nabla u_{\varepsilon}\otimes \nabla^{2}u_{\varepsilon}\nabla u_{\varepsilon} \rvert^{2}\nonumber\\&\le \gamma_{p}V_{\varepsilon}^{2(p-1)}\lvert \nabla^{2}u_{\varepsilon}\rvert^{2}\nonumber\\&\le\frac{\gamma_{p}}{\lambda}\sum_{j=1}^{n}\mleft\langle\nabla E_{\varepsilon}(\nabla u_{\varepsilon})\nabla\partial_{x_{j}}u_{\varepsilon}\mathrel{}\middle|\mathrel{}\nabla\partial_{x_{j}}u_{\varepsilon}\mright\rangle V_{\varepsilon}^{p}\quad \textrm{a.e. in }B\label{Eq: Energy estimate mid 0}
\end{align}
with \(\gamma_{p}\coloneqq 2(p^{2}-2p+2)>0\).
Noting this, we apply Lemma \ref{Lemma: Weak formulations of approximated equations} with
\[\psi(t)=t^{p}{\tilde\psi}(t)\quad (0\le t<\infty).\]
Here \({\tilde\psi}\) is a convex function that is to be chosen later as either of the following;
\begin{equation}\label{Eq: Energy estimate choice 1}
{\tilde\psi}(t)\equiv 1,
\end{equation}
\begin{equation}\label{Eq: Energy estimate choice 2}
{\tilde\psi}(t)=(t-\delta-k)_{+}^{2}\,\quad (0\le t<\infty)\quad \textrm{for some constant }k>0. 
\end{equation}
We choose a cutoff function \(\eta\in C_{c}^{1}(B_{\rho}(x_{0}))\) such that
\begin{equation}\label{Eq: Choice of cutoff function}
\eta\equiv 1\quad\textrm{on $B_{\tau \rho}(x_{0})$}\quad\textrm{and}\quad\lvert\nabla\eta\rvert\le \frac{2}{(1-\tau)\rho}\quad\textrm{in $B_{\rho}(x_{0})$},
\end{equation}
and set \(\zeta\coloneqq \eta^{2}\).
Then, we have 
\begin{align*}
&{\mathbf L}_{1}+{\mathbf L}_{2}\\
&
\coloneqq \int_{B}\sum\limits_{j=1}^{n}\mleft\langle\nabla^{2}E_{\varepsilon}(\nabla u_{\varepsilon})\nabla \partial_{x_{j}}u_{\varepsilon}\mathrel{}\middle|\mathrel{}\nabla\partial_{x_{j}}u_{\varepsilon}\mright\rangle V_{\varepsilon}^{p}{\tilde\psi}(V_{\varepsilon})\eta^{2}\,{\mathrm d}x\\&\quad
+
\int_{B}\mleft\langle\nabla^{2}E_{\varepsilon}(\nabla u_{\varepsilon})\nabla V_{\varepsilon}\mathrel{}\middle|\mathrel{}\nabla V_{\varepsilon}\mright\rangle \psi^{\prime}(V_{\varepsilon})V_{\varepsilon}\eta^{2}\,{\mathrm d}x
\\&\le 4\int_{B}\mleft\lvert \mleft\langle\nabla^{2}E_{\varepsilon}(\nabla u_{\varepsilon})\nabla V_{\varepsilon}\mathrel{}\middle|\mathrel{}\nabla\eta\mright\rangle\mright\rvert V_{\varepsilon}\psi(V_{\varepsilon})\eta\,{\mathrm d}x\\&\quad+
\frac{n}{\lambda}\int_{B}\lvert f_{\varepsilon}\rvert^{2}V_{\varepsilon}^{2-p}\mleft({\psi}(V_{\varepsilon})+V_{\varepsilon}\psi^{\prime}(V_{\varepsilon})\mright)\eta^{2}\,{\mathrm d}x
\\& \quad\quad+4\int_{B}\lvert f_{\varepsilon}\rvert\lvert \nabla \eta\rvert V_{\varepsilon}\psi(V_{\varepsilon})\,{\mathrm d}x
\\&\eqqcolon {\mathbf R}_{1}+{\mathbf R}_{2}+{\mathbf R}_{3}.
\end{align*}
For \({\mathbf R}_{1}\), we use the Cauchy--Schwarz inequality
\[\mleft\lvert \mleft\langle\nabla^{2}E_{\varepsilon}(\nabla u_{\varepsilon})\nabla V_{\varepsilon}\mathrel{}\middle|\mathrel{}\nabla\eta\mright\rangle\mright\rvert\le \sqrt{\mleft\langle\nabla^{2}E_{\varepsilon}(\nabla u_{\varepsilon})\nabla V_{\varepsilon}\mathrel{}\middle|\mathrel{}\nabla V_{\varepsilon}\mright\rangle}\,\cdot\,\sqrt{\mleft\langle\nabla^{2}E_{\varepsilon}(\nabla u_{\varepsilon})\nabla \eta\mathrel{}\middle|\mathrel{}\nabla\eta\mright\rangle},\]
which holds true by (\ref{Eq: Hessian estimate for approximated E}). Combining with Young's inequality, we have
\begin{align*}
{\mathbf R}_{1}&\le {\mathbf L}_{2}+4\int_{B}\mleft\langle\nabla^{2}E_{\varepsilon}(\nabla u_{\varepsilon})\nabla\eta\mathrel{}\middle|\mathrel{}\nabla\eta\mright\rangle \frac{V_{\varepsilon}\psi(V_{\varepsilon})^{2}}{\psi^{\prime}(V_{\varepsilon})}\,\mathrm{d}x
\end{align*}
For \({\mathbf R}_{3}\), we apply Young's inequality to get
\[{\mathbf R}_{3}\le 2\int_{B}\lvert f_{\varepsilon}\rvert^{2}\psi^{\prime}(V_{\varepsilon}) V_{\varepsilon}^{3-p}\eta^{2}\,{\mathrm d}x+2 \int_{B}\lvert \nabla\eta\rvert^{2}\frac{V_{\varepsilon}\psi(V_{\varepsilon})^{2}}{\psi^{\prime}(V_{\varepsilon})}V_{\varepsilon}^{p-2}\,{\mathrm d}x.\]
Therefore, by (\ref{Eq: Hessian estimate for approximated E}), we obtain
\begin{align}
{\mathbf L}_{1}&\le\int_{B}\lvert\nabla\eta\rvert^{2} \mleft[(4\Lambda+2)V_{\varepsilon}^{p-2}+4KV_{\varepsilon}^{-1}\mright]\frac{V_{\varepsilon}\psi(V_{\varepsilon})^{2}}{\psi^{\prime}(V_{\varepsilon})}\,{\mathrm d}x\nonumber\\&\quad +\mleft(2+\frac{n}{\lambda}\mright)\int_{B}\lvert f_{\varepsilon}\rvert^{2}\eta^{2}\frac{\psi(V_{\varepsilon})+V_{\varepsilon}\psi^{\prime}(V_{\varepsilon})}{V_{\varepsilon}^{p-2}}\,{\mathrm d}x.\label{Eq: Energy estimate mid 1}
\end{align}
Here we note that the assumptions (\ref{Eq: mu>delta})--(\ref{Eq: Gradient bound in Growth estimate}) enable us to have \(V_{\varepsilon}\le \delta+\mu\le 2\mu\) a.e. in \(B=B_{\rho}(x_{0})\), and \(\mu^{l}=\mu^{l-2p}\cdot\mu^{2p}\le\delta^{l-2p}\mu^{2p}\) for every \(0<l<2p\).
Hence it follows that
\begin{align}
&\mleft[(4\Lambda+2)V_{\varepsilon}^{p-2}+4KV_{\varepsilon}^{-1}\mright]\frac{V_{\varepsilon}\psi(V_{\varepsilon})^{2}}{\psi^{\prime}(V_{\varepsilon})}\nonumber\\&\le C(\Lambda,\,K)\mleft[V_{\varepsilon}^{p-2}+V_{\varepsilon}^{-1}\mright]\frac{V_{\varepsilon}^{p+2}{\tilde\psi}(V_{\varepsilon})^{2}}{p{\tilde \psi}(V_{\varepsilon})+V_{\varepsilon}{\tilde\psi}^{\prime}(V_{\varepsilon})}\nonumber\\&\le C(\Lambda,\,K)\frac{\mleft(1+\delta^{1-p}\mright)\mu^{2p}{\tilde\psi}(V_{\varepsilon})^{2}}{p{\tilde \psi}(V_{\varepsilon})+V_{\varepsilon}{\tilde\psi}^{\prime}(V_{\varepsilon})}\label{Eq: Energy estimate mid 2}\quad \textrm{a.e. in }B,
\end{align}
and
\begin{align}
\frac{\psi(V_{\varepsilon})+V_{\varepsilon}\psi^{\prime}(V_{\varepsilon})}{V_{\varepsilon}^{p-2}}&=(p+1)V_{\varepsilon}^{2}{\tilde\psi}(V_{\varepsilon})+V_{\varepsilon}^{3}{\tilde\psi}^{\prime}(V_{\varepsilon})\nonumber\\&\le 4(p+1)\delta^{2(1-p)}\mu^{2p}\mleft[{\tilde \psi}(V_{\varepsilon})+V_{\varepsilon}{\tilde\psi}^{\prime}(V_{\varepsilon})\mright]\quad \textrm{a.e. in }B.\label{Eq: Energy estimate mid 3}
\end{align}
By (\ref{Eq: Energy estimate mid 0}) and (\ref{Eq: Energy estimate mid 1})--(\ref{Eq: Energy estimate mid 3}), we obtain
\begin{align}
&\int_{B}\mleft\lvert D\mleft[G_{p,\,\varepsilon}(\nabla u_{\varepsilon})\mright]\mright\rvert^{2}\eta^{2}{\tilde \psi}(V_{\varepsilon})\,{\mathrm d}x \le\frac{C_{p}}{\lambda}{\mathbf L}_{1}\nonumber\\&\le C\mu^{2p}\mleft[\int_{B}\lvert\nabla\eta\rvert^{2}\frac{{\tilde\psi}(V_{\varepsilon})^{2}}{p{\tilde \psi}(V_{\varepsilon})+V_{\varepsilon}{\tilde\psi}^{\prime}(V_{\varepsilon})}\,{\mathrm d}x+\int_{B}\lvert f_{\varepsilon}\rvert^{2}\eta^{2} \mleft[{\tilde \psi}(V_{\varepsilon})+V_{\varepsilon}{\tilde\psi}^{\prime}(V_{\varepsilon})\mright]\,{\mathrm d}x\,\mright]\label{Eq: Energy estimate mid fin}
\end{align}
with \(C=C(n,\,p,\,\lambda,\,\Lambda,\,K,\,\delta)\in(0,\,\infty)\). From (\ref{Eq: Energy estimate mid fin}) we will deduce (\ref{Eq: Energy estimate 1})--(\ref{Eq: Energy estimate 2}).

We first apply (\ref{Eq: Energy estimate mid fin}) with \({\tilde\psi}\) and \(\eta\) given by (\ref{Eq: Energy estimate choice 1}) and (\ref{Eq: Choice of cutoff function}) respectively. Then, we have
\begin{align*}
&\int_{B_{\tau\rho}(x_{0})}\mleft\lvert D\mleft[G_{p,\,\varepsilon}(\nabla u_{\varepsilon})\mright]\mright\rvert^{2}\,{\mathrm d}x\\&\le  C(n,\,p,\,\lambda,\,\Lambda,\,K,\,\delta)\mu^{2p}\mleft[\frac{1}{p}\int_{B}\lvert \nabla\eta\rvert^{2}\,{\mathrm d}x+\int_{B}\lvert f_{\varepsilon}\rvert^{2}\eta^{2}{\mathrm d}x\mright]\\&\le C(n,\,p,\,q,\,\lambda,\,\Lambda,\,K,\,\delta)\mu^{2p}\mleft[\frac{1}{(1-\tau)^{2}\rho^{2}}+F^{2}\rho^{-\frac{2n}{q}}\mright]\lvert B_{\rho}(x_{0})\rvert
\end{align*}
by H\"{o}lder's inequality. We note \(\lvert B_{\tau\rho}(x_{0})\rvert=\tau^{n}\lvert B_{\rho}(x_{0})\rvert\), and this yields (\ref{Eq: Energy estimate 1}).

Next, we let \({\tilde\psi}\) and \(\eta\) satisfy (\ref{Eq: Energy estimate choice 2})--(\ref{Eq: Choice of cutoff function}). Here we determine the constant \(k>0\) by \[k\coloneqq (1-2\nu)\mu\ge\frac{\mu}{2}>0.\]
Then, from (\ref{Eq: Gradient bound in Growth estimate}), it follows that \((V_{\varepsilon}-\delta-k)_{+}\le \mu-k=2\nu\mu\) a.e. in \(B=B_{\rho}(x_{0})\). 
Hence by direct computations, we can check that
\[\mleft\{
\begin{array}{ccccccc}
\displaystyle\frac{{\tilde\psi}(V_{\varepsilon})^{2}}{p{\tilde \psi}(V_{\varepsilon})+V_{\varepsilon}{\tilde\psi}^{\prime}(V_{\varepsilon})}&=&\displaystyle\frac{(V_{\varepsilon}-\delta-k)_{+}^{3}}{p(V_{\varepsilon}-\delta-k)_{+}+2V_{\varepsilon}}&\le& \displaystyle\frac{(2\nu\mu)^{3}}{0+2(\delta+k)}&\le& 8\nu^{3}\mu^{2},\\
{\tilde\psi}(V_{\varepsilon})+V_{\varepsilon}{\tilde\psi}^{\prime}(V_{\varepsilon})&\le& 3V_{\varepsilon}(V_{\varepsilon}-\delta-k)_{+}&\le& 3\cdot 2\mu\cdot 2\nu\mu&=&12\nu\mu^{2},
\end{array}
\mright.\]
a.e. in $B$.
Also, there holds \(V_{\varepsilon}-\delta-k>(1-\nu)\mu-k=\nu\mu>0\) in \(S_{\tau\rho,\,\mu,\,\nu}(x_{0})\). These results and (\ref{Eq: Energy estimate mid fin}) enable us to compute
\begin{align*}
&\nu^{2}\mu^{2}\int_{S_{\tau\rho,\,\mu,\,\nu}(x_{0})}\mleft\lvert D\mleft[G_{p,\,\varepsilon}(\nabla u_{\varepsilon})\mright]\mright\rvert^{2}\,{\mathrm d}x \\&\le \int_{B}\eta^{2}\mleft\lvert D\mleft[G_{p,\,\varepsilon}(\nabla u_{\varepsilon})\mright]\mright\rvert^{2}{\tilde\psi}(V_{\varepsilon})\,{\mathrm d}x\\&\le C\mu^{2p}\mleft[8\nu^{3}\mu^{2}\int_{B}\lvert\nabla\eta\rvert^{2}\,{\mathrm d}x+12\nu\mu^{2} \int_{B}\lvert f_{\varepsilon}\rvert^{2}\eta^{2}\,{\mathrm d}x\mright]\\&\le C\nu^{2}\mu^{2p+2}\mleft[\frac{\nu}{(1-\tau)^{2}\rho^{2}}+\frac{F^{2}\rho^{-2n/q}}{\nu} \mright]\lvert B_{\rho}(x_{0})\rvert
\end{align*}
for some constant \(C=C(n,\,p,\,q,\,\lambda,\,\Lambda,\,K,\,\delta)\in(0,\,\infty)\).
From this and \(\lvert B_{\tau\rho}(x_{0})\rvert=\tau^{n}\lvert B_{\rho}(x_{0})\rvert\), (\ref{Eq: Energy estimate 2}) follows.
\end{proof}
From Lemma \ref{Lemma: Energy estimates} and (\ref{Eq: Boundedness of inversed mapping G-p-epsilon})--(\ref{Eq: L1 lower bound lemma on G-p-epsilon}), we prove Lemma \ref{Lemma: Excess decay estimate far from 0} below. This result is used later in the next Section \ref{Subsect: Average integral estimates}.
\begin{lemma}\label{Lemma: Excess decay estimate far from 0}
Let \(u_{\varepsilon}\) be a weak solution to (\ref{Eq: Regularized equation}) in \(\Omega\). Assume that positive numbers $\delta$, $\varepsilon$, $\mu$, $F$, $M$, and an open ball $B_{\rho}(x_{0})\Subset\Omega$ satisfy (\ref{Eq: Range of delta-epsilon}), (\ref{Eq: Bound on external force term et al.}) and (\ref{Eq: mu>delta})--(\ref{Eq: Gradient bound in Growth estimate}).
If (\ref{Eq: Levelset assumption 2}) holds for some \(\nu\in(0,\,1/6)\), then for all \(\tau\in(0,\,1)\), we have
\begin{equation}\label{Eq: Excess decay estimate far from 0}
\Phi(x_{0},\,\tau\rho)\le \frac{C_{\dagger}\mu^{2}}{\tau^{n}}\mleft[\frac{\nu^{2/n}}{(1-\tau)^{2}}+\frac{F^{2}}{\nu} \rho^{2\beta}\mright],
\end{equation}
where the constant \(C_{\dagger}\in(0,\,\infty)\) depends at most on $n$, $p$, $q$, $\lambda$, $\Lambda$, $K$, and $\delta$.
\end{lemma}
\begin{proof}
We first set \(\zeta\coloneqq G_{p,\,\varepsilon}^{-1}\mleft((G_{p,\,\varepsilon}(\nabla u_{\varepsilon}))_{\tau\rho}\mright)\in{\mathbb R}^{n}\), so that there holds
\[\fint_{B_{\tau\rho}}\mleft[G_{p,\,\varepsilon}(\nabla u_{\varepsilon})-G_{p,\,\varepsilon}(\zeta)\mright]\,{\mathrm d}x=0.\]
In particular, we can use the Poincar\'{e}--Sobolev inequality
\[\fint_{B_{\tau\rho}(x_{0})}\lvert G_{p,\,\varepsilon}(\nabla u_{\varepsilon})-G_{p,\,\varepsilon}(\zeta)\rvert^{2}\,{\mathrm d}x\le C(n)(\tau\rho)^{2}\mleft(\fint_{B_{\tau\rho}(x_{0})}\mleft\lvert D\mleft[G_{p,\,\varepsilon}(\nabla u_{\varepsilon})\mright] \mright\rvert^{\frac{2n}{n+2}}\,{\mathrm d}x\mright)^{\frac{n+2}{n}}.\]
By (\ref{Eq: mu>delta})--(\ref{Eq: Gradient bound in Growth estimate}), it is easy to get
\[\lvert G_{p,\,\varepsilon}(\nabla u_{\varepsilon})\rvert=V_{\varepsilon}^{p-1}\lvert\nabla u_{\varepsilon}\rvert\le V_{\varepsilon}^{p}\le(\delta+\mu)^{p}\le (2\mu)^{p}\quad  \textrm{a.e. in }B_{\rho}(x_{0})\]
This inequality and (\ref{Eq: Boundedness of inversed mapping G-p-epsilon}) yield
\[\lvert \zeta\rvert\le C(p)\lvert G_{p,\,\varepsilon}(\zeta)\rvert^{1/p}\le C(p)\esssup_{B_{\rho}(x_{0})}\,\lvert G_{p,\,\varepsilon}(\nabla u_{\varepsilon})\rvert^{1/p} \le C(p)\mu.\]
To prove (\ref{Eq: Excess decay estimate far from 0}), we fix arbitrary $\tau\in(0,\,1)$ and set
\[\mleft\{\begin{array}{rcl}\mathbf{I}&\coloneqq& \displaystyle\frac{1}{\lvert B_{\tau\rho}(x_{0})\rvert}\displaystyle\int_{B_{\tau\rho}(x_{0})\setminus S_{\rho,\,\mu,\,\nu}(x_{0})}\lvert \nabla u_{\varepsilon}-\zeta\rvert^{2}\,{\mathrm d}x,\\
\mathbf{II}&\coloneqq& \displaystyle\frac{1}{\lvert B_{\tau\rho}(x_{0})\rvert}\displaystyle\int_{B_{\tau\rho}(x_{0})\cap S_{\rho,\,\mu,\,\nu}(x_{0})}\lvert \nabla u_{\varepsilon}-\zeta\rvert^{2}\,{\mathrm d}x.
\end{array} \mright.\]
For \(\mathbf{I}\), we use (\ref{Eq: mu>delta})--(\ref{Eq: Gradient bound in Growth estimate}) and the triangle inequality to have 
\[\lvert \nabla u_{\varepsilon}-\zeta\rvert\le (\delta+\mu)+C(p)\mu\le C(p)\mu\quad \textrm{a.e. in }B_{\rho}(x_{0}).\]
Combining with (\ref{Eq: Levelset assumption 2}), we get
\[\mathbf{I}\le \frac{\lvert B_{\rho}(x_{0})\setminus S_{\rho,\,\mu,\,\nu}(x_{0})\rvert}{\lvert B_{\tau\rho}(x_{0})\rvert}\cdot C(p)\mu^{2}\le\frac{C(p)\nu\mu^{2}}{\tau^{n}}.\]
Here we have used \(\lvert B_{\tau\rho}(x_{0})\rvert=\tau^{n}\lvert B_{\rho}(x_{0})\rvert\).
Before estimating $\mathbf{II}$, we compute 
\begin{equation}\label{Eq: Lower bounds over E}
(1-\nu)\mu<V_{\varepsilon}-\delta\le \lvert \nabla u_{\varepsilon}\rvert+\varepsilon-\delta\le \lvert \nabla u_{\varepsilon}\rvert-\frac{7}{8}\delta\quad \textrm{a.e. in }S_{\rho,\,\mu,\,\nu}(x_{0}),
\end{equation}
which immediately follows from (\ref{Eq: Range of delta-epsilon}) and the definition of \(S_{\rho,\,\mu,\,\nu}(x_{0})\).
The assumption \(\nu\in(0,\,1/6)\) and (\ref{Eq: Lower bounds over E}) clearly yield
\[\frac{5}{6}\mu\le \lvert \nabla u_{\varepsilon}\rvert\quad \textrm{a.e. in }S_{\rho,\,\mu,\,\nu}(x_{0}).\]
With this in mind, we use (\ref{Eq: L1 lower bound lemma on G-p-epsilon}) and the Poincar\'{e}--Sobolev inequality to compute
\begin{align*}
\mathbf{II}&\le \frac{C(p)}{\mu^{2(p-1)}\lvert B_{\tau\rho}(x_{0})\rvert}\int_{B_{\tau\rho}(x_{0})\cap S_{\rho,\,\mu,\,\nu}(x_{0})}\lvert G_{p,\,\varepsilon}(\nabla u_{\varepsilon})-G_{p,\,\varepsilon}(\zeta)\rvert^{2}\,{\mathrm d}x\\&\le \frac{C(p)}{\mu^{2(p-1)}}(\tau\rho)^{2}\mleft(\fint_{B_{\tau\rho}(x_{0})}\mleft\lvert D\mleft[G_{p,\,\varepsilon}(\nabla u_{\varepsilon}) \mright]\mright\rvert^{\frac{2n}{n+2}} \,{\mathrm d}x\mright)^{\frac{n+2}{n}}\eqqcolon \frac{C(p)}{\mu^{2(p-1)}}(\tau\rho)^{2}\cdot \mathbf{III}^{1+2/n}.
\end{align*}
Noting that there holds \(B_{\tau\rho}(x_{0})\setminus S_{\rho,\,\mu,\,\nu}(x_{0})=B_{\tau\rho}(x_{0})\setminus S_{\tau\rho,\,\mu,\,\nu}(x_{0})\) as measurable sets, we decompose \(\mathbf{III}=\mathbf{III}_{1}+\mathbf{III}_{2}\) with
\[\mleft\{\begin{array}{rcl}\mathbf{III}_{1}&\coloneqq& \displaystyle\frac{1}{\lvert B_{\tau\rho}(x_{0})\rvert}\displaystyle\int_{B_{\tau\rho}(x_{0})\setminus S_{\rho,\,\mu,\,\nu}(x_{0})}\mleft\lvert D\mleft[G_{p,\,\varepsilon}(\nabla u_{\varepsilon})\mright] \mright\rvert^{\frac{2n}{n+2}}\,{\mathrm d}x,\\
\mathbf{III}_{2}&\coloneqq& \displaystyle\frac{1}{\lvert B_{\tau\rho}(x_{0})\rvert}\displaystyle\int_{S_{\tau\rho,\,\mu,\,\nu}(x_{0})}\mleft\lvert D\mleft[G_{p,\,\varepsilon}(\nabla u_{\varepsilon})\mright] \mright\rvert^{\frac{2n}{n+2}}\,{\mathrm d}x.
\end{array} \mright.\]
For \(\mathbf{III}_{1}\), by applying H\"{o}lder's inequality and (\ref{Eq: Energy estimate 1}), we compute
\begin{align*}
\mathbf{III}_{1}^{1+2/n}&\le \mleft[\frac{\lvert B_{\rho}(x_{0})\setminus S_{\rho,\,\mu,\,\nu}(x_{0})\rvert }{\lvert B_{\tau\rho}(x_{0})\rvert }\mright]^{2/n} \fint_{B_{\tau\rho}}\mleft\lvert D\mleft[G_{p,\,\varepsilon}(\nabla u_{\varepsilon})\mright] \mright\rvert^{2}\,{\mathrm d}x\\&\le C(n,\,p,\,q,\,\lambda,\,\Lambda,\,K,\,\delta)\frac{\nu^{2/n}\mu^{2p}}{\tau^{n+2}}\mleft[\frac{1}{(1-\tau)^{2}\rho^{2}}+F^{2}\rho^{-\frac{2n}{q}}\mright].
\end{align*}
For the remained integral \(\mathbf{III}_{2}\), we similarly use H\"{o}lder's inequality and (\ref{Eq: Energy estimate 2}) to have
\begin{align*}
\mathbf{III}_{2}^{1+2/n}&\le \mleft[\frac{\lvert S_{\tau\rho,\,\mu,\,\nu}(x_{0})\rvert}{\lvert B_{\tau\rho}(x_{0})\rvert}\mright]^{2/n}\frac{1}{\lvert B_{\tau\rho}(x_{0})\rvert} \int_{E_{\tau\rho,\,\mu,\,\nu}(x_{0})}\mleft\lvert D\mleft[G_{p,\,\varepsilon}(\nabla u_{\varepsilon})\mright]\mright\rvert^{2}\,{\mathrm d}x\\&\le C(n,\,p,\,q,\,\lambda,\,\Lambda,\,K,\,\delta)\frac{\mu^{2p}}{\tau^{n}}\mleft[\frac{\nu}{(1-\tau)^{2}\rho^{2}}+\frac{F^{2}\rho^{-\frac{2n}{q}}}{\nu}\mright].
\end{align*}

Finally, we apply (\ref{Eq: Average value minimizes L2 oscillation}) to obtain
\begin{align*}
\Phi(x_{0},\,\tau\rho)&\le \fint_{B_{\tau\rho}(x_{0})}\lvert \nabla u_{\varepsilon}-\zeta\rvert^{2}\,{\mathrm d}x=\mathbf{I}+\mathbf{II}\\&\le \frac{C(p)\nu\mu^{2}}{\tau^{n}}+\frac{C(n,\,p)}{\mu^{2p-2}}(\tau\rho)^{2}\mleft(\mathbf{III}_{1}^{1+2/n}+\mathbf{III}_{2}^{1+2/n}\mright)\\&\le \frac{C(n,\,p,\,q,\,\lambda,\,\Lambda,\,K,\,\delta)\mu^{2}}{\tau^{n}}\mleft[\nu+\frac{\nu\tau^{2}+\nu^{2/n}}{(1-\tau)^{2}}+\frac{\tau^{2}+\nu^{1+2/n}}{\nu}F^{2}\rho^{2\beta}\mright].
\end{align*}
Recalling \(0<\tau<1\) and \(0<\nu<1/6\), we are able to find \(C_{\dagger}\in(0,\,\infty)\) such that (\ref{Eq: Excess decay estimate far from 0}) holds.
\end{proof}
\subsection{Campanato-type decay estimates by shrinking arguments}\label{Subsect: Average integral estimates}
The aim of Section \ref{Subsect: Average integral estimates} is to give the proof of Proposition \ref{Prop: Perturbation result} by standard shrinking methods. A significant point therein is to verify that the average integral $(\nabla u_{\varepsilon})_{x_{0},\,r}$ never vanishes even when the radius $r$ tends to zero. We will justify this expectation by similar computations as in Lemma \ref{Lemma: Average lemma from Energy estimates}.

For preliminaries, we deduce Lemmata \ref{Lemma: Average lemma from Perturbaltion}--\ref{Lemma: Average lemma from Energy estimates}, which are key tools to make our shrinking arguments successful. There we use some energy estimates from Lemmata \ref{Lemma: Excess decay estimate 1} and \ref{Lemma: Excess decay estimate far from 0} in Sections \ref{Subsect: Perturbation outside facets}--\ref{Subsect: Energy estimates}.
\begin{lemma}\label{Lemma: Average lemma from Perturbaltion}
Let \(u_{\varepsilon}\) be a weak solution to (\ref{Eq: Regularized equation}) in \(\Omega\). Assume that positive numbers $\delta$, $\varepsilon$, $\mu$, $F$, $M$, and an open ball $B_{\rho}(x_{0})\Subset\Omega$ satisfy (\ref{Eq: Range of delta-epsilon}), (\ref{Eq: Bound on external force term et al.}) and (\ref{Eq: mu>delta})--(\ref{Eq: Gradient bound in Growth estimate}). Assume that there hold (\ref{Eq: Average assumption}) and
\begin{equation}\label{Eq: Average lemma assumption in Perturbation}
\Phi(x_{0},\,\rho)\le \tau^{\frac{n+2}{\vartheta}}\mu^{2}
\end{equation}
for some $\tau\in(0,\,1/2)$. Here \(\vartheta\) is the constant given in Lemma \ref{Lemma: Higher integrability}.
Then, we have
\begin{equation}\label{Eq: Quantitative excess decay estimate}
\Phi(x_{0},\,\tau\rho)\le C_{\ast}\mleft[\tau^{2}\Phi(x_{0},\,\rho)+\frac{\rho^{2\beta}}{\tau^{n}}\mu^{2}\mright].
\end{equation}
Here the constant \(C_{\ast}\in(0,\,\infty)\) depends at most on $n$, $p$, $q$, $\beta_{0}$, $\lambda$, $\Lambda$, $K$, $F$, $M$, and $\delta$,
\end{lemma}
Before the proof, we recall that there holds
\begin{equation}\label{Eq: Average value minimizes L2 oscillation}
\fint_{B_{r}(x_{0})}\mleft\lvert f(x)-(f)_{x_{0},\,r}\mright\rvert^{2}\,{\mathrm d}x=\min_{\zeta\in{\mathbb R}^{m}}\fint_{B_{r}(x_{0})}\lvert f(x)-\zeta\rvert^{2}\,{\mathrm d}x
\end{equation}
for all \(f\in L^{2}(B_{r}(x_{0});\,{\mathbb R}^{m})\). This is easy to deduce by considering a smooth convex function $g\in C^{\infty}({\mathbb R}^{m};\,{\mathbb R})$ defined by 
\begin{align*}
g(\zeta)&\coloneqq \int_{B_{r}(x_{0})}\lvert f(x)-\zeta\rvert^{2}\,{\mathrm d}x\\&=\lVert f\rVert_{L^{2}(B_{r}(x_{0}))}^{2}-2\lvert B_{r}(x_{0})\rvert\langle (f)_{x_{0},\,r}\mid \zeta\rangle+\lvert B_{r}(x_{0})\rvert\lvert \zeta\rvert^{2},
\end{align*}
which clearly satisfies \(\nabla g((f)_{x_{0},\,r})=0\).
\begin{proof}
Let \(v_{\varepsilon}\in u_{\varepsilon}+W_{0}^{1,\,2}(B_{\rho/2}(x_{0}))\) be the unique solution of (\ref{Eq: Dirichlet boundary Poisson problem}). We first apply (\ref{Eq: Average value minimizes L2 oscillation}) to get
\begin{align*}
\Phi(x_{0},\,\tau\rho)&\le\fint_{B_{\tau\rho}(x_{0})}\mleft\lvert \nabla u_{\varepsilon}-(\nabla v_{\varepsilon})_{x_{0},\,\tau\rho}\mright\rvert^{2}\,{\mathrm d}x \\&\le \frac{2}{(2\tau)^{n}}\fint_{B_{\rho/2}(x_{0})}\mleft\lvert\nabla u_{\varepsilon}-\nabla v_{\varepsilon}\mright\rvert^{2}\,{\mathrm d}x +2\fint_{B_{\tau\rho}(x_{0})}\mleft\lvert\nabla v_{\varepsilon}-(\nabla v_{\varepsilon})_{x_{0},\,\tau\rho} \mright\rvert^{2}\,{\mathrm d}x
\end{align*}
For the second average integral, we use (\ref{Eq: Perturbation estimates}) and (\ref{Eq: Average value minimizes L2 oscillation}) to compute
\begin{align*}
&\fint_{B_{\tau\rho}(x_{0})}\mleft\lvert\nabla v_{\varepsilon}-(\nabla v_{\varepsilon})_{x_{0},\,\tau\rho} \mright\rvert^{2}\,{\mathrm d}x\\&\le C\tau^{2}\fint_{B_{\rho/2}(x_{0})}\mleft\lvert\nabla v_{\varepsilon}-(\nabla v_{\varepsilon})_{x_{0},\,\rho/2}\mright\rvert^{2}\,{\mathrm d}x\le C\tau^{2}\fint_{B_{\rho/2}(x_{0})}\mleft\lvert\nabla v_{\varepsilon}-(\nabla u_{\varepsilon})_{x_{0},\,\rho}\mright\rvert^{2}\,{\mathrm d}x\\&\le C\mleft[\tau^{2}\fint_{B_{\rho/2}(x_{0})}\mleft\lvert\nabla u_{\varepsilon}-\nabla v_{\varepsilon}\mright\rvert^{2}\,{\mathrm d}x+2^{n}\tau^{2}\fint_{B_{\rho}(x_{0})}\mleft\lvert\nabla u_{\varepsilon}-(\nabla u_{\varepsilon})_{x_{0},\,\rho} \mright\rvert^{2}\,{\mathrm d}x  \mright]\\&\le C\mleft[\tau^{2}\fint_{B_{\rho/2}(x_{0})}\mleft\lvert\nabla u_{\varepsilon}-\nabla v_{\varepsilon}\mright\rvert^{2}\,{\mathrm d}x+\tau^{2}\Phi(x_{0},\,\rho)\mright]
\end{align*}
with the constant \(C\in(0,\,\infty)\) depending at most on $n$, $p$, $q$, $\beta_{0}$, $\lambda$, $\Lambda$, $K$, $M$, and $\delta$.
We use (\ref{Eq: Perturbation estimates}) and (\ref{Eq: Average lemma assumption in Perturbation}) to compute
\begin{align*}
\Phi(x_{0},\,\tau\rho)&\le C\mleft[\frac{1}{\tau^{n}}\fint_{B_{\rho/2}(x_{0})}\lvert\nabla u_{\varepsilon}-\nabla v_{\varepsilon}\rvert^{2}\,{\mathrm d}x+\tau^{2}\Phi(x_{0},\,\rho)\mright]\\&\le C\mleft\{ \mleft[\frac{\Phi(x_{0},\,\rho)}{\mu^{2}}\mright]^{\vartheta}\cdot\frac{\Phi(x_{0},\,\rho)}{\tau^{n}}+\frac{F^{2}+F^{2(1+\vartheta)}}{\tau^{n}}\rho^{2\beta}+\tau^{2}\Phi(x_{0},\,\rho)\mright\}\\&\le C\mleft[\tau^{2}\Phi(x_{0},\,\rho)+\mleft(F^{2}+F^{2(1+\vartheta)}\mright)\cdot\frac{\rho^{2\beta}}{\tau^{n}}\cdot\mleft(\frac{\mu}{\delta}\mright)^{2}\mright]\\&\le C_{\ast}(n,\,p,\,q,\,\beta_{0},\,\lambda,\,\Lambda,\,K,\,F,\,M,\,\delta)\mleft[\tau^{2}\Phi(x_{0},\,\rho)+\frac{\rho^{2\beta}}{\tau^{n}}\mu^{2}\mright].
\end{align*}
Here we have used (\ref{Eq: mu>delta}) to obtain (\ref{Eq: Quantitative excess decay estimate}).
\end{proof}
\begin{lemma}\label{Lemma: Average lemma from Energy estimates}
Let \(u_{\varepsilon}\) be a weak solution to (\ref{Eq: Regularized equation}) in \(\Omega\). Assume that positive numbers $\delta$, $\varepsilon$, $\mu$, $F$, $M$, and an open ball $B_{\rho}(x_{0})\Subset\Omega$ satisfy (\ref{Eq: Range of delta-epsilon}), (\ref{Eq: Bound on external force term et al.}) and (\ref{Eq: mu>delta})--(\ref{Eq: Gradient bound in Growth estimate}).
Then, for each fixed \(\theta\in(0,\,9/64)\), there exist numbers \(\nu\in(0,\,1/6)\) and \({\hat\rho}\in(0,\,1)\), which depend at most on $n$, $p$, $q$, $\lambda$, $\Lambda$, $K$, $F$, $M$, $\delta$, and $\theta$, but are independent of $\varepsilon$, such that when \(0<\rho<{\hat\rho}\) and (\ref{Eq: Levelset assumption 2}) hold, we have
\begin{equation}\label{Eq: Average lemma result 1 in Energy estimates}
\mleft\lvert(\nabla u_{\varepsilon})_{x_{0},\,\rho}\mright\rvert\ge \delta+\frac{\mu}{2},
\end{equation}
and
\begin{equation}\label{Eq: Average lemma result 2 in Energy estimates}
\Phi(x_{0},\,\rho)\le \theta\mu^{2}.
\end{equation}
\end{lemma}
\begin{proof}
We will later choose constants \(\tau\in(0,\,1)\), \(\nu\in(0,\,1/6)\), and \({\hat\rho}\in(0,\,1)\). Let the radius \(\rho\) satisfy \(0<\rho<{\hat\rho}\), and assume that there holds (\ref{Eq: Levelset assumption 2}). By (\ref{Eq: Average value minimizes L2 oscillation}), we have
\[\Phi(x_{0},\,\rho)\le\fint_{B_{\rho}(x_{0})}\mleft\lvert\nabla u_{\varepsilon}-(\nabla u_{\varepsilon})_{x_{0},\,\tau\rho}\mright\rvert^{2}\,{\mathrm d}x=\mathbf{J}_{1}+\mathbf{J}_{2}\]
with
\[\mleft\{\begin{array}{rcl}
\mathbf{J}_{1}&\coloneqq&\displaystyle\frac{1}{\lvert B_{\rho}(x_{0})\rvert}\displaystyle\int_{B_{\tau\rho}(x_{0})}\mleft\lvert\nabla u_{\varepsilon}-(\nabla u_{\varepsilon})_{x_{0},\,\tau\rho}\mright\rvert^{2}\,{\mathrm d}x, \\ \mathbf{J}_{2}&\coloneqq & \displaystyle\frac{1}{\lvert B_{\rho}(x_{0})\rvert}\displaystyle\int_{B_{\rho}(x_{0})\setminus B_{\tau\rho}(x_{0})}\mleft\lvert\nabla u_{\varepsilon}-(\nabla u_{\varepsilon})_{x_{0},\,\tau\rho}\mright\rvert^{2}\,{\mathrm d}x.
\end{array} \mright.\]
For \(\mathbf{J}_{1}\), we use (\ref{Eq: Excess decay estimate far from 0}) to obtain
\[\mathbf{J}_{1}=\tau^{n}\Phi(x_{0},\,\tau\rho)\le C_{\dagger}\mu^{2}\mleft[\frac{\nu^{2/n}}{(1-\tau)^{2}}+\frac{F^{2}}{\nu} \rho^{2\beta/q}\mright]\]
with \(C_{\dagger}=C_{\dagger}(n,\,p,\,\lambda,\,\Lambda,\,K,\,M,\,\delta)\in(0,\,\infty)\).
For \(\mathbf{J}_{2}\), we use (\ref{Eq: mu>delta}) to have \(\lvert\nabla u_{\varepsilon}\rvert\le V_{\varepsilon}\le \delta+\mu\le 2\mu\) a.e. in \(B_{\rho}(x_{0})\), and hence \(\mleft\lvert(\nabla u_{\varepsilon})_{x_{0},\,\tau\rho}\mright\rvert\le 2\mu\). This yields
\[\mathbf{J}_{2}\le 8\mu^{2}\cdot\frac{\lvert B_{\rho}(x_{0})\setminus B_{\tau\rho}(x_{0})\rvert}{\lvert B_{\rho}(x_{0})\rvert}=8\mu^{2}(1-\tau^{n})\le 8n\mu^{2}(1-\tau),\]
where we have used \(1-\tau^{n}=(1+\tau+\cdots+\tau^{n-1})(1-\tau)\le n(1-\tau)\).
As a result, we obtain
\[\Phi(x_{0},\,\rho)\le C_{\dagger}\mu^{2}\mleft[\frac{\nu^{2/n}}{(1-\tau)^{2}}+\frac{F^{2}}{\nu} {\hat\rho}^{2\beta}\mright]+8n(1-\tau)\mu^{2}.\]
We first fix \[\tau\coloneqq 1-\frac{\theta}{24n}\in(0,\,1),\quad\textrm{so that there holds}\quad 8n(1-\tau)=\frac{\theta}{3}.\] 
Next we choose \(\nu\in(0,\,1/6)\) sufficiently small that 
\[\nu\le\min\mleft\{\,\mleft(\frac{\theta(1-\tau)^{2}}{3C_{\dagger}}\mright)^{n/2},\,\frac{3-8\sqrt{\theta}}{23} \,\mright\},\]
so that we have
\[\frac{C_{\dagger}\nu^{2/n}}{(1-\tau)^{2}}\le \frac{\theta}{3},\quad\textrm{and}\quad \sqrt{\theta}\le \frac{3-23\nu}{8}.\]
Corresponding to this \(\nu\), we choose and fix sufficiently small \({\hat\rho}\in(0,\,1)\) satisfying 
\[{\hat\rho}^{2\beta}\le \frac{\nu\theta}{3C_{\dagger}(1+F^{2})},\]
which yields \(C_{\dagger}F^{2}{\hat\rho}^{2\beta}/\nu\le \theta/3\).
Our settings of \(\tau,\,\nu,\,{\hat\rho}\) clearly yield (\ref{Eq: Average lemma result 2 in Energy estimates}).

We are left to show (\ref{Eq: Average lemma result 1 in Energy estimates}). 
By (\ref{Eq: Levelset assumption 2}) and (\ref{Eq: Lower bounds over E}), we compute
\begin{align*}
\fint_{B_{\rho}(x_{0})}\lvert\nabla u_{\varepsilon}\rvert\,{\mathrm d}x&\ge \frac{\lvert S_{\rho,\,\mu,\,\nu}(x_{0})\rvert}{\lvert B_{\rho}(x_{0})\rvert}\cdot \essinf\limits_{S_{\rho,\,\mu,\,\nu}(x_{0})}\,\lvert\nabla u_{\varepsilon}\rvert\\&\ge (1-\nu)\cdot\mleft[(1-\nu)\mu+\frac{7}{8}\delta\mright]>0.
\end{align*}
On the other hand, by applying the triangle inequality, the Cauchy--Schwarz inequality and (\ref{Eq: Average lemma result 2 in Energy estimates}), it is easy to get
\begin{align*}
\mleft\lvert\fint_{B_{\rho}(x_{0})}\lvert\nabla u_{\varepsilon}\rvert\,{\mathrm d}x -\mleft\lvert(\nabla u_{\varepsilon})_{x_{0},\,\rho}\mright\rvert\mright\rvert&=\mleft\lvert \fint_{B_{\rho}(x_{0})}\mleft[\lvert\nabla u_{\varepsilon}\rvert -\mleft\lvert(\nabla u_{\varepsilon})_{x_{0},\,\rho}\mright\rvert\mright]\,{\mathrm d}x\mright\rvert\\&\le \fint_{B_{\rho}(x_{0})}\mleft\lvert\nabla u_{\varepsilon}-(\nabla u_{\varepsilon})_{x_{0},\,\rho}\mright\rvert\,{\mathrm d}x\\&\le \sqrt{\Phi(x_{0},\,\rho)}\le \sqrt{\theta}\mu.
\end{align*}
Again by the triangle inequality, we obtain
\begin{align*}
\mleft\lvert(\nabla u_{\varepsilon})_{x_{0},\,\rho}\mright\rvert&\ge \mleft\lvert\fint_{B_{\rho}(x_{0})}\lvert\nabla u_{\varepsilon}\rvert\,{\mathrm d}x \mright\rvert-\mleft\lvert\fint_{B_{\rho}(x_{0})}\lvert\nabla u_{\varepsilon}\rvert\,{\mathrm d}x -\mleft\lvert(\nabla u_{\varepsilon})_{x_{0},\,\rho}\mright\rvert\mright\rvert\\&\ge \mleft((1-\nu)^{2}-\sqrt{\theta}\mright)\mu+\frac{7}{8}(1-\nu)\delta.
\end{align*}
By our setting of \(\nu\) and (\ref{Eq: mu>delta}), we can check that
\begin{align*}
\mleft((1-\nu)^{2}-\sqrt{\theta}\mright)\mu+\frac{7}{8}(1-\nu)\delta-\mleft(\delta+\frac{\mu}{2}\mright)&=\mleft(\frac{1}{2}-2\nu+\nu^{2}-\sqrt{\theta}\mright)\mu-\mleft(\frac{1}{8}+\frac{7}{8}\nu\mright)\delta\\&\ge \mleft(\frac{3-23\nu}{8}-\sqrt{\theta} \mright)\mu\ge 0,
\end{align*}
which completes the proof of (\ref{Eq: Average lemma result 2 in Energy estimates}).
\end{proof}
We infer to an elementary lemma on Campanato-type growth estimates (Lemma \ref{Lemma: Campanato-type estimate gives Lebesgue points everywhere}).
\begin{lemma}\label{Lemma: Campanato-type estimate gives Lebesgue points everywhere}
Fix an open ball $B_{\rho}(x_{0})\subset{\mathbb R}^{n}$, and let $A\in(0,\,\infty)$, $\gamma\in (0,\,1)$ be given constants. Assume that a function $f\in L^{2}(B_{\rho}(x_{0});\,{\mathbb R}^{m})$ satisfy
\begin{equation}\label{Eq: Assumption of Campanato growth estimate}
\fint_{B_{r}(x_{0})}\mleft\lvert f(x)-(f)_{x_{0},\,r} \mright\rvert^{2}\,{\mathrm{d}}x\le A^{2}\mleft(\frac{r}{\rho}\mright)^{2\gamma}
\end{equation}
for all $r\in(0,\,\rho\rbrack$. Then, the limit
\[F(x_{0})\coloneqq \lim_{r\to 0}\mleft(f\mright)_{x_{0},\,r}\in{\mathbb R}^{m}\]
exists, and there holds
\[\fint_{B_{r}(x_{0})}\mleft\lvert f(x)-F(x_{0}) \mright\rvert^{2}\,{\mathrm{d}}x\le c_{\ddagger}A^{2}\mleft(\frac{r}{\rho}\mright)^{2\gamma}\]
for all $r\in(0,\,\rho\rbrack$. Here the constant $c_{\ddagger}\in(2,\,\infty)$ depends at most on $n$ and $\gamma$.
\end{lemma}
Lemma \ref{Lemma: Campanato-type estimate gives Lebesgue points everywhere} can be proved in a straightforward way. More precisely, we fix $\tau\in(0,\,1)$ arbitrarily, and define $\rho_{k}\coloneqq \tau^{k}\rho\in (0,\,\rho\rbrack$ and $F_{k}\coloneqq (f)_{x_{0},\,\rho_{k}}\in{\mathbb R}^{m}$ for each $k\in{\mathbb Z}_{\ge 0}$. Then, by the Cauchy--Schwarz inequality and (\ref{Eq: Assumption of Campanato growth estimate}), it is easy to check that $\{F_{k}\}_{k=0}^{\infty}\subset{\mathbb R}^{m}$ is a Cauchy sequence. Moreover, there exists a constant $c_{\ddagger}=c_{\ddagger}(n,\,\gamma,\,\tau)\in(1,\,\infty)$ such that the limit $F_{\infty}\coloneqq \lim\limits_{k\to\infty} F_{k}\in{\mathbb R}^{m}$ satisfies
\[\fint_{B_{r}(x_{0})}\mleft\lvert f(x)-F_{\infty} \mright\rvert^{2}\,{\mathrm{d}}x\le c_{\ddagger}A^{2}\mleft(\frac{r}{\rho}\mright)^{2\gamma}\quad \textrm{for all }r\in(0,\,\rho\rbrack,\]
from which we are able to conclude that the limit $F(x_{0})$ exists and coincides with $F_{\infty}$.

Finally we give the proof of Proposition \ref{Prop: Perturbation result}.
\begin{proof}
Let \(\beta=\beta(n,\,q)\) and \(\vartheta=\vartheta(n,\,p,\,q,\,\beta_{0},\,\lambda,\,\Lambda,\,K,\,M,\,\delta)\) be positive constants given by (\ref{Eq: Definition of betas}) and Lemma \ref{Lemma: Higher integrability} respectively. We will later determine a sufficiently small constant \(\tau\in(0,\,1/2)\), and corresponding to this \(\tau\), we will put the desired constants \(\rho_{\star}\in(0,\,1)\) and \(\nu\in(0,\,1/6)\). 
We first assume that 
\begin{equation}\label{Eq: Determination of tau 1}
0<\tau<\tau^{1-\beta}<\frac{1}{16},\quad \textrm{and therefore}\quad
\theta\coloneqq \tau^{\frac{n+2}{\vartheta}}\in\mleft(0,\,\frac{9}{64}\mright)
\end{equation}
holds.
Throughout the proof, we let \(\nu\in(0,\,1/6)\) and \({\hat\rho}\in(0,\,1)\) be sufficiently small constants satisfying Lemma \ref{Lemma: Average lemma from Energy estimates} with \(\theta\in(0,\,9/64)\) given by (\ref{Eq: Determination of tau 1}). Corresponding to this $\hat\rho$, we assume that \(\rho_{\star}\) is so small that there holds
\begin{equation}\label{Eq: Determination of rho-star 1}
0<\rho_{\star}\le {\hat\rho}\mleft(n,\,p,\,q,\,\beta_{0},\,\lambda,\,\Lambda,\,K,\,F,\,M,\,\delta,\,\tau\mright)<1,
\end{equation}
so that Lemma \ref{Lemma: Average lemma from Energy estimates} can be applied for an open ball whose radius is less than $\rho_{\star}$.

Let the ball \(B_{\rho}(x_{0})\) satisfy \(0<\rho<\rho_{\star}\), and let (\ref{Eq: Levelset assumption 2}) hold for the constant \(\nu\in(0,\,1/6)\).
We set a non-negative decreasing sequence \(\{\rho_{k}\}_{k=0}^{\infty}\) by \(\rho_{k}\coloneqq \tau^{k}\rho\) for $k\in{\mathbb Z}_{\ge 0}$. We will choose suitable \(\tau\) and \(\rho_{\ast}\) such that there hold
\begin{equation}\label{Eq: Induction claim 1}
\Phi(x_{0},\,\rho_{k})\le \tau^{2\beta k}\tau^{\frac{n+2}{\vartheta}}\mu^{2},
\end{equation}
and
\begin{equation}\label{Eq: Induction claim 2}
\lvert (\nabla u_{\varepsilon})_{x_{0},\,\rho_{k}} \rvert\ge \delta+\mleft[\frac{1}{2}-\frac{1}{8}\sum_{j=0}^{k-1}2^{-j}\mright]\mu \ge \delta+\frac{\mu}{4}
\end{equation}
for all \(k\in{\mathbb Z}_{\ge 0}\), which will be proved by mathematical induction.
For \(k=0,\,1\), we apply Lemma \ref{Lemma: Average lemma from Energy estimates} to deduce (\ref{Eq: Average lemma result 1 in Energy estimates})--(\ref{Eq: Average lemma result 2 in Energy estimates}) with \(\theta=\tau^{\frac{n+2}{\vartheta}}\). In particular, we have
\begin{equation}\label{Eq: Phi estimate for the first step}
\Phi(x_{0},\,\rho)\le \tau^{\frac{n+2}{\vartheta}}\mu^{2},
\end{equation}
and hence (\ref{Eq: Induction claim 1}) is obvious when $k=0$. From (\ref{Eq: Average lemma result 1 in Energy estimates}), we have already known that (\ref{Eq: Induction claim 2}) holds for \(k=0\). Also, the results (\ref{Eq: Average lemma result 1 in Energy estimates}) and (\ref{Eq: Phi estimate for the first step}) enable us to apply Lemma \ref{Lemma: Average lemma from Perturbaltion} to obtain
\begin{align*}
\Phi(x_{0},\,\rho_{1})&\le C_{\ast}\mleft[\tau^{2}\Phi(x_{0},\,\rho)+\frac{\rho^{2\beta}}{\tau^{n}}\mu^{2}\mright]\\&\le C_{\ast}\tau^{2(1-\beta)}\cdot\tau^{2\beta}\tau^{\frac{n+2}{\vartheta}}\mu^{2}+\frac{C_{\ast}\rho_{\star}^{2\beta}}{\tau^{n}}\mu^{2},
\end{align*}
where \(C_{\ast}=C_{\ast}(n,\,p,\,q,\,\beta_{0},\,\lambda,\,\Lambda,\,K,\,F,\,M,\,\delta)\in(0,\,\infty)\) is a constant given in Lemma \ref{Lemma: Average lemma from Perturbaltion}.
Now we assume that \(\tau\) and \(\rho_{\star}\) satisfy
\begin{equation}\label{Eq: Determination of tau 2}
C_{\ast}\tau^{2(1-\beta)}\le \frac{1}{3},
\end{equation}
and
\begin{equation}\label{Eq: Determination of rho-star 2}
C_{\ast}\rho_{\star}^{2\beta}\le \frac{1}{3}\tau^{n+2\beta+\frac{n+2}{\vartheta}},
\end{equation}
so that (\ref{Eq: Induction claim 1}) is satisfied for \(k=1\).
In particular, by (\ref{Eq: Determination of tau 1}), (\ref{Eq: Phi estimate for the first step}) and the Cauchy--Schwarz inequality, we obtain
\begin{align*}
\mleft\lvert(\nabla u_{\varepsilon})_{x_{0},\,\rho_{1}}-(\nabla u_{\varepsilon})_{x_{0},\,\rho_{0}} \mright\rvert&\le \fint_{B_{\rho_{1}}(x_{0})}\mleft\lvert \nabla u_{\varepsilon}-(\nabla u_{\varepsilon})_{x_{0},\,\rho_{0}}\mright\rvert\,{\mathrm d}x\nonumber\\&\le \mleft(\fint_{B_{\rho_{1}}(x_{0})}\mleft\lvert \nabla u_{\varepsilon}-(\nabla u_{\varepsilon})_{x_{0},\,\rho_{0}}\mright\rvert^{2}\,{\mathrm d}x\mright)^{1/2}=\tau^{-\frac{n}{2}}\Phi(x_{0},\,\rho)^{1/2}\nonumber\\&\le \tau^{\frac{n+2}{2\vartheta}-\frac{n}{2}}\mu\le \tau\mu\le \frac{1}{8}\mu,
\end{align*}
where we have used (\ref{Eq: Determination of tau 1}) and $0<\vartheta<1$.
Combining this result with (\ref{Eq: Average lemma result 1 in Energy estimates}), we use the triangle inequality to get
\[\lvert (\nabla u_{\varepsilon})_{x_{0},\,\rho_{1}}\rvert\ge \lvert (\nabla u_{\varepsilon})_{x_{0},\,\rho_{0}}\rvert-\lvert(\nabla u_{\varepsilon})_{x_{0},\,\rho_{1}}-(\nabla u_{\varepsilon})_{x_{0},\,\rho_{0}} \rvert\ge \mleft(\delta+\frac{\mu}{2}\mright)-\frac{\mu}{8},\]
which yields (\ref{Eq: Induction claim 2}) for \(k=1\).
Next, we assume that the claims (\ref{Eq: Induction claim 1})--(\ref{Eq: Induction claim 2}) are valid for an integer \(k\ge 1\).
Then, the estimate \(\Phi(x_{0},\,\rho_{k})\le \tau^{\frac{n+2}{\vartheta}}\mu^{2}\) clearly holds. Combining this result with (\ref{Eq: Gradient bound in Growth estimate}) and the induction hypothesis (\ref{Eq: Induction claim 2}), we have clarified that the solution \(u_{\varepsilon}\) satisfies all the assumptions of Lemma \ref{Lemma: Average lemma from Perturbaltion} in a smaller ball \(B_{\rho_{k}}(x_{0})\subset B_{\rho_{0}}(x_{0})\). By Lemma \ref{Lemma: Average lemma from Perturbaltion}, (\ref{Eq: Determination of tau 2}), and the induction hypothesis (\ref{Eq: Induction claim 1}), we compute
\begin{align*}
\Phi(x_{0},\,\rho_{k+1})&\le C_{\ast}\mleft[\tau^{2}\Phi(x_{0},\,\rho_{k})+\frac{\rho_{k}^{2\beta}}{\tau^{n}}\mu^{2}\mright]\\&\le C_{\ast}\tau^{2(1-\beta)}\cdot \tau^{2\beta k}\tau^{\frac{n+2}{\vartheta}}\mu^{2}+\frac{C_{\ast}\rho_{\star}^{2\beta}}{\tau^{n}}\cdot\tau^{2\beta k}\mu^{2}\\&\le \tau^{2\beta(k+1)}\cdot \tau^{\frac{n+2}{\vartheta}}\mu^{2},
\end{align*}
which implies that (\ref{Eq: Induction claim 1}) holds for \(k+1\). By the Cauchy--Schwarz inequality and the induction hypothesis (\ref{Eq: Induction claim 1}), we have
\begin{align*}
\mleft\lvert(\nabla u_{\varepsilon})_{x_{0},\,\rho_{k+1}}- (\nabla u_{\varepsilon})_{x_{0},\,\rho_{k}}\mright\rvert&\le \fint_{B_{\rho_{k+1}}(x_{0})}\lvert \nabla u_{\varepsilon}- (\nabla u_{\varepsilon})_{x_{0},\,\rho_{k}}\rvert\,{\mathrm d}x \\&\le \mleft(\fint_{B_{\rho_{k+1}}(x_{0})}\mleft\lvert \nabla u_{\varepsilon}- (\nabla u_{\varepsilon})_{x_{0},\,\rho_{k}}\mright\rvert^{2}\,{\mathrm d}x\mright)^{1/2}=\tau^{-n/2}\Phi(x_{0},\,\rho_{k})^{1/2}\\&\le \tau^{\beta k}\tau^{\frac{n+2}{\vartheta}-\frac{n}{2}}\mu\le 2^{-k}\cdot \frac{1}{8} \mu,
\end{align*}
where we have used (\ref{Eq: Determination of tau 1}) and $0<\vartheta<1$.
Therefore, by the induction hypothesis (\ref{Eq: Induction claim 2}) and the triangle inequality, we finally get
\begin{align*}
\mleft\lvert(\nabla u_{\varepsilon})_{x_{0},\,\rho_{k+1}}\mright\rvert&\ge\mleft\lvert(\nabla u_{\varepsilon})_{x_{0},\,\rho_{k}}\mright\rvert- \mleft\lvert (\nabla u_{\varepsilon})_{x_{0},\,\rho_{k+1}}-(\nabla u_{\varepsilon})_{x_{0},\,\rho_{k}}\mright\rvert\\& \ge \delta+ \mleft[\frac{1}{2}-\frac{1}{8}\sum_{j=0}^{k-1}2^{-j}\mright]\mu-\frac{1}{8}\cdot 2^{-k}\mu,
\end{align*}
which implies that (\ref{Eq: Induction claim 2}) is valid for \(k+1\). This completes the proof of (\ref{Eq: Induction claim 1})--(\ref{Eq: Induction claim 2}).

We would like to complete the proof of Proposition \ref{Prop: Perturbation result}. For each $r\in(0,\,\rho\rbrack$, there uniquely exists $k\in{\mathbb Z}_{\ge 0}$ such that $\rho_{k+1}<r\le \rho_{k}$. Then, by (\ref{Eq: Average value minimizes L2 oscillation}), (\ref{Eq: Induction claim 1}) and $\tau^{k}\le \tau^{-1}(r/\rho)$, we have
\begin{equation}\label{Eq: beta-Campanato growth of gradients}
\Phi(x_{0},\,r)\le \tau^{-n}\Phi(x_{0},\,\rho_{k})\le \tau^{\frac{n+2}{\vartheta}-n-2\beta}\mu^{2}\mleft(\frac{r}{\rho}\mright)^{2\beta}
\end{equation}
for all \(r\in(0,\,\rho\rbrack\). 
Recalling Lemma \ref{Lemma: Lipschitz estimate on relaxed vector fields} and using (\ref{Eq: Average value minimizes L2 oscillation}) again, we compute
\begin{align*}
&\fint_{B_{r}(x_{0})}\mleft\lvert {\mathcal G}_{2\delta,\,\varepsilon}(\nabla u_{\varepsilon})-\mleft({\mathcal G}_{2\delta,\,\varepsilon}(\nabla u_{\varepsilon}) \mright)_{x_{0},\,r} \mright\rvert^{2}\,{\mathrm{d}}x\\&\le \fint_{B_{r}(x_{0})} \mleft\lvert {\mathcal G}_{2\delta,\,\varepsilon}(\nabla u_{\varepsilon})-{\mathcal G}_{2\delta,\,\varepsilon}\mleft((\nabla u_{\varepsilon})_{x_{0},\,r} \mright)\mright\rvert^{2}\,{\mathrm{d}}x\\&\le c_{\dagger}^{2}\Phi(x_{0},\,r)
\end{align*}
with $c_{\dagger}=1+64/\sqrt{255}$. Combining this result with (\ref{Eq: beta-Campanato growth of gradients}), by Lemma \ref{Lemma: Campanato-type estimate gives Lebesgue points everywhere}, we are able to conclude that the limit $\Gamma_{2\delta,\,\varepsilon}(x_{0})\in{\mathbb R}^{n}$ exists, and there holds
\begin{equation}\label{Eq: Continuous Campanato-Growth estimate in Perturbation}
\fint_{B_{r}(x_{0})}\mleft\lvert {\mathcal G}_{2\delta,\,\varepsilon}(\nabla u_{\varepsilon})-\Gamma_{2\delta,\,\varepsilon}(x_{0})\mright\rvert^{2}\,{\mathrm{d}}x\le c_{\ddagger}c_{\dagger}^{2}\tau^{\frac{n+2}{\vartheta}-n-2\beta}\mu^{2}\mleft(\frac{r}{\rho}\mright)^{2\beta}
\end{equation}
for all $r\in(0,\,\rho\rbrack$. Here the constant $c_{\ddagger}\in(2,\,\infty)$ depends at most on $n$ and $q$. Finally, we let
\begin{equation}\label{Eq: Determination of tau 4}
\tau^{2(1-\beta)}\le c_{\ddagger}^{-1}c_{\dagger}^{-2}.
\end{equation}
Then, the desired estimate (\ref{Eq: Campanato-type beta-growth estimate}) clearly follows from (\ref{Eq: Continuous Campanato-Growth estimate in Perturbation})--(\ref{Eq: Determination of tau 4}) and $0<\vartheta<1$. We note that (\ref{Eq: Range of delta-epsilon}) and (\ref{Eq: Gradient bound in Growth estimate}) imply \(\lvert{\mathcal G}_{2\delta,\,\varepsilon}(\nabla u_{\varepsilon})\rvert\le \mu\) a.e. in \(B_{\rho}(x_{0})\), and therefore (\ref{Eq: G-delta-epsilon average-limit bound}) is obvious.

Finally, we mention that we may choose a sufficiently small constant \(\tau=\tau(C_{\ast},\,\beta)\in(0,\,1)\) enjoying (\ref{Eq: Determination of tau 1}), (\ref{Eq: Determination of tau 2}), and (\ref{Eq: Determination of tau 4}). For this fixed \(\tau\), we take sufficient small numbers \(\nu\in(0,\,1/6),\,{\hat\rho}\in(0,\,1)\) as in Lemma \ref{Lemma: Average lemma from Perturbaltion}, depending at most on $n$, $p$, $q$, $\lambda$, $\Lambda$, $K$, $F$, $M$, $\delta$, and $\theta=\tau^{\frac{n+2}{\vartheta}}$. Then, we are able to determine a sufficiently small radius \(\rho_{\star}=\rho_{\star}(C_{\ast},\,\beta,\,{\hat\rho})\in(0,\,1)\) satisfying (\ref{Eq: Determination of rho-star 1}) and (\ref{Eq: Determination of rho-star 2}), and this completes the proof.
\end{proof}

\begin{Remark}\label{Rmk: beta condition}\upshape
An induction claim (\ref{Eq: Induction claim 2}) tells us that the gradient $\nabla u_{\varepsilon}$ itself may not vanish at the point $x_{0}$, and thus $\nabla u_{\varepsilon}$ will satisfy the Campanato-type growth estimate (\ref{Eq: beta-Campanato growth of gradients}) at $x_{0}$. 
To justify these, as well as energy estimates in Section \ref{Subsect: Energy estimates}, we have appealed to freezing coefficient arguments in Section \ref{Subsect: Perturbation outside facets}. There the regularity assumption $E\in C_{\mathrm{loc}}^{2,\,\beta_{0}}({\mathbb R}^{n}\setminus\{ 0\})$ is used.
Also, it should be noted that the condition $\beta<1$ is substantially used in (\ref{Eq: Determination of tau 2}), which makes our recursive proof of (\ref{Eq: Induction claim 1})--(\ref{Eq: Induction claim 2}) successful.
Hence, in the case $q=\infty$, we are not allowed to let $\beta=1$. Here it is worth recalling that it is generally impossible to get $C^{1,\,1}$-regularity for a weak solution to the Poisson equation $-\Delta v=f\in L^{\infty}$.
\end{Remark}

\subsection{Caccioppoli-type energy bounds}\label{Subsect: De Giorgi Oscillation}
In Section \ref{Subsect: De Giorgi Oscillation}, we prove Proposition \ref{Prop: De Giorgi Oscillation lemma} by De Giorgi's truncation. Here we would like to show a Caccioppoli-type estimate (Lemma \ref{Lemma: Caccioppoli-type estimate 1}).

\begin{lemma}\label{Lemma: Caccioppoli-type estimate 1}
Assume that \(u_{\varepsilon}\) is a weak solution to (\ref{Eq: Regularized equation}) in \(\Omega\). Assume that positive numbers $\delta$, $\varepsilon$, $\mu$, $M$, and an open ball $B_{\rho}(x_{0})\Subset\Omega$ satisfy (\ref{Eq: Range of delta-epsilon}) and (\ref{Eq: Gradient bound in De Giorgi estimate}).
Then, the scalar function \(U_{\delta,\,\varepsilon}\in L^{\infty}(B_{\rho}(x_{0}))\cap W^{1,\,2}(B_{\rho}(x_{0}))\) satisfies
\begin{equation}\label{Eq: U is subsolution}
\int_{B_{\rho}(x_{0})}\langle {\mathcal A}_{\delta,\,\varepsilon}(\nabla u_{\varepsilon})\nabla U_{\delta,\,\varepsilon}\mid \nabla \zeta\rangle\,{\mathrm d}x\le C_{0}\mleft[\int_{B_{\rho}(x_{0})}\lvert f_{\varepsilon}\rvert^{2}\zeta\,{\mathrm d}x+\mu\int_{B_{\rho}(x_{0})}\lvert f_{\varepsilon}\rvert\lvert \nabla \zeta\rvert\,{\mathrm d}x\mright]
\end{equation}
for any non-negarive function \(\zeta\in W^{1,\,\infty}(B_{\rho}(x_{0}))\) that is compactly supported in \(B_{\rho}(x_{0})\). Here the constant \(C_{0}\in(0,\,\infty)\) depends at most on $n$, $p$, $\lambda$, $\Lambda$, $M$, and $\delta$. The matrix-valued function \({\mathcal A}_{\delta,\,\varepsilon}(\nabla u_{\varepsilon})\) satisfies
\begin{equation}\label{Eq: Local uniformly elliptic operator on subsol}
\lambda_{\ast}\mathrm{id}\leqslant {\mathcal A}_{\delta,\,\varepsilon}(\nabla u_{\varepsilon})\leqslant \Lambda_{\ast} \mathrm{id} \quad\textrm{in }B_{\rho}(x_{0}),
\end{equation}
where the constants \(0<\lambda_{\ast}\le\Lambda_{\ast}<\infty\) depend at most on $p$, $\lambda$, $\Lambda$, $K$, $M$, $\delta$, but are independent of \(\varepsilon\).
In particular, we have
\begin{equation}\label{Eq: Caccioppoli 1}
\int_{B_{\rho}(x_{0})}\lvert\nabla\mleft[\eta (U_{\delta,\,\varepsilon}-k)_{+}\mright]\rvert^{2}\,{\mathrm d}x\le C\mleft[\int_{B_{\rho}(x_{0})}\lvert\nabla\eta\rvert^{2} (U_{\delta,\,\varepsilon}-k)_{+}^{2}\,{\mathrm d}x+\mu^{2}\int_{A_{k,\,\rho}(x_{0})}\lvert f_{\varepsilon}\rvert^{2}\eta^{2}\,{\mathrm d}x\mright]
\end{equation}
for all \(k\in(0,\,\infty)\) and for any non-negative function \(\eta\in C_{c}^{1}(B_{\rho}(x_{0}))\). Here \(A_{k,\,\rho}(x_{0})\coloneqq \{x\in B_{\rho}(x_{0})\mid U_{\delta,\,\varepsilon}(x)>k\}\), and the constant \(C\in(0,\,\infty)\) depends on \(\lambda_{\ast}\), \(\Lambda_{\ast}\), and \(C_{0}\).
\end{lemma}
Lemma \ref{Lemma: Caccioppoli-type estimate 1} implies that a function \(U_{\delta,\,\varepsilon}\) is a subsolution in the sense that the inequality (\ref{Eq: U is subsolution}) always hold true for any non-negative test function \(\zeta\). The key point is to make sure that the coefficient matrix \({\mathcal A}_{\delta,\,\varepsilon}(\nabla u_{\varepsilon})\) in (\ref{Eq: U is subsolution}) is uniformly elliptic. This is possible because the function \(U_{\delta,\,\varepsilon}\) is supported in \(\{V_{\varepsilon}>\delta\}\), where the ellipticity ratio of \(\nabla^{2} E_{\varepsilon}(\nabla u_{\varepsilon})\) becomes bonded, uniformly for \(\varepsilon\in(0,\,\delta/8)\). It should be noted again that this uniform ellipticity is substantially dependent on \(\delta\in(0,\,1)\).
\begin{proof}
We write \(B=B_{\rho}(x_{0}),\,A_{k}=A_{k,\,\rho}(x_{0})\) for notational simplicity.
The assumption (\ref{Eq: Gradient bound in De Giorgi estimate}) clearly yields \[U_{\delta,\,\varepsilon}=\lvert{\mathcal G}_{\delta,\,\varepsilon}(\nabla u_{\varepsilon}) \rvert^{2}\le \mu^{2}\quad \textrm{a.e. in }B_{\rho}(x_{0}).\]
We define a set $A\coloneqq \{x\in B\mid V_{\varepsilon}(x)>\delta\}\subset B$. 
We claim that the desired matrix \(A_{\delta,\,\varepsilon}(\nabla u_{\varepsilon})\) is given by
\begin{equation}\label{Eq: Determination of A}
{\mathcal A}_{\delta,\,\varepsilon}(\nabla u_{\varepsilon})\coloneqq \chi_{A}V_{\varepsilon}\cdot \nabla^{2} E_{\varepsilon}(\nabla u_{\varepsilon})+\chi_{B\setminus A}\cdot \mathrm{id}.
\end{equation}
Here for a measurable set $E\subset B$, the function $\chi_{E}\colon B\rightarrow \{\,0,\,1\,\}$ denotes the characteristic function of $E$,
\[\textrm{i.e.}, \chi_{E}(x)\coloneqq \mleft\{\begin{array}{cc}
1 & (x\in E),\\ 0 & (x\in B\setminus E).
\end{array} \mright.\]
We prove this claim by applying Lemma \ref{Lemma: Weak formulations of approximated equations} with \(\psi(t)\coloneqq (t-\delta)_{+}\,(0\le t<\infty)\). Under this setting, there holds \(\psi(V_{\varepsilon})=\lvert{\mathcal G}_{\delta,\,\varepsilon}(\nabla u_{\varepsilon})\rvert\).
Therefore, we recall Remark \ref{Rmk: Remark on truncated composition of vector fields} and use (\ref{Eq: Gradient bound in De Giorgi estimate}) to get
\[\mleft\{\begin{array}{ccccc}\displaystyle\frac{\psi(V_{\varepsilon})}{V_{\varepsilon}^{p-2}}&= &\displaystyle\frac{\lvert{\mathcal G}_{\delta,\,\varepsilon}(\nabla u_{\varepsilon})\rvert V_{\varepsilon}}{V_{\varepsilon}^{p-1}} & \le &\displaystyle\frac{\mu M}{\delta^{p-1}} ,\\ \displaystyle\frac{V_{\varepsilon}^{2}}{V_{\varepsilon}^{p-1}}\cdot\psi^{\prime}(V_{\varepsilon}) &\le & \displaystyle\frac{M^{2}}{V_{\varepsilon}^{p-1}}\cdot \chi_{A} &\le &\displaystyle\frac{M^{2}}{\delta^{p-1}},\\ \psi(V_{\varepsilon})V_{\varepsilon} & \le & \mu M, & &\end{array}\mright.\]
a.e. in $B$.
Hence, we obtain
\[J_{4}+J_{5}+J_{6}\le C(n,\,p,\,M,\,\delta)\mleft(\int_{B}\lvert f_{\varepsilon}\rvert^{2}\zeta\,{\mathrm d}x+\mu\int_{B}\lvert f_{\varepsilon}\rvert\lvert \nabla \zeta\rvert\,{\mathrm d}x\mright).\]
Since the function $U_{\delta,\,\varepsilon}$ vanishes on $B\setminus A$ and the identity \(\nabla U_{\delta,\,\varepsilon}=2\psi(V_{\varepsilon})\nabla V_{\varepsilon}\) holds, we obtain
\[J_{1}=\int_{A}\mleft\langle\nabla^{2}E_{\varepsilon}(\nabla u_{\varepsilon})\nabla V_{\varepsilon}\mathrel{}\middle|\mathrel{}\nabla\zeta \mright\rangle \psi(V_{\varepsilon})V_{\varepsilon}\,{\mathrm d}x=\frac{1}{2}\int_{B}\mleft\langle {\mathcal A}_{\delta,\,\varepsilon}(\nabla u_{\varepsilon})\nabla U_{\delta,\,\varepsilon}\mathrel{}\middle|\mathrel{}\nabla\zeta\mright\rangle\,{\mathrm d}x.\]
We also note that the integrals \(J_{2},\,J_{3}\) defined by (\ref{Eq: Definitions of Jk}) are non-negative by (\ref{Eq: positivity of J2-J3}). Hence by discarding these integrals, we are able to conclude (\ref{Eq: U is subsolution}).
Over the set \(A\subset B\), we have
\[V_{\varepsilon}\cdot\nabla^{2}E_{\varepsilon}(\nabla u_{\varepsilon})\geqslant \lambda V_{\varepsilon}^{p-1}\mathrm{id}\geqslant \lambda\delta^{p-1}\mathrm{id},\quad\textrm{and}\]
\[V_{\varepsilon}\cdot\nabla^{2}E_{\varepsilon}(z)\leqslant \mleft(\Lambda V_{\varepsilon}^{p-1}+K\mright)\mathrm{id}\leqslant \mleft(\Lambda M^{p-1}+K\mright)\mathrm{id}.\]
Here we have used (\ref{Eq: Gradient bound in Growth estimate}).
Therefore, by setting
\[\mleft\{\begin{array}{rcl}
\lambda_{\ast}(p,\,\lambda,\,\delta)&\coloneqq&\min\mleft\{\,1,\,\lambda\delta^{p-1}\,\mright\},\\ \Lambda_{\ast}(p,\,\Lambda,\,K,\,M)&\coloneqq& \max\mleft\{\,1,\,\Lambda M^{p-1}+K\,\mright\},
\end{array}   \mright.\]
we have (\ref{Eq: Local uniformly elliptic operator on subsol}).

By approximation, we may test non-negative function \(\zeta\coloneqq \eta^{2}(U_{\delta,\,\varepsilon}-k)_{+}\in W_{0}^{1,\,2}(B)\) into (\ref{Eq: U is subsolution}), since it is compactly supported.
Here we note that \(\zeta\) vanishes in \(B\setminus A_{k}\), and there holds \((U_{\delta,\,\varepsilon}-k)_{+}\le U_{\delta,\,\varepsilon}\le \mu^{2}\) in $B$.
Hence by (\ref{Eq: Local uniformly elliptic operator on subsol}) and Young's inequality, we have
\begin{align*}
&\lambda_{\ast}\int_{B}\lvert\nabla (U_{\delta,\,\varepsilon}-k)_{+}\rvert\eta^{2}\,{\mathrm d}x \\&\le \int_{B}\eta^{2}\mleft\langle {\mathcal A}_{\delta,\,\varepsilon}(\nabla u_{\varepsilon})\nabla (U_{\delta,\,\varepsilon}-k)_{+}\mathrel{}\middle|\mathrel{}\nabla(U_{\delta,\,\varepsilon}-k)_{+}\mright\rangle\,{\mathrm d}x\\&\le -2\int_{B}\eta (U_{\delta,\,\varepsilon}-k)_{+}\mleft\langle {\mathcal A}_{\delta,\,\varepsilon}(\nabla u_{\varepsilon})\nabla (U_{\delta,\,\varepsilon}-k)_{+}\mathrel{}\middle|\mathrel{}\nabla\eta\mright\rangle\,{\mathrm d}x\\&\quad +C_{0}\mleft[\mu^{2}\int_{A_{k}}\lvert f_{\varepsilon}\rvert^{2}\eta^{2}\,{\mathrm d}x+\mu\int_{A_{k}}\lvert f_{\varepsilon}\rvert\mleft(\eta^{2}\lvert \nabla(U_{\delta,\,\varepsilon}-k)_{+}\rvert+2\eta\lvert\nabla\eta\rvert(U_{\delta,\,\varepsilon}-k)_{+} \mright)\,{\mathrm d}x\mright]\\&\le \frac{\lambda_{\ast}}{2}\int_{B}\lvert\nabla (U_{\delta,\,\varepsilon}-k)_{+}\rvert\eta^{2}\,{\mathrm d}x\\&\quad +C(\lambda_{\ast},\,\Lambda_{\ast},\,C_{0})\mleft[\int_{B}\lvert\nabla\eta\rvert^{2} (U_{\delta,\,\varepsilon}-k)_{+}^{2}\,{\mathrm d}x+\mu^{2}\int_{A_{k}}\lvert f_{\varepsilon}\rvert^{2}\eta^{2}\,{\mathrm d}x\mright].
\end{align*}
From this, we are able to deduce (\ref{Eq: Caccioppoli 1}).
\end{proof}
The Caccioppoli-type estimate (\ref{Eq: Caccioppoli 1}) suggests that the function \(U_{\delta,\,\varepsilon}\) should be in a certain De Giorgi classes. This fact enables us to show a variety of levelset lemmata for $U_{\delta,\,\varepsilon}$ (Lemmata \ref{Lemma: Oscillation lemma 1}--\ref{Lemma: Oscillation lemma 2}). The results below are applied to complete the proof of Proposition \ref{Prop: De Giorgi Oscillation lemma}. 
\begin{lemma}\label{Lemma: Oscillation lemma 1}
Let $u_{\varepsilon}$ be a weak solution to (\ref{Eq: Reg Eq}) in $\Omega$.
Assume that positive numbers $\delta$, $\varepsilon$, $\mu$, $F$, $M$, and an open ball $B_{\rho}(x_{0})\Subset\Omega$ satisfy (\ref{Eq: Range of delta-epsilon}), (\ref{Eq: Bound on external force term et al.}) and (\ref{Eq: Gradient bound in De Giorgi estimate}). Then, there exists a number \({\hat \nu}\in(0,\,1)\), depending at most on $n$, $p$, $q$, $\lambda$, $\Lambda$, $K$, $M$, and $\delta$, such that if there holds
\begin{equation}\label{Eq: Measure assumption on Oscillation lemma 1}
\frac{\mleft\lvert\{x\in B_{\rho/2}(x_{0})\mid U_{\delta,\,\varepsilon}(x)>(1-\theta)\mu^{2}\}\mright\rvert}{\lvert B_{\rho/2}(x_{0})\rvert}\le {\hat\nu},
\end{equation}
for some \(\theta\in(0,\,1)\), then we have either
\[\mu^{2}<\frac{\rho^{\beta}}{\theta}\quad \textrm{or}\quad \esssup_{B_{\rho/4}(x_{0})}\,U_{\delta,\,\varepsilon}\le \mleft(1-\frac{\theta}{2}\mright)\mu^{2}.\]
\end{lemma}
\begin{lemma}\label{Lemma: Oscillation lemma 2}
Under the assumptions of Proposition \ref{Prop: De Giorgi Oscillation lemma}, for every integer \(i_{\star}\in{\mathbb N}\), we have either
\[\mu^{2}<\frac{2^{i_{\star}}\rho^{\beta}}{\nu}\quad \textrm{or}\quad \frac{\mleft\lvert\mleft\{x\in B_{\rho/2}(x_{0})\mathrel{}\middle|\mathrel{}U_{\delta,\,\varepsilon}(x)>\mleft(1-2^{-i_{\star}}\nu \mright)\mu^{2}\mright\} \mright\rvert}{\lvert B_{\rho/2}(x_{0})\rvert}<\frac{C_{\dagger}}{\nu\sqrt{i_{\star}}}.\]
Here the constant \(C_{\dagger}\in(0,\,\infty)\) depends at most on $n$, $p$, $q$, $\lambda$, $\Lambda$, $K$, $F$, $M$, and $\delta$.
\end{lemma}
Following standard arguments as in \cite[Chapter 10, Propositions 4.1 \& 5.1]{MR2566733}, it is easy to prove Lemmata \ref{Lemma: Oscillation lemma 1}--\ref{Lemma: Oscillation lemma 2}.
For the reader's convenience, we would like to provide the proofs in the appendix (Section \ref{Sect: Appendix De Giorgi-type levelset argument}).

Combining these lemmata, we give the proof of Proposition \ref{Prop: De Giorgi Oscillation lemma}.
\begin{proof}
We choose the constants \({\hat\nu}\in(0,\,1)\) and \(C_{\dagger}\in(0,\,\infty)\), depending at most on \(n\), \(p\), \(q\), \(\lambda\), \(\Lambda\), \(K\), \(F\), \(M\), and \(\delta\), as in Lemmata \ref{Lemma: Oscillation lemma 1}--\ref{Lemma: Oscillation lemma 2}. Corresponding to these numbers, we fix a sufficiently large number \(i_{\star}=i_{\star}(C_{\dagger},\,\nu,\,{\hat\nu})\in{\mathbb N}\) enough to verify
\begin{equation}\label{Eq: Choice of i-star}
\frac{C_{\dagger}}{\nu\sqrt{i_{\star}}}\le {\hat \nu},\quad\textrm{and}\quad 0<2^{-(i_{\star}+1)}<1-2^{-2\beta}.
\end{equation}
Then, the desired constants are given by \(C_{\star}\coloneqq 2^{i_{\star}}/\nu\in \lbrack 1,\,\infty),\,\kappa\coloneqq \sqrt{1-2^{-(i_{\star}+1)}}\in(2^{-\beta},\,1)\). 
In fact, by Lemma \ref{Lemma: Oscillation lemma 2} and (\ref{Eq: Choice of i-star}), we have either \(\mu^{2}<2^{i_{\star}}\rho^{\beta}/\nu=C_{\star}\rho^{\beta}\) or (\ref{Eq: Measure assumption on Oscillation lemma 1}) with \(\theta\coloneqq 2^{-i_{\star}}\nu\). The first case clearly yields (\ref{Eq: Case 1}). The second case enables us to apply Lemma \ref{Lemma: Oscillation lemma 1}, and hence we have either \(\mu^{2}<\rho^{\beta}/\theta=C_{\star}\rho^{\beta}\) or
\[\esssup_{B_{\rho/4}(x_{0})}\,\mleft\lvert{\mathcal G}_{\delta,\,\varepsilon}(\nabla u_{\varepsilon})\mright\rvert^{2}=\esssup_{B_{\rho/4}(x_{0})}\,U_{\delta,\,\varepsilon}\le\mleft(1-\frac{\theta}{2}\mright)\mu^{2}=(\kappa\mu)^{2}.\]
In all the possible cases, we obtain either (\ref{Eq: Case 1}) or (\ref{Eq: Case 2}).
\end{proof}
\section{Appendix: Proofs for some basic estimates}\label{Sect: Appendix De Giorgi-type levelset argument}
In Section \ref{Sect: Appendix De Giorgi-type levelset argument}, we would like to provide the proofs Lemmata \ref{Lemma: Oscillation lemma 1}--\ref{Lemma: Oscillation lemma 2} in Section \ref{Subsect: De Giorgi Oscillation} for the reader's convenience. 
These results can be deduced by the fact that the function \(U_{\delta,\,\varepsilon}\) satisfies Caccioppoli-type energy bounds since it is a weak subsolution to a uniformly elliptic problem (Lemma \ref{Lemma: Caccioppoli-type estimate 1}). We mention that the strategy herein are based on De Giorgi's levelset argument given in \cite[Chapter 10, \S 4--5]{MR2566733}. See also \cite[\S 7]{BDGPdN} as a related item.
 
The proof of Lemma \ref{Lemma: Oscillation lemma 1} is substantially found in \cite[Chapter 10, Proposition 4.1]{MR2566733}.
\begin{proof}
For every \(i\in{\mathbb Z}_{\ge 0}\), we set \[\rho_{i}\coloneqq \frac{\rho}{4}\mleft(1+2^{-i}\mright),\quad k_{i}\coloneqq \mleft(1-\frac{1}{2}\mleft(1+2^{-i}\mright)\theta \mright)\mu^{2},\]
and define measurable sets \(B_{i}\coloneqq B_{\rho_{i}}(x_{0}),\,A_{i}\coloneqq \{ x\in B_{i}\mid U_{\delta,\,\varepsilon}(x)>k_{i}\}\). 

To show Lemma \ref{Lemma: Oscillation lemma 1}, we will find constants \(C_{\ast}=C_{\ast}(n,\,p,\,q,\,\lambda,\,\Lambda,\,K,\,M,\,\delta,\,\theta)\in\lbrack 1,\,\infty)\) and \(\varsigma=\varsigma(n,\,q)\in(0,\,2\beta/n\rbrack\) such that there holds
\begin{equation}\label{Eq: Claim on sequences}
R_{i+1}\le \frac{C_{\ast}}{2}\cdot 4^{2i}\mleft[1+\frac{\rho^{2\beta}}{\theta^{2}\mu^{4}} \mright]R_{i}^{1+\varsigma},\quad \textrm{where }R_{i}\coloneqq \frac{\lvert A_{i}\rvert}{\lvert B_{i}\rvert},
\end{equation}
for every \(i\in{\mathbb Z}_{\ge 0}\). 
For fixed \(i\in{\mathbb Z}_{\ge 0}\), we choose a cutoff function \(\eta\in C_{c}^{1}(B_{i})\) satisfying
\begin{equation}\label{Eq: Choice of cutoff function truncation case}
\eta\equiv 1\textrm{ on }B_{i+1},\quad \lvert \nabla \eta\rvert\le \frac{2}{\rho_{i}-\rho_{i+1}}=\frac{2^{i+4}}{\rho}.
\end{equation}
We also note that
\begin{equation}\label{Eq: Inequalities on truncated functions}
\mleft\{\begin{array}{cccccr}
(U_{\delta,\,\varepsilon}-k_{i})_{+}&\ge &k_{i+1}-k_{i}&=&2^{-(i+2)}\theta\mu^{2} & \textrm{a.e. in }A_{i+1},\\
(U_{\delta,\,\varepsilon}-k_{i})_{+}&\le &\mu^{2}-k_{i}&\le &\theta\mu^{2}& \textrm{a.e. in }B_{\rho}(x_{0}),
\end{array} \mright.
\end{equation}
and
\begin{equation}\label{Eq: Ratio bound on measures of ball}
\frac{\lvert B_{i}\rvert}{\lvert B_{i+1}\rvert}=\mleft(\frac{\rho_{i}}{\rho_{i+1}}\mright)^{n}=\mleft(\frac{1+2^{-i}}{1+2^{-(i+1)}}\mright)^{n}=\mleft(1+\frac{1}{1+2^{i+1}}\mright)^{n}\le 2^{n}
\end{equation}
for every \(i\in{\mathbb Z}_{\ge 0}\). 
We first consider the case \(n\ge 3\). In this case, we may apply the Sobolev embedding \(W_{0}^{1,\,2}(B_{i})\hookrightarrow L^{{\frac{2n}{n+2}}}(B_{i})\) to the function \(\eta(U_{\delta,\,\varepsilon}-k_{i})_{+}\in W_{0}^{1,\,2}(B_{i})\).
By (\ref{Eq: Choice of cutoff function truncation case})--(\ref{Eq: Inequalities on truncated functions}), H\"{o}lder's inequality and the Caccioppoli-type estimate (\ref{Eq: Caccioppoli 1}), we have
\begin{align*}
\lvert A_{i+1}\rvert \cdot 4^{-(i+2)}\theta^{2}\mu^{4}&\le \int_{A_{i+1}}\mleft[\eta(U_{\delta,\,\varepsilon}-k_{i})_{+}\mright]^{2}\,{\mathrm d}x\\&\le C(n)\lvert A_{i}\rvert^{\frac{2}{n}}\int_{B_{i}}\mleft\lvert\nabla \mleft[\eta(U_{\delta,\,\varepsilon}-k_{i})_{+}\mright]\mright\rvert^{2}\,{\mathrm d}x\\&\le C\lvert A_{i}\rvert^{\frac{2}{n}}\mleft[ \int_{B_{i}}\lvert\nabla\eta\rvert^{2} (U_{\delta,\,\varepsilon}-k)_{+}^{2}\,{\mathrm d}x+\mu^{2}\int_{A_{i}}\lvert f_{\varepsilon}\rvert^{2}\eta^{2}\,{\mathrm d}x\mright]\\&\le C\lvert A_{i}\rvert^{\frac{2}{n}}\mleft[\frac{4^{i+4}}{\rho^{2}}\cdot \mleft(\theta\mu^{2}\mright)^{2}\cdot \lvert A_{i}\rvert +F^{2}\lvert A_{i}\rvert^{1-\frac{2}{q}}\mright]\\& \le C\cdot 4^{i}\theta^{2}\mu^{4}\lvert A_{i}\rvert^{1+\frac{2\beta}{n}}\mleft[\frac{\lvert A_{i}\rvert^{\frac{2}{n}(1-\beta)}}{\rho^{2}}+\frac{1}{\theta^{2}\mu^{4}} \mright],
\end{align*}
which yields
\[\lvert A_{i+1}\rvert\le C\cdot 4^{2i}\lvert A_{i}\rvert^{1+\frac{2\beta}{n}}\mleft[\frac{\lvert A_{i}\rvert^{\frac{2}{n}(1-\beta)}}{\rho^{2}}+\frac{1}{\theta^{2}\mu^{4}} \mright]\]
for some constant \(C=C(n,\,p,\,q,\,\lambda,\,\Lambda,\,K,\,F,\,M,\,\delta)\in(0,\,\infty)\).
From this estimate and (\ref{Eq: Ratio bound on measures of ball}), the desired inequality (\ref{Eq: Claim on sequences}) follows with \(\varsigma=2\beta/n\). In fact, we obtain
\begin{align*}
R_{i+1}&\le 2^{n}\cdot\frac{\lvert A_{i+1}\rvert}{\lvert B_{i}\rvert^{1+\varsigma}}\cdot\lvert B_{i}\rvert^{\varsigma} \\&\le C\cdot 4^{2i}\cdot\mleft[\frac{\lvert A_{i}\rvert^{\frac{2}{n}(1-\beta)}}{\rho^{2}}+\frac{1}{\theta^{2}\mu^{4}}\mright]\lvert B\rvert^{\varsigma}\cdot\mleft(\frac{\lvert A_{i}\rvert}{\lvert B_{i}\rvert}\mright)^{1+\varsigma}\\&\le \frac{C_{\ast}}{2}\cdot 4^{2i}\cdot \mleft[1+\frac{\rho^{2\beta}}{\theta^{2}\mu^{4}}\mright]R_{i}^{1+\varsigma}\quad \textrm{for every }i\in{\mathbb Z}_{\ge 0},
\end{align*}
where we have used \(\lvert A_{i}\rvert\le \lvert B_{i}\rvert=C(n)\rho_{i}^{n}\le C(n)\rho^{n}\).
In the remaining case \(n=2\), we choose a sufficiently large constant \(\sigma\in(1,\,\infty)\) satisfying
\begin{equation}\label{Eq: Determination of sigma}
\varsigma\coloneqq \beta-\frac{1}{\sigma}=1-\frac{2}{q}-\frac{1}{\sigma}>0,
\end{equation}
and apply the alternative Sobolev embedding \(W_{0}^{1,\,2}(B_{i})\hookrightarrow L^{2\sigma}(B_{i})\) to the function \(\eta(U_{\delta,\,\varepsilon}-k_{i})_{+}\in W_{0}^{1,\,2}(B_{i})\). Then, we similarly obtain
\begin{align*}
\lvert A_{i+1}\rvert \cdot 4^{-(i+2)}\theta^{2}\mu^{4}&\le \int_{A_{i+1}}\mleft[\eta(U_{\delta,\,\varepsilon}-k_{i})_{+}\mright]^{2}\,{\mathrm d}x\\&\le\lvert A_{i}\rvert^{1-\frac{1}{\sigma}}\mleft[\int_{B_{i}}\mleft[\eta(U_{\delta,\,\varepsilon}-k_{i})_{+}\mright]^{2\sigma} \,{\mathrm d}x \mright]^{1/\sigma} \\&\le C(n,\,\sigma) \lvert A_{i}\rvert^{1-\frac{1}{\sigma}}\cdot \rho_{i}^{\frac{2}{\sigma}}\int_{B_{i}}\mleft\lvert\nabla \mleft[\eta(U_{\delta,\,\varepsilon}-k_{i})_{+}\mright]\mright\rvert^{2}\,{\mathrm d}x\\&\le C \lvert A_{i}\rvert^{1-\frac{1}{\sigma}}\cdot \rho_{i}^{\frac{2}{\sigma}}\mleft[ \int_{B_{i}}\lvert\nabla\eta\rvert^{2} (U_{\delta,\,\varepsilon}-k)_{+}^{2}\,{\mathrm d}x+\mu^{2}\int_{A_{i}}\lvert f_{\varepsilon}\rvert^{2}\eta^{2}\,{\mathrm d}x\mright]\\&\le C\lvert A_{i}\rvert^{1-\frac{1}{\sigma}}\cdot \rho_{i}^{\frac{2}{\sigma}}\mleft[\frac{4^{i+4}}{\rho^{2}}\cdot \mleft(\theta\mu^{2}\mright)^{2}\cdot \lvert A_{i}\rvert +F^{2}\lvert A_{i}\rvert^{1-\frac{2}{q}}\mright]\\ & \le C\cdot 4^{i}\theta^{2}\mu^{4}\lvert A_{i}\rvert^{2-\frac{2}{q}-\frac{1}{\sigma}}\cdot \rho_{i}^{\frac{2}{\sigma}}\mleft[\frac{\lvert A_{i}\rvert^{\frac{2}{q}}}{\rho^{2}}+\frac{1}{\theta^{2}\mu^{4}} \mright].
\end{align*}
Recall (\ref{Eq: Determination of sigma}), and we get
\[\lvert A_{i+1}\rvert\le C\cdot 4^{2i}\lvert A_{i}\rvert^{1+\varsigma}\cdot \rho_{i}^{\frac{2}{\sigma}}\mleft[\frac{\lvert A_{i}\rvert^{1-\beta}}{\rho^{2}}+\frac{1}{\theta^{2}\mu^{4}} \mright]\]
for some \(C=C(n,\,p,\,q,\,\lambda,\,\Lambda,\,K,\,M,\,\delta)\in(0,\,\infty)\).
The desired estimate (\ref{Eq: Claim on sequences}) follows from this estimate with \(\varsigma\) given by (\ref{Eq: Determination of sigma}).
In fact, by (\ref{Eq: Ratio bound on measures of ball})--(\ref{Eq: Determination of sigma}) we obtain
\begin{align*}
R_{i+1}&\le 2^{2}\cdot\frac{\lvert A_{i+1}\rvert}{\lvert B_{i}\rvert^{1+\varsigma}}\cdot\lvert B_{i}\rvert^\varsigma \\&\le C\cdot 4^{2i}\cdot\mleft[\frac{\lvert A_{i}\rvert^{1-\beta}}{\rho^{2}}+\frac{1}{\theta^{2}\mu^{4}}\mright]\rho_{i}^{\frac{2}{\sigma}}\lvert B\rvert^{\varsigma}\cdot\mleft(\frac{\lvert A_{i}\rvert}{\lvert B_{i}\rvert}\mright)^{1+\varsigma}\\&\le \frac{C_{\ast}}{2}\cdot 4^{2i}\cdot \mleft[1+\frac{\rho^{2\beta}}{\theta^{2}\mu^{4}}\mright]R_{i}^{1+\varsigma}\quad \textrm{for every }i\in{\mathbb Z}_{\ge 0},
\end{align*}
where we have used \(\lvert A_{i}\rvert\le \lvert B_{i}\rvert=C\rho_{i}^{n}\le C\rho^{2}\) and \(\rho_{i}^{\frac{2}{\sigma}}\lvert B\rvert^{\varsigma}\le C\rho^{2\beta}\).

Now we determine the constant \({\hat\nu}={\hat\nu}(n,\,p,\,q,\,\lambda,\,\Lambda,\,K,\,F,\,M,\,\delta)\in(0,\,1)\) by
\[{\hat\nu}\coloneqq C_{\ast}^{-\frac{1}{\varsigma}}16^{-\frac{1}{\varsigma^{2}}}.\]
Let \(\mu\) satisfy \(\mu\ge \rho^{\beta}/\theta\). Then, by (\ref{Eq: Measure assumption on Oscillation lemma 1}) and (\ref{Eq: Claim on sequences}), we have already known that
\[R_{i+1}\le C_{\ast}\cdot 16^{i}\cdot R_{i}^{1+\varsigma}\quad \textrm{for every }k\in{\mathbb Z}_{\ge 0},\]
and \(R_{0}\le{\hat\nu}\). By applying \cite[Chapter 2, Lemma 4.7]{MR0244627} to the sequence $\{R_{i}\}_{i=0}^{\infty}$, we are able to conclude \(R_{i}\to 0\) as $i\to \infty$. From this, it follows that
\[0\le \mleft\lvert \mleft\{x\in B_{\rho/4}(x_{0})\mathrel{}\middle|\mathrel{} U_{\delta,\,\varepsilon}(x)>(1-\theta/2)\mu^{2}\mright\}\mright\rvert \le \liminf_{i\to \infty}\,\lvert A_{i}\rvert=0.\]
Hence we have
\[\esssup_{B_{\rho/4}(x_{0})}\,U_{\delta,\,\varepsilon}\le \mleft(1-\frac{\theta}{2}\mright)\mu^{2},\]
which completes the proof.
\end{proof}

The proof of Lemma \ref{Lemma: Oscillation lemma 2} is substantially based on \cite[Chapter 10, Proposition 5.1]{MR2566733}.
Here we use a well-known inequality;
\begin{equation}\label{Eq: Isoperimetric inequality}
(l-k)\cdot \mleft\lvert B_{\rho}\cap \{v<k\}\mright\rvert \le \frac{C(n)\rho^{n+1}}{\mleft\lvert B_{\rho}\cap \{v>l\}\mright\rvert}\int_{B_{\rho}\cap\{k<v<l\}}\lvert\nabla v\rvert\,{\mathrm d}x
\end{equation}
for all \(v\in W^{1,\,1}(B_{\rho})\) and \(-\infty<k<l<\infty\) (see \cite[Chapter 10, \S 5.1]{MR2566733} for the proof).
\begin{proof}
For each \(i\in{\mathbb Z}_{\ge 0}\), we set a measurable set
\[A_{i}\coloneqq \{x\in B_{\rho/2}(x_{0})\mid U_{\delta,\,\varepsilon}(x)>k_{i}\}\quad \textrm{with }k_{i}\coloneqq \mleft(1-2^{-i}\nu\mright)\mu^{2}.\]
By (\ref{Eq: Levelset assumption}), it is easy to check that
\begin{equation}\label{Eq: Measure assumption rewrite}
\lvert B_{\rho/2}(x_{0})\setminus A_{i}\rvert \ge \lvert B_{\rho/2}(x_{0})\setminus A_{0}\rvert \ge \nu\lvert B_{\rho/2}(x_{0})\rvert=C(n)\nu\rho^{n}
\end{equation}
for every \(i\in {\mathbb Z}_{\ge 0}\).
We apply (\ref{Eq: Isoperimetric inequality}) to the function \(U_{\delta,\,\varepsilon}\in W^{1,\,1}(B_{\rho/2}(x_{0}))\) with \((k,\,l)=(k_{i},\,k_{i+1})\). By the Cauchy--Schwarz inequality and (\ref{Eq: Measure assumption rewrite}), we have
\begin{align*}
\frac{\nu\mu^{2}}{2^{i+1}}\cdot \lvert A_{i+1}\rvert&\le \frac{C(n)\rho^{n+1}}{\lvert B_{\rho/2}(x_{0})\setminus A_{i}\rvert}\int_{A_{i}\setminus A_{i+1}}\lvert\nabla U_{\delta,\,\varepsilon}\rvert\,{\mathrm d}x\\&\le \frac{C(n)\rho}{\nu}\mleft(\int_{B_{\rho/2}(x_{0})}\lvert\nabla (U_{\delta,\,\varepsilon}-k_{i})_{+}\rvert^{2}\,{\mathrm d}x \mright)^{1/2}\mleft(\lvert A_{i}\rvert-\lvert A_{i+1}\rvert \mright)^{1/2}.
\end{align*}
Therefore we obtain
\begin{equation}\label{Eq: Levelset estimate from isoperimetric}
\lvert A_{i+1}\rvert^{2}\le \frac{C(n)4^{i}\rho^{2}}{\nu^{4}\mu^{4}}\mleft(\lvert A_{i}\rvert-\lvert A_{i+1}\rvert \mright)\int_{B_{\rho/2}(x_{0})}\lvert\nabla (U_{\delta,\,\varepsilon}-k_{i})_{+}\rvert^{2}\,{\mathrm d}x
\end{equation}
for every \(i\in{\mathbb Z}_{\ge 0}\).
We choose a cutoff function \(\eta\in C_{c}^{1}(B_{\rho}(x_{0}))\) such that
\[\eta\equiv 1\textrm{ on }B_{\rho/2}(x_{0}),\quad \lvert \nabla \eta\rvert\le \frac{4}{\rho},\]
and apply Lemma \ref{Lemma: Caccioppoli-type estimate 1}. Then, for every \(i\in{\mathbb Z}_{\ge 0}\), we have
\begin{align}\label{Eq: Estimate on truncated energy}
&\int_{B_{\rho/2}(x_{0})}\lvert\nabla (U_{\delta,\,\varepsilon}-k_{i})_{+}\rvert^{2}\,{\mathrm d}x\nonumber\\&\le \int_{B_{\rho}(x_{0})}\lvert\nabla\mleft[\eta (U_{\delta,\,\varepsilon}-k_{i})_{+}\mright]\rvert^{2}\,{\mathrm d}x\nonumber\\&\le C\mleft[\int_{B_{\rho}(x_{0})}\lvert\nabla\eta\rvert^{2} (U_{\delta,\,\varepsilon}-k_{i})_{+}^{2}\,{\mathrm d}x+\mu^{2}\int_{A_{i}}\lvert f_{\varepsilon}\rvert^{2}\eta^{2}\,{\mathrm d}x\mright]\nonumber\\&\le \frac{C}{\rho^{2}}\mleft[\frac{\nu^{2}\mu^{4}}{4^{i}}\cdot \lvert B_{\rho}(x_{0})\rvert +F^{2}\rho^{n+2\beta}\mright]\le \frac{C\nu^{2}\mu^{4}}{4^{i}\rho^{2}}\mleft[\mleft(\frac{2^{i}\rho^{\beta}}{\nu\mu^{2}}\mright)^{2}+1\mright]\cdot \lvert B_{\rho/2}(x_{0})\rvert
\end{align}
for some constant \(C=C(n,\,p,\,q,\,\lambda,\,\Lambda,\,K,\,F,\,M,\,\delta)\in(0,\,\infty)\).
Here we have used H\"{o}lder's inequality and
\[(U_{\delta,\,\varepsilon}-k_{i})_{+}\le \mu^{2}-k_{i}=2^{-i}\nu\mu^{2}\quad \textrm{a.e. in }B_{\rho}(x_{0}).\]

Take and fix \(i_{\star}\in {\mathbb N}\) arbitrarily, and assume that \(\mu\) satisfies \(\mu^{2}\ge 2^{i_{\star}}\rho^{\beta}/\nu\). Then, for every $i\in\{\,0,\,\dots\,,\,i_{\star}-1\,\}$, there clearly holds
\[\frac{2^{i}\rho^{\beta}}{\nu\mu^{2}}\le \frac{2^{i_{\star}}\rho^{\beta}}{\nu\mu^{2}}\le 1.\]
In particular, by (\ref{Eq: Levelset estimate from isoperimetric})--(\ref{Eq: Estimate on truncated energy}), we are able to find a constant \(C_{\dagger}\in(1,\,\infty)\), depending at most on \(n\), \(p\), \(q\), \(\lambda\), \(\Lambda\), \(K\), \(F\), \(M\), and \(\delta\), such that
\[\lvert A_{i+1}\rvert^{2}\le \frac{C_{\dagger}^{2}}{\nu^{2}}\lvert B_{\rho/2}(x_{0})\rvert\cdot \mleft( \lvert A_{i}\rvert-\lvert A_{i+1}\rvert \mright)\]
holds for each $i\in\{\,0,\,\dots\,,\,i_{\star}-1\,\}$. This estimate yields
\begin{align*}
i_{\star}\lvert A_{i_{\star}}\rvert^{2}&\le \sum_{i=0}^{i_{\star}-1}\lvert A_{i+1}\rvert^{2}\\&\le \frac{C_{\dagger}^{2}}{\nu^{2}}\lvert B_{\rho/2}(x_{0})\rvert \sum_{i=0}^{i_{\star}-1}\mleft( \lvert A_{i}\rvert-\lvert A_{i+1}\rvert \mright)\\&\le \frac{C_{\dagger}^{2}}{\nu^{2}}\lvert B_{\rho/2}(x_{0})\rvert\cdot \lvert A_{0}\rvert\le \frac{C_{\dagger}^{2}}{\nu^{2}}\lvert B_{\rho/2}(x_{0})\rvert^{2}.
\end{align*}
Hence we obtain
\[\lvert A_{i_{\star}}\rvert \le \frac{C_{\dagger}}{\nu\sqrt{i_{\star}}} \lvert B_{\rho/2}(x_{0})\rvert,\]
which completes the proof.
\end{proof}
\bibliographystyle{plain}

\end{document}